\documentclass{amsart}
\usepackage[backend=bibtex]{biblatex}
\usepackage{amssymb, mathpazo, amsmath,graphicx,color,enumitem,verbatim,fullpage, stmaryrd}
\usepackage{empheq,hyperref}
\renewbibmacro{in:}{}
\nocite{*}
\renewcommand{\a}{\mathfrak a}
\newcommand{\p}{\mathfrak p}
\newcommand{\q}{\mathfrak q}

\newcommand{\Z}{\mathbb Z}
\newcommand{\N}{\mathbb N}
\newcommand{\Q}{\mathbb Q}
\newcommand{\R}{\mathbb R}
\newcommand{\C}{\mathbb C}
\newcommand{\F}{\mathbb F}
\newcommand{\qbar}{\overline{\mathbb Q}}
\renewcommand{\P}{\mathbb P}
\renewcommand{\phi}{\varphi}

\newcommand{\PrePer}{\operatorname{PrePer}}
\newcommand{\ord}{\operatorname{ord}}
\newcommand{\Jac}{\operatorname{Jac}}
\newcommand{\Spec}{\operatorname{Spec}}
\newcommand{\Gal}{\operatorname{Gal}}

\newcommand{\calC}{\mathcal C}
\newcommand{\calI}{\mathcal I}
\newcommand{\calL}{\mathcal L}
\newcommand{\calP}{\mathcal P}
\newcommand{\calX}{\mathcal X}
\newcommand{\K}{\mathcal K}
\renewcommand{\O}{\mathcal O}
\renewcommand{\to}{\rightarrow}
\newcommand{\into}{\hookrightarrow}
\renewcommand{\ss}{\mathcal S}

\newtheorem{thm}{Theorem}[section]
\newtheorem{lem}[thm]{Lemma}
\newtheorem{prop}[thm]{Proposition}
\newtheorem{cor}[thm]{Corollary}

\newtheorem{speculation}[thm]{Speculation}

\newtheorem*{thm*}{Theorem}
\newtheorem*{ubc}{Uniform Boundedness Conjecture}

\theoremstyle{definition}
\newtheorem{alg}{Algorithm}
\theoremstyle{remark}
\newtheorem*{rem}{Remark}

\numberwithin{equation}{section}

\title{Preperiodic points for quadratic polynomials over quadratic fields}

\author{John R. Doyle}
\address{Department of Mathematics \\
University of Georgia \\
Athens, GA  30602} 
\email{jdoyle@math.uga.edu}

\author{Xander Faber}
\address{Department of Mathematics \\
University of Hawaii \\
Honolulu, HI  96822} 
\email{xander@math.hawaii.edu}

\author{David Krumm}
\address{Department of Mathematics \\
Claremont McKenna College \\
Claremont, CA 91711}
\email{dkrumm@cmc.edu}

\begin{document}
\begin{abstract} 
	To each quadratic number field $K$ and each quadratic polynomial $f$ with $K$-coefficients, one can associate a finite directed graph $G(f,K)$ whose vertices are the $K$-rational preperiodic points for $f$, and whose edges reflect the action of $f$ on these points. This paper has two main goals. (1) For an abstract directed graph $G$, classify the pairs $(K,f)$ such that the isomorphism class of $G$ is realized by $G(f,K)$. We succeed completely for many graphs $G$ by applying a variety of dynamical and Diophantine techniques. (2) Give a complete description of the set of isomorphism classes of graphs that can be realized by some $G(f,K)$. A conjecture of Morton and Silverman implies that this set is finite. Based on our theoretical considerations and a wealth of empirical evidence derived from an algorithm that is developed in this paper, we speculate on a complete list of isomorphism classes of graphs that arise from quadratic polynomials over quadratic fields. 
\end{abstract}
\maketitle
 
\section{Introduction}\label{paper_intro}

\subsection{Background}\label{intro_background}
Let $K$ be a number field and let $f(z)=A(z)/B(z)$ be a rational function defined over $K$, where $A(z)$ and $B(z)$ are coprime polynomials with coefficients in $K$.  The function $f(z)$ naturally induces a map $f:\P^1(K)\to\P^1(K)$; a fundamental problem in dynamics is that of describing the behavior of points $P\in\P^1(K)$ under repeated iteration of the map $f$. Thus, we consider the sequence 
\[P,f(P),f(f(P)),f(f(f(P))),\ldots.\]
For convenience we denote by $f^m$ the $m$-fold composition of $f$: $f^0$ is the identity map, and $f^m = f \circ f^{m-1}$ for all $m \ge 1$. We say that a point $P \in \P^1(K)$ is {\it preperiodic} for $f$ if the orbit of $P$ under $f$, i.e., the set $\{f^m(P): m\geq 0\}$, is finite. Furthermore, we say that $P$ is {\it periodic} for $f$ if it satisfies the stronger condition that $f^m(P)=P$ for some $m>0$; in this case, the least positive integer $m$ with this property is called the {\it period} of $P$. The set of all points $P\in\P^1(K)$ that are preperiodic for $f$ is denoted by $\PrePer(f,K)$. Finally, the {\it degree} of $f(z)$ is defined to be the number $d := \max\{\deg A,\deg B\}$.

Using the theory of height functions, Northcott \cite{northcott} proved that the set $\PrePer(f,K)$ is finite as long as the degree of $f$ is greater than 1. This set can be given the structure of a directed graph by letting the elements $P\in \PrePer(f,K)$ be the vertices of the graph, and drawing directed edges $P\to f(P)$ for every such point $P$. Thus, we obtain a finite directed graph representing the $K$-rational preperiodic points for $f$. It is then natural to ask how large the set $\PrePer(f,K)$ can be, and what structure the associated graph can have. Drawing an analogy between preperiodic points of maps and torsion points on abelian varieties, Morton and Silverman \cite[100]{morton-silverman} proposed the following conjecture regarding the size of the set $\PrePer(f,K)$:

\begin{ubc}[Morton-Silverman] Fix integers $n\geq 1$ and  $d\geq 2$. There exists a constant $M(n,d)$ such that for every number field $K$ of degree $n$, and every rational function $f(z)\in K(z)$ of degree $d$, \[\#\PrePer(f,K)\leq M(n,d).\]
\end{ubc}

Very little is currently known about this conjecture; indeed, it has not been proved that such a constant $M(n,d)$ exists, even in the simplified setting where $K=\Q$ and $f(z)\in K[z]$ is a quadratic polynomial. However, Poonen \cite[Cor. 1]{poonen_prep} proposed an upper bound of 9 in this case, and moreover gave a conjecturally complete list of all possible graph structures arising in this context --- see \cite[17]{poonen_prep}.

\begin{thm}[Poonen]\label{poonen_thm} Assume that there is no quadratic polynomial over $\Q$ having a rational periodic point of period greater than 3. Then, for every quadratic polynomial $f$ with rational coefficients,
\[\#\PrePer(f,\Q)\leq 9.\] Moreover, there are exactly 12 graphs that arise from $\PrePer(f,\Q)$ as $f$ varies over all quadratic polynomials with rational coefficients.
\end{thm}

Regarding the assumption made in Poonen's result, it is known that there is no quadratic polynomial over $\Q$ having a rational periodic point of period $m=4$ or 5 (see \cite[Thm. 1]{flynn-poonen-schaefer} and \cite[Thm. 4]{morton_4cycles}), and assuming standard conjectures on $L$-series of curves, the same holds for $m=6$ (see \cite[Thm. 7]{stoll_6cycles}). In \cite{flynn-poonen-schaefer} Flynn, Poonen, and Schaefer conjecture that no quadratic polynomial over $\Q$ has a rational point of period greater than 3, a hypothesis which Hutz and Ingram have verified extensively by explicit computation (see \cite[Prop. 1]{hutz-ingram}). However, a proof of this conjecture seems distant at present.

One direction in which to continue the kind of work carried out by Poonen in studying the Uniform Boundedness Conjecture (UBC) is to consider preperiodic points for higher degree polynomials over $\Q$, such as was done by Benedetto et al. \cite{benedetto_cubics} in the case of cubics. In this paper we take a different approach and consider maps defined over number fields of degree $n>1$. The case of degree $n=1$ is very special because there is only one number field with this degree; hence, in this case the UBC is a statement only about uniformity as one varies the rational function $f(z)$ with $\Q$-coefficients. More striking is that the conjecture predicts upper bounds even when all number fields of a fixed degree $n>1$ are considered. Our goal in this article is to carry out an initial study of preperiodic points for maps defined over quadratic number fields (the case $n=2$ of the UBC). Having fixed this value of $n$, we will focus on the simplest family of maps to which the conjecture applies, namely quadratic polynomials. Thus, we wish to address the following questions:
\begin{enumerate}[itemsep=1mm]
\item How large can the set $\PrePer(f,K)$ be as $K$ varies over all quadratic number fields and $f$ varies over all quadratic polynomials with coefficients in $K$?
\item What are all the possible graph structures corresponding to sets $\PrePer(f,K)$ as $K$ and $f$ vary as above?
\end{enumerate}

\subsection{Outline of the paper}\label{intro_outline} Our initial guesses for answers to the above questions were obtained by gathering large amounts of data, and doing this required an algorithm for computing all the preperiodic points of a given quadratic polynomial defined over a given number field. In \S\ref{data_gathering} we develop an algorithm for doing this which relies heavily on a new method, due to the first and third authors \cite{doyle-krumm}, for listing elements of bounded height in number fields. Our algorithm for computing preperiodic points can in principle be applied to quadratic polynomials over any number field, but we will focus here on the case of quadratic fields. Using this algorithm we computed the set $\PrePer(f,K)$ for roughly 250,000 pairs $(K,f)$ consisting of a quadratic field $K$ and a quadratic polynomial $f$ with coefficients in $K$. Our strategy for choosing the fields $K$ and polynomials $f$ is explained in \S\ref{quad_prep_comp}. The graph structures found by this computation are shown in Appendix \ref{graph_pictures}, and for each such graph $G$ we give in Appendix \ref{graph_data} an example of a pair $(K,f)$ for which the graph associated to $\PrePer(f,K)$ is isomorphic to $G$.

In order to state the more refined questions addressed in this article and our main results, we introduce some notation. From a dynamical standpoint, quadratic polynomials form a one-parameter family; more precisely, if $K$ is a number field and $f(z)\in K[z]$ is a quadratic polynomial, then $f(z)$ is equivalent, in a dynamical sense, to a unique polynomial of the form $f_c(z)=z^2+c$. (See the introduction to \S\ref{classification} for more details.) In studying the dynamical properties of quadratic polynomials, we will thus consider only polynomials of the form $f_c(z)$. We denote by $G(f_c,K)$ the directed graph corresponding to the set $\PrePer(f_c,K)$, excluding the point at infinity.

The graphs arising from our computation did not all occur with the same frequency: some of them appeared only a few times, while others were extremely common. For each graph $G$ that was found we may then ask: 
\begin{enumerate}[itemsep=1mm]
\item How many pairs $(K,c)$ are there for which the graph $G(f_c,K)$ is isomorphic to $G$?
\item If there are only finitely many such pairs, can they be completely determined?
\item If there are infinitely many such pairs, can they be explicitly described?
\end{enumerate}

Our strategy for addressing these questions is to translate them into Diophantine problems of determining the set of quadratic points on certain algebraic curves over $\Q$. In essence, the idea is to attach to each graph $G$ an algebraic curve $C$ whose points parameterize instances of the graph $G$. This philosophy of studying rational preperiodic points via algebraic curves was first taken up by Morton \cite{morton_4cycles}, and then pushed much further by Flynn-Poonen-Schaefer \cite{flynn-poonen-schaefer}, Poonen \cite{poonen_prep}, and Stoll \cite{stoll_6cycles}. However, the Diophantine questions we need to answer in this article differ from those studied by previous authors, since they were interested primarily in finding $\Q$-rational points on curves, whereas we need to determine all $K$-rational points on a given curve $C/\Q$, where $K$ is allowed to vary over all quadratic number fields. A survey of known theoretical results on this type of question is given in \S\ref{quad_pts_on_crvs}, where we also develop our basic computational methods for attacking the problem in practice. In \S\ref{classification} we construct the algebraic curves corresponding to the graphs found by our computation, and the methods of  \S\ref{quad_pts_on_crvs} are used to describe or completely determine their sets of quadratic points.

\subsection{Main results}\label{intro_results} In our extensive computation of preperiodic points mentioned above, we obtained a total of 46 non-isomorphic graphs, and the maximum number of preperiodic points for the polynomials considered was 15 (counting the fixed point at infinity). This data may provide the correct answers to questions (1) and (2) posed in \S\ref{intro_background}, though it is not our goal here to make this claim and attempt a proof. However, our computations do yield the following result:

\begin{thm} Suppose that there exists a constant $N$ such that $\#\PrePer(f,K) \leq N$ for every quadratic number field $K$ and quadratic polynomial $f$ with coefficients in $K$. Then $N \geq 15$. Moreover, there are at least 46 directed graphs that arise from the set $\PrePer(f, K)$ for such a field $K$ and polynomial $f$. 
\end{thm}

For each of the 46 graphs that were found we would like to answer questions (1) - (3) stated in \S\ref{intro_outline}. This can be done easily for 12 of the graphs, namely those that appeared in Poonen's paper \cite{poonen_prep} --- see \S\ref{rational_prep_section} below. The essential tool for this is Northcott's theorem on height bounds for the preperiodic points of a given map. For 15 of the 34 remaining graphs we were able to determine all pairs $(K,c)$ giving rise to the given graph. In most cases this was done by finding all quadratic points on the parameterizing curve of the graph, using our results concerning quadratic points on elliptic curves and on curves of genus 2 with Mordell-Weil rank 0. With the notation used in Appendix \ref{graph_pictures}, the labels for these graphs are
\begin{center}4(1), 5(1,1,)b, 5(2)a, 5(2)b, 6(2,1), 7(1,1)a, 7(1,1)b, 7(2,1,1)a, 

7(2,1,1)b, 9(2,1,1), 10(1,1)a, 12(2,1,1)a, 14(2,1,1), 14(3,1,1), 14(3,2). \end{center} 
With the exception of graph 5(2)a --- which occurs for two Galois conjugate pairs, namely $(\Q(i),\pm i)$ --- all of these graphs turned out to be unique; i.e., they occur for exactly one pair $(K,c)$. Moreover, the parameter $c$ in all of these pairs is rational.

Of the remaining 19 graphs, there are 4 for which we were not able to determine all pairs $(K,c)$ giving rise to the graph, but instead proved an upper bound on the number of all such pairs that could possibly exist. This was done by reducing the problem of determining all quadratic points on the parameterizing curve to a problem of finding all \textit{rational} points on certain hyperelliptic curves. Table \ref{upper_bound_table} below summarizes our results for these graphs. The first column gives the label of the graph under consideration, the second column gives the number of known pairs $(K,c)$ corresponding to this graph structure, and the third column gives an upper bound for the number of such pairs.

\begin{table}[ht]
\centering
\begin{tabular}{c c c}
\hline\hline
Graph & Known pairs & Upper bound \\
\hline\hline
12(2) & 2 & 6 \\
12(2,1,1)b & 2 & 6 \\
12(4) & 1 & 6 \\
12(4,2) & 1 & 2 \\
\hline
\end{tabular}
\bigskip
\caption{}
\label{upper_bound_table}
\end{table}

In order to complete our analysis of these four graphs we would need to determine all rational points on the hyperelliptic curves defined by the following equations:
\begin{align*}
&w^2 = x^7 + 3x^6 + x^5 - 3x^4 + x^3 + 3x^2 - 3x + 1,\\
&s^2 =-6x^7 - 6x^6 + 34x^5 + 22x^4 - 18x^3 + 38x^2 - 10x + 10,\\
&w^2 =x^{12} + 2x^{11} - 13x^{10} - 26x^9 + 67x^8 + 124x^7 + 26x^6 - 44x^5 + 179x^4 - 62x^3 - 5x^2 + 6x + 1,\\
&w^2 = (x^2 + 1)(x^2 - 2x - 1)(x^6-3x^4-16x^3+3x^2-1).
\end{align*}

Of the 15 graphs that remain to be considered, there are 9 for which we showed that the graph occurs infinitely many times over quadratic fields. This is achieved by using results from Diophantine geometry giving asymptotics for counting functions on the set of rational points on a curve.

\begin{thm} For each of the graphs
\begin{center}\rm{ 8(1,1)a, 8(1,1)b, 8(2)a, 8(2)b, 8(4), 10(2,1,1)a, 10(2,1,1)b}\end{center}
there exist infinitely many pairs $(K,c)$ consisting of a real (resp. imaginary) quadratic field and an element $c\in K$ for which $G(f_c,K)$ contains a graph of this type. The same holds for the graphs {\rm10(3,1,1)} and {\rm10(3,2)}, but these occur only over real quadratic fields.
\end{thm}

\begin{rem} Showing the existence of infinitely many pairs for which $G(f_c,K)$ not only \textit{contains} a graph of a given type but in fact \textit{is} itself of this type is a more difficult problem. This kind of result was achieved in the article \cite{faber} for several of the graphs $G(f_c,\Q)$ with $c\in\Q$.
\end{rem}

For the six graphs that remain, our methods did not yield a satisfactory upper bound on the number of possible instances --- see \S\ref{intro_future} below for more information.

Finally, we point out some results in this article which may be of independent interest. First, our description of quadratic points on elliptic curves in \S\ref{ellcrv_quad_pts_section}, though completely elementary, has been useful in practice for quickly generating many quadratic points on a given elliptic curve, as well as for proving several of our results stated above. The formula in \S\ref{genus2_quad_section} for the number of ``non-obvious" quadratic points on a curve of genus 2 does not seem to be explicitly stated in the literature, nor is its application (in Theorem \ref{131618_quad}) to the study of quadratic points on the modular curves $X_1(N)$ of genus 2. The techniques used here to determine all quadratic points on certain curves also appear to be new. As an example of our methods we mention our study of the graph 12(2,1,1)a in \S\ref{12_211a_section}, which illustrates our approach to finding all quadratic points on a curve $C$ having maps $C\to E_1$ and $C\to E_2$, where $E_1$ and $E_2$ are elliptic curves of rank 0. A second example is our study of the graph 12(4,2) in \S\ref{12_42_section}, in which we bound the number of quadratic points on a curve $C$ having a map $C\to X_1$, where $X_1$ is a curve of genus 2 and Mordell-Weil rank 0, and a map $C\to X_2$, where $X_2$ is a curve of genus 3 and Mordell-Weil rank 1.

\subsection{Future work}\label{intro_future} This paper leaves open several questions that we intend to address in subsequent articles. First of all, there are a few graphs in Appendix \ref{graph_pictures} for which we are at present not able to carry out a good analysis of the quadratic points on the corresponding curves; precisely, these are the graphs labeled 
\begin{center}10(1,1)b, 10(2), 10(3)a, 10(3)b, 12(3), and 12(6).\end{center} 
In most cases the difficulties do not seem insurmountable, but the methods required to study these graphs may be rather different from the ones used in this paper. Second, for some of the graphs analyzed in \S\ref{classification} we have only partially determined the quadratic points on the parameterizing curve, the obstruction being a problem of finding all rational points on certain hyperelliptic curves (listed above). We expect that the method of Chabauty and Coleman can be successfully applied to determine all rational points on these curves, thus completing our study of the corresponding graphs; this analysis will appear in a sequel to the present paper.

The next open question concerns 5-cycles. All evidence currently available suggests that there does not exist a quadratic polynomial $f$ defined over a quadratic field $K$ such that $f$ has a $K$-rational point of period 5. Such a polynomial did not show up in our computations, nor was it found in a related search carried out by Hutz and Ingram \cite{hutz-ingram}. We therefore set the following goal for future research:

\begin{center}Either find an example of a 5-cycle over a quadratic field, or show that it does not exist.\end{center}

As a result of their own extensive search for periodic points with large period defined over quadratic fields, Hutz and Ingram \cite[Prop. 2]{hutz-ingram} provide evidence supporting the conjecture that 6 is the longest cycle length that can appear in this setting. Moreover, they found exactly one example of a 6-cycle over a quadratic field, which is the same one found during our computations and the same one that had been found earlier by Flynn, Poonen, and Schaefer \cite[461]{flynn-poonen-schaefer}; namely the example given in Appendix \ref{graph_data} under the label 12(6). While the question of proving that 6 is the longest possible cycle length may be too ambitious, we do aim to study the following question:

\begin{center} Determine all instances of a 6-cycle over a quadratic field.\end{center}

We remark that it follows from known results in Diophantine geometry (see Theorem \ref{hs_finiteness} below) applied to the curve parameterizing 6-cycles \footnote{The geometry of this curve, and its set of rational points, were studied by Stoll in \cite{stoll_6cycles}. In particular, it is shown that the curve is not hyperelliptic and not bielliptic.} that there are only finitely many pairs $(K,c)$, with $K$ quadratic, giving rise to a 6-cycle.

As a final goal, we wish to prove a theorem analogous to Poonen's result (Theorem \ref{poonen_thm}), but in the context of quadratic fields. In view of the above discussion, we propose the following:

\begin{speculation} Assume that there is no quadratic polynomial over a quadratic field having a periodic point of period $n=5$ or $n>6$. Then, for every quadratic polynomial $f$ with coefficients in a quadratic field $K$, 
\[ \# \PrePer(f,K) \le 15. \]
Moreover, there are exactly 46 graphs that arise from $\PrePer(f, K)$ as $f$ varies over all quadratic polynomials with coefficients in a quadratic field $K$.
\end{speculation}

Using techniques similar to those applied in \S\ref{classification}, substantial progress towards proving this result --- or one very similar to it --- has already been made by the first author and will appear in a later paper.

\subsection*{Acknowledgements}
Part of the research for this article was undertaken during an NSF-sponsored VIGRE research seminar at the University of Georgia, supervised by the second author and Robert Rumely. The second author was partially supported by an NSF postdoctoral research fellowship. We would like to acknowledge the contributions of the other members of our research group: Adrian Brunyate, Allan Lacy, Alex Rice, Nathan Walters, and Steven Winburn. We are especially grateful to the Brown University Center for Computation and Visualization for providing access to their high performance computing cluster, and to the Institute for Computational and Experimental Research in Mathematics for facilitating this access. Finally, we thank Dino Lorenzini and Robert Rumely for many comments and suggestions; Anna Chorniy for her help in preparing the figures shown in Appendix \ref{graph_pictures}; and the anonymous referee for pointing out some gaps in our exposition in \S3.


\section{Quadratic points on algebraic curves}\label{quad_pts_on_crvs}
The material in this section forms the basis for our analysis of preperiodic graph structures in \S\ref{classification}. As mentioned in the introduction, questions concerning the sets of preperiodic points for quadratic polynomials over quadratic number fields can be translated into questions about the sets of quadratic points on certain algebraic curves defined over $\Q$. We will therefore require some basic facts about quadratic points on curves, from a theoretical as well as computational perspective.

Let $k$ be a number field, and fix an algebraic closure $\bar k$ of $k$. Let $C$ be a smooth, projective, geometrically connected curve defined over $k$. We say that a point $P\in C(\bar k)$ is {\it quadratic over} $k$ if $[k(P):k]=2$, where $k(P)$ denotes the field of definition of $P$, i.e., the residue field of $C$ at $P$. The set of all quadratic points on $C$ will be denoted by $C(k,2)$. It may well happen that $C(k,2)=\emptyset$: for instance, it follows from a theorem of Clark \cite[Cor. 4]{clark} that there are infinitely many curves of genus 1 over $k$ with this property. The set $C(k,2)$ may also be nonempty and finite; several examples of this will be seen in \S\ref{classification}. Finally, the set $C(k,2)$ may be infinite. Suppose, for instance, that $C$ is hyperelliptic, so that it admits a morphism $C\to\P^1$ of degree 2 defined over $k$. By pulling back $k$-rational points on $\P^1$ we obtain --- by Hilbert's irreducibility theorem \cite[Chap. 12]{fried-jarden} ---  infinitely many quadratic points on $C$. Similarly, suppose $C$ is bielliptic, so that it admits a morphism of degree 2 to an elliptic curve $E/k$. If the group $E(k)$ has positive rank, then the same argument as above shows that $C$ has infinitely many quadratic points. Using Faltings' theorem \cite{faltings} concerning rational points on subvarieties of abelian varieties (formerly a conjecture of Lang), Harris and Silverman \cite[Cor. 3]{harris-silverman} showed that these are the only two types of curves that can have infinitely many quadratic points.

\begin{thm}[Harris-Silverman]\label{hs_finiteness}
Let $k$ be a number field and let $C/k$ be a curve. If $C$ is neither bielliptic nor hyperelliptic, then the set $C(k,2)$ is finite.
\end{thm}

Thus, we have simple geometric criteria for deciding whether a given curve $C/k$ has finitely many or infinitely many quadratic points. However, we are not only interested here in abstract finiteness statements, but also in the practical question of explicitly determining all quadratic points on a curve, given a specific model for it. If the given curve has infinitely many quadratic points, then we will require an explicit description of all such points, and if it has only a finite number of quadratic points, then we require that they be determined. For the purposes of this paper we will mostly need to address these questions in the case of elliptic and hyperelliptic curves, or curves with a map of degree 2 to such a curve.

We begin by discussing quadratic points on hyperelliptic curves in general. Suppose that the curve $C/k$ admits a morphism $\phi: C\to\P_k^1$ of degree 2. Let $\sigma$ be the hyperelliptic involution on $C$, i.e., the unique involution such that $\phi\circ\sigma=\phi$. Corresponding to the map $\phi$ there is an affine model for $C$ of the form $y^2=f(x)$, where $f(x)\in k[x]$ has nonzero discriminant. With respect to this equation,  $\sigma$ is given by $(x,y)\mapsto (x,-y)$, and the quotient map $\phi: C\to C/\langle\sigma\rangle=\P_k^1$ is given by $(x,y)\mapsto x$. We wish to distinguish between two kinds of quadratic points on $C$; first, there is the following obvious way of generating quadratic points: by choosing any element $x_0\in k$ we obtain a point $(x_0,\sqrt{f(x_0)})\in C(\bar k)$ which will often be quadratic as we vary $x_0$. Indeed, Hilbert's irreducibility theorem implies that this will occur for infinitely many $x_0\in k$. Points of this form will be called {\it obvious quadratic points} for the given model. Stated differently, these are the quadratic points $P\in C(\bar k)$ such that $\phi(P)\in\P^1(k)$, or equivalently $\sigma(P)=\overline P$, where $\overline P$ is the Galois conjugate of $P$. Quadratic points that do not arise in this way will be called {\it non-obvious quadratic points} on $C$. Though a hyperelliptic curve always has infinitely many obvious quadratic points, this is not the case for non-obvious points; in fact, a theorem of Vojta \cite[Cor. 0.3]{vojta} implies that the set of non-obvious quadratic points on a hyperelliptic curve of genus $\ge 4$ is always finite. We focus now on studying the non-obvious quadratic points on elliptic curves and on curves of genus 2.

\subsection{Case of elliptic curves}\label{ellcrv_quad_pts_section} 
The following result gives an explicit description of all non-obvious quadratic points on an elliptic curve.

\begin{lem}\label{ellcrv_quad_pts} Let $E/k$ be an elliptic curve defined by an equation of the form \[y^2=ax^3+bx^2+cx+d,\] where $a,b,c,d\in k$ and $a\neq 0$. Suppose that  $(x,y)\in E(\bar k)$ is a quadratic point with $x\notin k$. Then there exist a point $(x_0,y_0)\in E(k)$ and an element $v\in k$ such that $y=y_0+v(x-x_0)$ and \[x^2 + \frac{ax_0 - v^2 + b}{a}x + \frac{ax_0^2 + v^2x_0 + bx_0 - 2y_0v + c}{a}=0.\]
\end{lem}

\begin{proof} Since $y\in k(x)$, we can write $y=p(x)$ for some polynomial $p(t)\in k[t]$ of degree at most 1. Note that $x$ is a root of the polynomial \[F(t):=at^3+bt^2+ct+d-p(t)^2,\]  so $F(t)$ must factor as $F(t)=a(t-x_0)m(t)$, where $m(t)$ is the minimal polynomial of $x$ and $x_0\in k$. Since $F(x_0)=0$, then $(x_0,p(x_0))\in E(k)$. Letting $y_0=p(x_0)$ we can write $p(t)=y_0+v(t-x_0)$ for some $v\in k$; in particular, $y=p(x)=y_0+v(x-x_0)$. Carrying out the division $F(t)/(a(t-x_0))$ we obtain \[m(t)=t^2 + \frac{ax_0 - v^2 + b}{a}t + \frac{ax_0^2 + v^2x_0 + bx_0 - 2y_0v + c}{a}.\qedhere\]
\end{proof}

\begin{rem}The description of quadratic points on $E$ given above has the following geometric interpretation: suppose $P=(x,y)\in E(\bar k)$ is quadratic over $k$, let $K=k(x,y)$ be the field of definition of $P$, and let $\sigma$ be the nontrivial element of $\Gal(K/k)$. We can then consider the point $Q=P+P^{\sigma}\in E(k)$, where $P^{\sigma}=(\sigma(x),\sigma(y))$ denotes the Galois conjugate of $P$. If $Q$ is the point at infinity on $E$, then the line through $P$ and $P^{\sigma}$ is vertical, so that $x=\sigma(x)$ and hence $x\in k$; this gives rise to obvious quadratic points on $E$. If $Q$ is not the point at infinity, then it is an affine point in $E(k)$, say $Q=(x_0,-y_0)$ for some elements $x_0,y_0\in k$. The points $P,P^{\sigma}$, and $(x_0,y_0)$ are collinear, and the line containing them has slope in $k$, say equal to $v\in k$. We then have $y=y_0+v(x-x_0)$, and this gives rise to the formula in Lemma \ref{ellcrv_quad_pts}.
\end{rem}

\subsection{Curves of genus 2}\label{genus2_quad_section}
Suppose now that $C/k$ is a curve of genus 2. Fix an affine model $y^2=f(x)$ for $C$, where $f(x)$ has degree 5 or 6, and let $\sigma$ be the hyperelliptic involution on $C$.

\medskip
\begin{lem}\label{genus2_quad_pts} Suppose that $C/k$ has genus 2 and $C(k)\neq\emptyset$. Let $J$ be the Jacobian variety of $C$.
\begin{enumerate}[itemsep = 1.1mm]
\item The set of  non-obvious quadratic points for the model $y^2=f(x)$ is finite if and only if $J(k)$ is finite.
\item Suppose that $J(k)$ is finite, and let $q$ denote the number of non-obvious quadratic points for the given model. Then there is a relation \[q=2j-2+w-c^2,\] where $j=\#J(k),\; c=\# C(k)$,\; and $w$ is the number of points in $C(k)$ that are fixed by $\sigma$.
\end{enumerate}
\end{lem}

\begin{proof} Fix a point $P_0\in C(k)$ and let $\iota:C\into J$ be the embedding taking $P_0$ to 0. Let $\ss=$ Sym$^2(C)$ denote the symmetric square of $C$. Points in $\ss(\bar k)$ correspond to unordered pairs $\{P,Q\}$, where $P,Q\in C(\bar k)$. The embedding $\iota$ induces a morphism $f:\ss\to J$ taking $\{P,Q\}$ to $\iota(P)+\iota(Q)$. We will need a few facts concerning the fibers of this morphism; see the article of Milne \cite{milne} for the necessary background material. There is a copy of $\P^1_k$ inside $\ss$ whose points correspond to pairs of the form $\{P,\sigma(P)\}$.  The image of $\P^1$ under $f$ is a single point $\ast\in J(k)$, and $f$ restricts to an isomorphism $f:U=\ss\backslash \P^1\;\stackrel{\sim}\longrightarrow \;J\backslash \{\ast\}$. In particular, there is a bijection \begin{equation}\label{genus2_aux1}U(k)=\ss(k)\backslash \P^1(k) \longleftrightarrow J(k)\backslash \{\ast\}.
\end{equation}

Points in $\ss(k)$ correspond to pairs of the form $\{P,Q\}$ where either $P$ and $Q$ are both in $C(k)$, or they are quadratic over $k$ and $Q=\overline P$; in particular, points in $\P^1(k)\subset\ss(k)$ correspond to pairs $\{P,\sigma(P)\}$ where either $P\in C(k)$ or $P$ is an obvious quadratic point. Finally, the points of $U(k)$ are either pairs $\{P,\overline P\}$ with $P$ a non-obvious quadratic point, or pairs $\{P,Q\}$ with $P,Q\in C(k)$ but $Q\neq\sigma(P)$.

Hence, there are three essentially distinct ways of producing points in $\ss(k)$: first, we can take points $P$ and $Q$ in $C(k)$ and obtain a point $\{P,Q\}\in\ss(k)$. Second, we can take an obvious quadratic point $P$ and obtain $\{P,\sigma(P)\}\in\P^1(k)\subset\ss(k)$. Finally, we can take a non-obvious quadratic point $P$ and obtain $\{P,\overline P\}\in U(k)\subset\ss(k)$.

Let $Q^o$ and $Q^n$ denote, respectively, the set of obvious and non-obvious quadratic points on $C$. We then have maps $\psi^o:Q^o\to \P^1(k)$ and $\psi^n:Q^n\to U(k)$, and a map $\phi:C(k)\times C(k)\to \ss(k)$ defined as above. The proof of the lemma will be a careful analysis of the images of these three maps.

We have $\ss(k)=\mathrm{im}(\phi)\;\sqcup\;\mathrm{im}(\psi^o)\;\sqcup\;\mathrm{im}(\psi^n)\;$. Removing the points of $\P^1(k)$ from both sides we obtain
\begin{equation}\label{genus2_aux2}U(k)=(\mathrm{im}(\phi)\backslash \P^1(k))\;\sqcup\;\mathrm{im}(\psi^n).
\end{equation}

To prove part (1), suppose first that $Q^n$ is finite. We know by Faltings' theorem that $C(k)$ is finite, so it follows from \eqref{genus2_aux2} that $U(k)$ is finite. By \eqref{genus2_aux1} we conclude that $J(k)$ is finite. Conversely, assume that $J(k)$ is finite. Then $U(k)$ is finite by \eqref{genus2_aux1}, so im$(\psi^n)$ is finite by \eqref{genus2_aux2}. But $\psi^n$ is 2-to-1 onto its image, so we conclude that $Q^n$ is finite. This completes the proof of part (1).

To prove part (2), suppose that $J(k)$ is finite and let  $q=\#Q^n$, so that $\#\mathrm{im}(\psi^n)=q/2$. By \eqref{genus2_aux1} and \eqref{genus2_aux2} we have
\begin{equation}\label{genus2_aux3}q/2=\# U(k)-\#(\mathrm{im}(\phi)\backslash \P^1(k))=j-1-\#(\mathrm{im}(\phi)\backslash \P^1(k)).\end{equation}

By simple combinatorial arguments we see that \[ \#\text{im}(\phi)=c+\frac{c(c-1)}{2}\;\;\;\mathrm{and}\;\;\;\; \#(\P^1(k)\cap\text{im}(\phi))=w+\frac{c-w}{2}.\] Therefore, \[\#(\mathrm{im}(\phi)\backslash \P^1(k))=c+\frac{c(c-1)}{2}-w-\frac{c-w}{2}=\frac{c^2-w}{2}.\] 

By \eqref{genus2_aux3} we then have \[j-1=\frac{q}{2}+\frac{c^2-w}{2},\] and part (2) follows immediately.
\end{proof}

\subsection{Application to the modular curves $X_1(N)$ of genus 2}\label{modular_genus2_section} For later reference we record here some consequences of Lemma \ref{genus2_quad_pts} in the particular case that $k=\Q$ and $C$ is one of the three modular curves $X_1(N)$ of genus 2. We will fix models for these curves to be used throughout this section. The following equations are given in \cite[15]{poonen_prep}, \cite[774]{washington}, and \cite[39]{rabarison}, respectively: 
\begin{align*}
X_1(13): & \; y^2=x^6+2x^5+x^4+2x^3+6x^2+4x+1;\\
X_1(16): &\; y^2=-x(x^2+1)(x^2-2x-1);\\
X_1(18): &\; y^2=x^6 + 2x^5 + 5x^4 + 10x^3 + 10x^2 + 4x + 1.
\end{align*}

\begin{thm}\label{131618_quad}\mbox{}
\begin{enumerate}
\item Every quadratic point on $X_1(13)$ is obvious.
\item The only non-obvious quadratic points on $X_1(16)$ are the following four: 
\[(\sqrt{-1},0),\; (-\sqrt{-1},0),\; (1+\sqrt 2,0), \;(1-\sqrt 2,0).\]
\item The only non-obvious quadratic points on $X_1(18)$ are the following four: \[(\omega,\omega-1), \;(\omega^2,\omega^2-1),\;(\omega,1-\omega),\;(\omega^2,1-\omega^2),\] where $\omega$ is a primitive cube root of unity.
\end{enumerate} 
\end{thm}

\begin{proof} It is known that the Jacobians $J_1(N)$ for $N\in\{13,16,18\}$ have only finitely many rational points (see \cite[\S 4]{mazur_tate_13} for the case of $J_1(13)$, \cite[Thm. 1]{kenku} for $J_1(16)$, and \cite[Thm. IV.5.7]{kubert} for $J_1(18)$). Hence, Lemma \ref{genus2_quad_pts} implies that the corresponding curves $X_1(N)$ have only finitely many non-obvious quadratic points. The various quantities appearing in the lemma are either known or can be easily computed. Indeed, Ogg \cite[226]{ogg} showed that $\#X_1(N)(\Q)=6$ for $N\in\{13,16,18\}$
and 
\[\#J_1(13)(\Q)=19\;,\;\#J_1(16)(\Q)=20\;,\;\#J_1(18)(\Q)=21.\]
The number $w$ in Lemma \ref{genus2_quad_pts} is determined by the number of rational roots of the polynomial $f(x)$.

Applying Lemma \ref{genus2_quad_pts} to the curve $C=X_1(13)$, we have $j=19, c=6, w=0$, and hence $q=0$. Therefore, all quadratic points on $X_1(13)$ are obvious.
 
Similarly, with $C=X_1(16)$ we have $j=20, c=6,w=2$, and hence $q=4$. We have already listed four non-obvious quadratic points, so these are all.

Finally, with $C=X_1(18)$ we have $j=21, c=6, w=0$, and hence $q=4$. Therefore, $X_1(18)$ has exactly four non-obvious quadratic points. Since we have already listed four such points, these must be all.
\end{proof}


\section{Classification of preperiodic graph structures}\label{classification} For each graph $G$ appearing in the list of 46 graphs in Appendix \ref{graph_pictures}, it is our goal in this section to describe --- as explicitly as possible --- all the quadratic fields $K$ and quadratic polynomials $f(z)\in K[z]$ such that the graph corresponding to the set $\PrePer(f,K)$ is isomorphic to $G$. There are several graphs in the list for which this description can be achieved without too much work: see \S\ref{rational_prep_section} below for the case of graphs arising from quadratic polynomials over $\Q$, and \S\ref{pcf_ufp} for other graphs with special properties. Together, these two sections will cover all graphs in the appendix up to and including the one labeled 7(2,1,1)b, and also 8(2,1,1), 8(3), and 9(2,1,1); in this section we will focus on studying the remaining graphs. Our general approach is to construct a curve parameterizing occurrences of a given graph, then apply results from \S\ref{quad_pts_on_crvs} to study the quadratic points on this curve, and from there obtain the desired description. As mentioned in \S\ref{intro_future}, there are a few graphs for which this approach does not yet yield the type of result we are looking for; hence, we will exclude these graphs from consideration in this section.

For the purpose of studying preperiodic points of quadratic polynomials over a number field $K$, it suffices to consider only polynomials of the form \[f_c(z) := z^2 + c\] with $c\in K$. Indeed, for every quadratic polynomial $f(z) \in K[z]$ there is a unique linear polynomial $g(z) \in K[z]$ and a unique $c \in K$ such that $g \circ f \circ g^{-1} = f_c$. One sees easily that the graph representing the set of preperiodic points is unchanged upon passing from $f$ to $g \circ f \circ g^{-1}$; hence, for our purposes we may restrict attention to the one-parameter family $\{f_c(z):c\in K\}$. For a quadratic polynomial $f(z)$ with coefficients in $K$, we denote by $G(f,K)$ the directed graph corresponding to the set of $K$-rational preperiodic points for $f$, excluding the fixed point at infinity. If $m$ and $n$ are positive integers, a \textit{point of type $m_n$} for $f(z)$ is an element $x\in K$ which enters an $m$-cycle after $n$ iterations of the map $f$. 

\subsection{Graphs occurring over $\Q$}\label{rational_prep_section} Given a rational number $c$ and a quadratic field $K$, we may consider the two sets $\PrePer(f_c,K)$ and $\PrePer(f_c,\Q)$. For all but finitely many quadratic fields $K$ these sets will be equal, since Northcott's theorem \cite[Thm. 3.12]{silverman_dynamics} implies that there are only finitely many quadratic elements of $\qbar$ that are preperiodic for $f_c$. Therefore, every graph appearing in Poonen's paper \cite{poonen_prep} --- that is, every graph of the form $G(f_c,\Q)$ with $c\in\Q$ --- will also occur as a graph $G(f_c,K)$ for some quadratic field $K$ (in fact, for all but finitely many such $K$). These graphs all appear in Appendix \ref{graph_pictures} and are labeled  
\begin{center}0, 2(1), 3(1,1), 3(2), 4(1,1), 4(2), 5(1,1){\rm a}, 6(1,1), 6(2), 6(3), 8(2,1,1), and 8(3).\end{center}

For every graph $G$ in the above list, Poonen provided an explicit parameterization of the rational numbers $c$ for which $G(f_c,\Q)\cong G$. This essentially achieves, for each of the graphs above, our stated goal of describing the pairs $(K,c)$ giving rise to a given graph. There still remains the following question, which will not be further discussed here: given $c\in\Q$, how can one determine the quadratic fields $K$ for which $\PrePer(f_c,K)\ne\PrePer(f_c,\Q)$ and moreover, what are all the graphs $G(f_c,K)$ that can arise in this way? From the data in Appendix \ref{graph_data} we see that many of the graphs shown in Appendix \ref{graph_pictures} are induced by a rational number $c$.

\subsection{Preliminaries}\label{analysis_prelims}
We collect here a few results that will be used repeatedly throughout this section.

Let $X$ be a  smooth, projective, geometrically integral curve defined over $\Q$; let $g$ denote the genus of $X$, and assume that $g\ge 2$. Let $r$ be the rank of the group $\Jac(X)(\Q)$, where $\Jac(X)$ denotes the Jacobian variety of $X$. By Faltings' theorem we know that $X(\Q)$ is a finite set; the following three results can be used to obtain explicit upper bounds on the size of this set under the assumption that $r<g$.

\begin{thm}[Coleman]\label{coleman_bound} Suppose that $r<g$ and let $p>2g$ be a prime of good reduction for $X$. Let $\calX / \Z_p$ be a model of $X$ with good reduction. Then \[ \#X(\Q) \le \#\calX(\F_p) + 2g-2. \]
\end{thm}
\begin{proof} See the proof of Corollary 4a in \cite{coleman} and the remark following the corollary.
\end{proof}

\begin{thm}[Lorenzini-Tucker] \label{lorenzini_tucker_bound} Assume that $r<g$. Let $p$ be a prime of good reduction for $X$, and let $\calX / \Z_p$ be a model of $X$ with good reduction. If $d$ is a positive integer such that $p>d$ and $p^d>2g-1+d$, then \[\#X(\Q)\le\#\calX(\F_p) + \left(\frac{p-1}{p-d}\right)(2g-2).\]
\end{thm}
\begin{proof} This follows from \cite[Thm. 1.1]{lorenzini-tucker}\end{proof}

\begin{thm}[Stoll]\label{stoll_bound}
Suppose that $r<g$ and let $p > 2r + 2$ be a prime of good reduction for $X$. Let $\calX / \Z_p$ be a model of $X$ with good reduction. Then \[ \#X(\Q) \le \#\calX(\F_p) + 2r. \]
\end{thm}
\begin{proof} This is a consequence of \cite[Cor. 6.7]{stoll_bound}.\end{proof}

The next result is crucial for obtaining equations for the parameterizing curves of the graphs to be considered in this section. Different versions of the various parts of this result have appeared elsewhere (for instance, \cite{poonen_prep} and \cite{walde-russo}), but not exactly in the form we will need. Hence, we include here the precise statements we require for our purposes.

\begin{prop}\label{cycles_prop} Let $K$ be a number field and let $f(z)=z^2+c$ with $c\in K$.
\begin{enumerate}[itemsep=1.2mm]
\item If $f(z)$ has a fixed point $p\in K$, then there is an element $x\in K$ such that
\[p=x+1/2\;\;, \;\;c=1/4-x^2.\]
Moreover, the point $p'=1/2-x$ is also fixed by $f(z)$.
\item If $f(z)$ has a point $p\in K$ of period 2, then there is a nonzero element $x\in K$ such that
\[p=x-1/2\;\;, \;\;c=-3/4-x^2.\]
Moreover, the orbit of $p$ under $f$ consists of the points $p$ and $f(p)=-x-1/2$.
\item If $f(z)$ has a point $p\in K$ of period 3, then there is an element $x\in K$ such that $x(x+1)(x^2+x+1)\ne 0$ and
\[p=\frac{x^3+2x^2+x+1}{2x(x+1)}\;\;,\;\;c=-\frac{ x^6 + 2x^5 + 4x^4 + 8x^3 + 9x^2 + 4x + 1}{4x^2(x+1)^2}.\]
Moreover, the orbit of $p$ under $f$ consists of the points $p$ and
\begin{align*}
f(p) &= \frac{x^3-x-1}{2x(x+1)},\\
f^2(p) &= -\frac{x^3+2x^2+3x+1}{2x(x+1)}.
\end{align*}
\item If $f(z)$ has a point $p\in K$ of period 4, then there are elements $x,y\in K$ with $y(x^2-1)\ne 0$ such that \[y^2 = F_{16}(x):=-x(x^2 + 1)(x^2 - 2x - 1)\] and
\[p = \frac{x-1}{2(x+1)} + \frac{y}{2x(x-1)}\;\;,\;\;c = \frac{(x^2 - 4x - 1)(x^4 + x^3 + 2x^2 - x + 1)}{4x(x-1)^2(x+1)^2}.\]
 Moreover, the orbit of $p$ under $f$ consists of the points $p$ and
\begin{align*}
f(p) &= -\frac{x+1}{2(x-1)} + \frac{y}{2x(x+1)},\\
f^2(p) &= \frac{x-1}{2(x+1)} - \frac{y}{2x(x-1)},\\
f^3(p) &= -\frac{x+1}{2(x-1)} - \frac{y}{2x(x+1)}.
\end{align*}
\end{enumerate}
\end{prop}

\begin{proof}\mbox{}
\begin{enumerate}
\item The equation $p^2+c=p$ can be rewritten as $(p-1/2)^2+c-1/4=0$. Letting $x=p-1/2$ we then have $x^2+c-1/4=0$, and the result follows. 
\item Since $f^2(p)=p$ and $f(p)\ne p$, then we have the equation
\[p^2+p+c+1=\frac{f^2(p)-p}{f(p)-p}=0.\] Letting $x=p+1/2$, this equation becomes $x^2+c+3/4=0$, and hence $c=-3/4-x^2$. Expressing $p$ and $c$ in terms of $x$ we obtain $f(p)=p^2+c=-x-1/2$. We must have $x\ne 0$ since $p$ and $f(p)$ are distinct.
\item See the proof of \cite[Thm. 3]{walde-russo}.
\item The existence of $x$ and $y$ and the expressions for $p$ and $c$ in terms of $x$ and $y$ can be obtained from the discussion in \cite[91-93]{morton_4cycles}; the expressions for the elements of the orbit of $p$ are obtained by a straightforward calculation from the expressions for $p$ and $c$. Finally, we must have $y\ne 0$ since otherwise $p$ would have period smaller than 4: indeed, note that $p=f^2(p)$ if $y=0$. \qedhere
\end{enumerate}
\end{proof}

\begin{rem} Note that part (2) of Proposition \ref{cycles_prop} implies that $f(z)$ can have at most two points of period 2 in $K$, so that the graph $G(f,K)$ can have at most one 2-cycle. This fact will be needed in the analysis of some of the graphs below.
\end{rem}

The following lemma will allow us to show that certain preperiodic graph structures occur infinitely many times over quadratic fields.

\begin{lem}\label{poly_lemma}
Let $p(x)\in\Q[x]$ have nonzero discriminant and degree $\ge 3$. For every rational number $r$, define a field $K_r$ by \[K_r:=\Q\left(\sqrt{p(r)}\right).\] Then, for every interval $I\subset\R$ of positive length, the set $\Sigma(I,p)=\{K_r:r\in\Q\cap I\}$ contains infinitely many quadratic fields. In particular, if the polynomial function $p:\R\to\R$ induced by $p(x)$ takes both positive and negative values, then $\Sigma(I,p)$ contains infinitely many real (resp. imaginary) quadratic fields.
\end{lem}

\begin{proof} For the proof of the statement that $\Sigma(I,p)$ contains infinitely many quadratic fields, see Appendix \ref{counting_appendix}. The second part follows from this statement by choosing an interval $I_1$ where $p>0$ and an interval $I_2$ where $p<0$.
\end{proof}

\subsection*{Use of computational software}
In preparing this article we have made extensive use of both the Magma \cite{magma} and Sage \cite{sage} computer algebra systems. Our data-gathering computations explained in \S\ref{quad_prep_comp} were carried out by implementing the algorithms of \S\ref{prep_algorithm} in Sage; these methods rely on the algorithm \cite{doyle-krumm} for listing elements of bounded height in number fields, which is also implemented in Sage. In Magma, we have made use of some rather sophisticated tools; in particular, we apply the \texttt{RankBound} function, which implements Stoll's algorithm \cite{stoll_2descent} of 2-descent for bounding the rank of the group of rational points on the Jacobian of a hyperelliptic curve over $\Q$. In addition, we frequently use the \texttt{CurveQuotient} function relying on Magma's invariant theory functionality to determine the quotient of a curve by an automorphism. Whenever this function is used in our paper, one can easily check by hand that the output is correct. Finally, for the analysis of the graph 14(3,1,1) in \S\ref{14_311_section} we require the \texttt{Chabauty} function, which implements a method due to Bruin and Stoll \cite[\S 4.4]{bruin-stoll} combining the method of Chabauty and Coleman with a Mordell-Weil sieve in order to determine the set of rational points on a curve of genus 2 with Jacobian of rank 1.


We can now proceed to the main task of this paper, namely to study the preperiodic graph structures appearing in Appendix \ref{graph_pictures}. We will consider the graphs one at a time, following the order in which they are listed in the appendix; however, the graphs discussed in \S\ref{rational_prep_section} and \S\ref{pcf_ufp} will not be considered henceforth in this section. The format for our discussion of each graph $G$ is roughly the same for all graphs: first, a parameterizing curve $C$ is constructed and an explicit map is given which shows how to use points on $C$ to obtain instances of $G$. Next, a theorem is proved which describes the set of quadratic points on $C$; finally, we use this theorem to deduce information about the instances of $G$ defined over quadratic fields. When there are infinitely many examples of a particular graph occurring over quadratic fields, we will also be interested in deciding whether it occurs over both real and imaginary fields, or only one type of quadratic field.

\subsection{Graph 8(1,1)a}\label{8_11a_section}
\begin{lem}\label{8_11a_curve} Let $C/\Q$ be the affine curve of genus 1 defined by the equation
 \[y^2 = -(x^2-3)(x^2+1).\]
Consider the rational map $\varphi: C \dashrightarrow \mathbb{A}^3 = \Spec \Q[a,b,c]$ given by 
\[a= -\frac{2x}{x^2 -1 } \;\;,\;\;b=\frac{y}{x^2-1}\;\;,\;\;c=\frac{-2(x^2 + 1)}{(x^2 - 1)^2}.\]
For every number field $K$, the map $\phi$ induces a bijection from the set $\{(x,y)\in C(K): (x^4-1)(x^2+3)\ne 0\}$ to the set of all triples $(a,b,c)\in K^3$ such that $a$ and $b$ are points of type $1_2$ for the map $f_c$ satisfying $f_c^2(a)\ne f_c^2(b)$. 
\end{lem}
\begin{figure}[h!]
\begin{center}\includegraphics[scale=0.5]{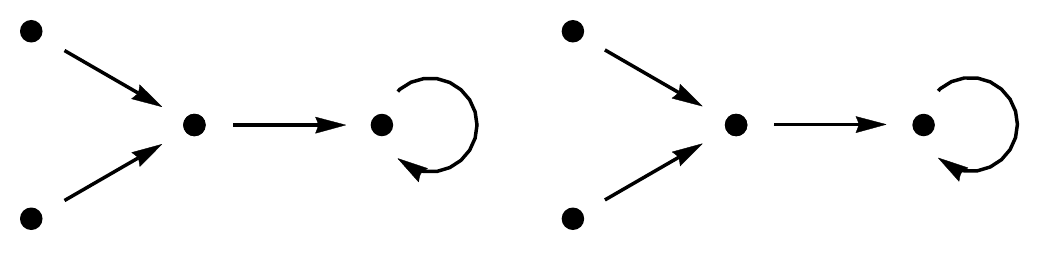}\end{center}
\caption{Graph type 8(1,1)a}
\end{figure}
\begin{proof}
Fix a number field $K$ and suppose that $(x,y)\in C(K)$ satisfies $x^2\ne 1$. Defining $a,b,c$ as in the lemma, it is straightforward to verify that $a$ is a point of type $1_2$ for the map $f_c$; that $f_c^2(b)$ is fixed by $f_c$; and that the following relations hold:
\begin{equation}\label{8_11a_relations}
f_c^2(b)-f_c(b)=\frac{2(x^2+1)}{x^2-1}\;\;,\;\;f_c^2(b)-f_c^2(a)=\frac{x^2+3}{x^2-1}.
\end{equation} 
It follows from these relations that if $(x^2+1)(x^2+3)\ne 0$, then $b$ is of type $1_2$ and $f_c^2(a)\ne f_c^2(b)$. Hence, $\phi$ gives a well-defined map.

To see that $\phi$ is surjective, suppose that $a,b,c\in K$ are such that $a$ and $b$ are points of type $1_2$ for the map $f_c$ satisfying $f_c^2(a)\ne f_c^2(b)$. The argument given in \cite[19]{poonen_prep} then shows that there exists a point $(x,y) \in C(K)$ with $x^2\ne 1$ such that $\phi(x,y)=(a,b,c)$. Furthermore, the relations \eqref{8_11a_relations} imply that necessarily $(x^2+1)(x^2+3)\ne 0$. To see that $\phi$ is injective, one can verify that if $\phi(x,y)=(a,b,c)$, then
\[x=\frac{-a}{a^2+c}\;\;,\;\;y=\frac{2b}{a^2+c}.\qedhere\]
\end{proof}

\begin{rem} As shown in \cite[19]{poonen_prep}, the curve $C$ is birational over $\Q$ to the elliptic curve 24a4 in Cremona's tables \cite{cremona}.
\end{rem} 
	
\begin{prop}
There are infinitely many real (resp. imaginary) quadratic fields $K$ containing an element $c$ for which $G(f_c,K)$ admits a subgraph of type {\rm 8(1,1)a}. 
\end{prop}

\begin{proof} 
Let $p(x)=-(x^2-3)(x^2+1)\in\Q[x]$. Applying Lemma \ref{poly_lemma} to the polynomial $p(x)$ we obtain infinitely many real (resp. imaginary) quadratic fields of the form $K_r=\Q(\sqrt{p(r)})$ with $r\in\Q$. For every such field there is a point $(r,\sqrt{p(r)})\in C(K_r)$ which necessarily satisfies $(r^4-1)(r^2+3)\ne 0$; hence, by Lemma \ref{8_11a_curve} there is an element $c\in K_r$ such that $f_c$ has points $a, b\in K_r$ of type $1_2$ with $f_c^2(a)\ne f_c^2(b)$. In order to conclude that $G(f_c,K_r)$ contains a subgraph of type 8(1,1)a we need the additional condition that $ab\ne 0$, so that the points $f_c(a)$ and $f_c(b)$ each have two distinct preimages. We have
\[ab^2=\frac{2r(r^2-3)(r^2+1)}{(r^2-1)^3},\]
so the condition $ab\ne 0$ will be satisfied as long as $r\ne 0$.
\end{proof}


\subsection{Graph 8(1,1)b}\label{8_11b_section}
\begin{lem}\label{8_11b_curve} Let $C/\Q$ be the affine curve of genus 1 defined by the equation
 \[y^2 = 2(x^3 + x^2 - x + 1).\]
Consider the rational map $\varphi: C \dashrightarrow \mathbb{A}^2 = \Spec \Q[p,c]$ given by 
\[p= \frac{y}{x^2 - 1} \;\;,\;\; c=\frac{-2(x^2 + 1)}{(x^2 - 1)^2}.\]
For every number field $K$, the map $\phi$ induces a bijection from the set $\{(x,y)\in C(K): x^2\ne 1\}$ to the set of all pairs $(p,c)\in K^2$ such that $p$ is a point of type $1_3$ for the map $f_c$. 
\end{lem}
\begin{figure}[h!]
\begin{center}\includegraphics[scale=0.5]{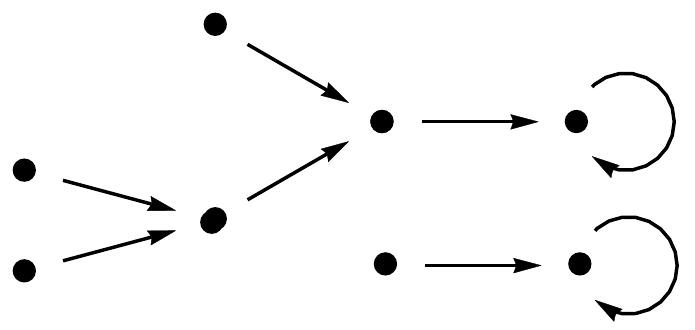}\end{center}
\caption{Graph type 8(1,1)b}
\end{figure}
\begin{proof}
Fix a number field $K$ and suppose that $(x,y)\in C(K)$ satisfies $x^2\ne 1$. Defining $p$ and $c$ as in the lemma, it is straightforward to verify that $p$ is a point of type $1_3$ for the map $f_c$. Hence, $\phi$ gives a well-defined map.

To see that $\phi$ is surjective, suppose that $p,c\in K$ are such that $p$ is a point of type $1_3$ for $f_c$. Then an argument given in \cite[22]{poonen_prep} shows that there exists a point $(x,y) \in C(K)$ with $x^2\ne 1$ such that $\phi(x,y)=(p,c)$. To see that $\phi$ is injective, one can verify that if $\phi(x,y)=(p,c)$, then
\[x=\frac{f_c(p)}{f_c^2(p)}\;\;,\;\;y=p(x^2-1).\qedhere\]
\end{proof}

\begin{rem} As shown in \cite[23]{poonen_prep}, the curve $C$ is birational over $\Q$ to the elliptic curve 11a3 in Cremona's tables \cite{cremona}, which is the modular curve $X_1(11)$.
\end{rem} 
	
\begin{prop}
There are infinitely many real (resp. imaginary) quadratic fields $K$ containing an element $c$ for which $G(f_c,K)$ admits a subgraph of type {\rm 8(1,1)b}. 
\end{prop}

\begin{proof} 
Let $q(x)=2(x^3 + x^2 - x + 1)\in\Q[x]$. Applying Lemma \ref{poly_lemma} to the polynomial $q(x)$ we obtain infinitely many real (resp. imaginary) quadratic fields of the form $K_r=\Q(\sqrt{q(r)})$ with $r\in\Q$. For every such field there is a point $(r,\sqrt{q(r)})\in C(K_r)$ which necessarily satisfies $r^2\ne 1$; hence, by Lemma \ref{8_11b_curve} there is an element $c\in K_r$ such that $f_c$ has a point $p\in K_r$ of type $1_3$. In order to conclude that $G(f_c,K_r)$ contains a subgraph of type 8(1,1)b we need the additional condition that 
$p\cdot f_c(p)\cdot c(4c-1)\ne 0$. Indeed, the condition $p\cdot f_c(p)\ne 0$ ensures that $f_c(p)$ and $f_c^2(p)$ each have two distinct preimages, while the condition $c(4c-1)\ne 0$ guarantees that $f_c$ has two distinct fixed points, and each fixed point has a preimage different from itself. Now, one can check that 
\[p^2\cdot f_c(p)\cdot c(4c-1)=\frac{8r(r^3+r^2-r+1)(r^2+1)(r^2+3)^2}{(r^2-1)^7},\]
so we will have $p\cdot f_c(p)\cdot c(4c-1)\ne 0$ as long as $r\ne 0$.
\end{proof}


\subsection{Graph 8(2)a}\label{8_2a_section}
\begin{lem}\label{8_2a_curve} Let $C/\Q$ be the affine curve of genus 1 defined by the equation
\[y^2 = 2(x^4 + 2x^3 - 2x + 1).\]

Consider the rational map $\varphi: C \dashrightarrow \mathbb{A}^3 = \Spec \Q[a,b,c]$ given by 
\[a= -\frac{x^2+1}{x^2-1} \;\;,\;\; b= \frac{y}{x^2-1} \;\;,\;\; c=- \frac{x^4 + 2x^3 + 2x^2 -2x + 1}{(x^2 - 1)^2}.\]
For every number field $K$, the map $\phi$ induces a bijection from the set 
\[\{(x,y)\in C(K): x(x^2-1)(x^2+4x-1)(x^2+2x-1)\ne 0\}\] to the set of all triples $(a,b,c)\in K^3$ such that $a$ and $b$ are points of type $2_2$ for the map $f_c$ satisfying $f_c^2(a)\ne f_c^2(b)$. 
\end{lem}
\begin{figure}[h!]
\begin{center}\includegraphics[scale=0.5]{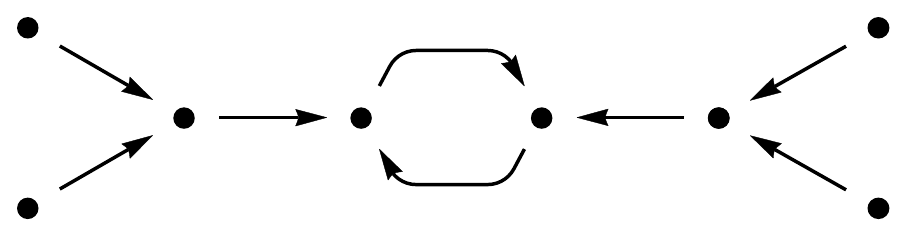}\end{center}
\caption{Graph type 8(2)a}
\end{figure}
\begin{proof}
Fix a number field $K$ and suppose that $(x,y)\in C(K)$ satisfies $x^2\ne 1$. Defining $a,b,c$ as in the lemma, it is a routine calculation to verify that $f_c^2(a)=f_c^4(a)$, $f_c^2(b)=f_c^4(b)$, $f_c^2(b)=f_c^3(a)$, and
\begin{equation}\label{8_2a_relations}
f_c^3(a)-f_c^2(a)=\frac{x^2+4x-1}{x^2-1}\;\;,\;\;f_c^3(a)-f_c(a)=\frac{4x}{x^2-1}\;\;,\;\;f_c(b)-f_c^3(b)=\frac{2(x^2+2x-1)}{x^2-1}.
\end{equation}
From these relations it follows that if $x(x^2+4x-1)(x^2+2x-1)\ne 0$, then $a$ and $b$ are points of type $2_2$ for $f_c$ with $f_c^2(a)\ne f_c^2(b)$. Hence, $\phi$ gives a well-defined map.

To see that $\phi$ is surjective, suppose that $a,b,c\in K$ are such that $a$ and $b$ are points of type $2_2$ for the map $f_c$ satisfying $f_c^2(a)\ne f_c^2(b)$. Since the map $f_c$ can have only one 2-cycle, the points $f_c^2(a)$ and $f_c^2(b)$ must form a 2-cycle. The argument given in \cite[20]{poonen_prep} then shows that there is a point $(x,y)\in C(K)$ with $x^2\ne 1$ such that $\phi(x,y)=(a,b,c)$. Furthermore, the relations \eqref{8_2a_relations} imply that $x(x^2+4x-1)(x^2+2x-1)\ne 0$. To see that $\phi$ is injective, one can verify that if $\phi(x,y)=(a,b,c)$, then
\[x=\frac{a-1}{a^2+c}\;\;,\;\;y=b(x^2-1).\qedhere\]
\end{proof}

\begin{rem} As shown in \cite[20]{poonen_prep}, the curve $C$ is birational over $\Q$ to the elliptic curve 40a3 in Cremona's tables \cite{cremona}.
\end{rem}

\begin{prop}
There are infinitely many real (resp. imaginary) quadratic fields $K$ containing an element $c$ for which $G(f_c,K)$ admits a subgraph of type {\rm 8(2)a}. 
\end{prop}

\begin{proof}
Let $p(x)=2(x^4 + 2x^3 - 2x + 1)\in\Q[x]$. Applying Lemma \ref{poly_lemma} to the polynomial $p(x)$ we obtain infinitely many quadratic fields of the form $K_r=\Q(\sqrt{p(r)})$ with $r\in\Q^{\ast}$. For every such field there is a point $(r,\sqrt{p(r)})\in C(K_r)$ which necessarily satisfies $(r^2-1)(r^2+4r-1)(r^2+2r-1)\ne 0$; hence, by Lemma \ref{8_2a_curve} there is an element $c\in K_r$ such that $f_c$ has points $a, b\in K_r$ of type $2_2$ satisfying $f_c^2(a)\ne f_c^2(b)$. In order to conclude that $G(f_c,K_r)$ contains a subgraph of type 8(2)a we need the additional condition $ab\ne 0$ so that $f_c(a)$ and $f_c(b)$ each have two distinct preimages. Now, one can check that
\begin{equation}\label{8_2a_exception}
-ab^2=\frac{2(r^2+1)(r^4 + 2r^3 - 2r + 1)}{(r^2-1)^3},
\end{equation}
so the condition $ab\ne 0$ is automatically satisfied since $r\in\Q$.

Note that the polynomial function $p:\R\to\R$ induced by $p(x)$ only takes positive values, so that all fields $K_r$ are real quadratic fields. Hence, this argument proves the statement only for real quadratic fields. To prove the statement for imaginary quadratic fields, we first obtain a Weierstrass equation for the elliptic curve birational to $C$. The following are inverse rational maps between $C$ and the elliptic curve $E$ with equation $Y^2=X^3-2X+1:$
\begin{align*}
(x,y) \mapsto & \left(\frac{2x^2 + y}{(x-1)^2},\frac{3x^3 + 3x^2 + 2xy - 3x + 1}{(x-1)^3}\right), \\
(X,Y)\mapsto & \left(\frac{X^2 + 2Y}{X^2 - 4X + 2},\frac{2X^4 + 8X^3 + 8X^2Y - 24X^2 - 8XY + 24X - 8}{(X^2 - 4X + 2)^2}\right).
\end{align*}
Let $q(X)=X^3-2X+1\in\Q[X]$. Applying Lemma \ref{poly_lemma} to the polynomial $q(X)$ we obtain infinitely many imaginary quadratic fields of the form $K_R=\Q(\sqrt{q(R)})$ with $R\in\Q$. For every such field there is a point $(R,\sqrt{q(R)})\in E(K_R)$; applying the change of variables above we obtain a point $(r,s)\in C(K_R)$ with
\begin{equation}\label{8_2a_rR}
r=\frac{R^2+2\sqrt{q(R)}}{R^2-4R+2}.\end{equation}
In particular, $r$ must satisfy $r(r^2-1)(r^2+4r-1)(r^2+2r-1)\ne 0$, since otherwise $r$ would be rational or would generate a real quadratic field. We can now apply Lemma \ref{8_2a_curve} to see that there is an element $c\in K_R$ such that $f_c$ has points $a, b\in K_R$ of type $2_2$ satisfying $f_c^2(a)\ne f_c^2(b)$. From \eqref{8_2a_exception} and \eqref{8_2a_rR} it follows that there are only finitely many values of $R\in\Q$ for which we might have $ab=0$. Hence, for all but finitely many values of $R$, this construction will yield a graph $G(f_c,K_R)$ containing a subgraph of type 8(2)a.
\end{proof}


\subsection{Graph 8(2)b}\label{8_2b_section}
\begin{lem}\label{8_2b_curve} Let $C/\Q$ be the affine curve of genus 1 defined by the equation
\[y^2 = 2(x^3 + x^2 - x + 1).\]
Consider the rational map $\varphi: C \dashrightarrow \mathbb{A}^2 = \Spec \Q[p,c]$ given by 
\[p= \frac{y}{x^2 - 1} \;\;,\;\; c=- \frac{x^4 + 2x^3 + 2x^2 -2x + 1}{(x^2 - 1)^2}.\]
For every number field $K$, the map $\phi$ induces a bijection from the set $\{(x,y)\in C(K): x(x^2-1)(x^2+4x-1)\ne 0\}$ to the set of all pairs $(p,c)\in K^2$ such that $p$ is a point of type $2_3$ for the map $f_c$. 
\end{lem}
\begin{figure}[h!]
\begin{center}\includegraphics[scale=0.5]{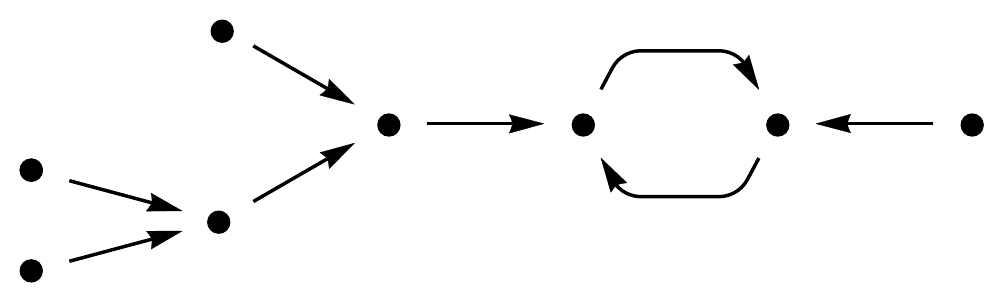}\end{center}
\caption{Graph type 8(2)b}
\end{figure}
\begin{proof}
Fix a number field $K$ and suppose that $(x,y)\in C(K)$ satisfies $x^2\ne 1$. Defining $p,c$ as in the lemma, it is straightforward to verify that $f_c^3(p)=f_c^5(p)$ and
\begin{equation}\label{8_2b_relations}
f_c^4(p)-f_c^2(p)=\frac{4x}{x^2-1}\;\;,\;\;f_c^4(p)-f_c^3(p)=\frac{x^2+4x-1}{x^2-1}.
\end{equation}
From these relations it follows that if $x(x^2+4x-1)\ne 0$, then $p$ is of type $2_3$ for $f_c$. Hence, $\phi$ gives a well-defined map.

To see that $\phi$ is surjective, suppose that $p,c\in K$ are such that $p$ is a point of type $2_3$ for the map $f_c$. Then an argument given in \cite[23]{poonen_prep} shows that there is a point $(x,y)\in C(K)$ with $x^2\ne 1$ such that $\phi(x,y)=(p,c)$. Furthermore, the relations \eqref{8_2b_relations} imply that we must have $x(x^2+4x-1)\ne 0$.  To see that $\phi$ is injective, one can verify that if $\phi(x,y)=(p,c)$, then
\[x=\frac{f_c(p)-1}{f_c^2(p)}\;\;,\;\;y=p(x^2-1).\qedhere\]
\end{proof}

\begin{rem} The curve $C$ is the same curve parameterizing the graph 8(1,1)b. As noted earlier, $C$ is birational to the modular curve $X_1(11)$.
\end{rem}
	
\begin{prop}
There are infinitely many real (resp. imaginary) quadratic fields $K$ containing an element $c$ for which $G(f_c,K)$ admits a subgraph of type {\rm 8(2)b}. 
\end{prop}

\begin{proof}
Let $q(x)=2(x^3 + x^2 - x + 1)\in\Q[x]$. Applying Lemma \ref{poly_lemma} to the polynomial $q(x)$ we obtain infinitely many real (resp. imaginary) quadratic fields of the form $K_r=\Q(\sqrt{q(r)})$ with $r\in\Q^{\ast}$. For every such field there is a point $(r,\sqrt{q(r)})\in C(K_r)$ which necessarily satisfies $r^2\ne 1$; hence, by Lemma \ref{8_2b_curve} there is an element $c\in K_r$ such that $f_c$ has a point $p\in K_r$ of type $2_3$. In order to conclude that $G(f_c, K_r)$ contains a subgraph of type 8(2)b we need the additional condition $p\cdot f_c(p)\cdot f_c^3(p)\ne 0$ so that $f_c(p), f_c^2(p), $ and $f_c^4(p)$ each have two distinct preimages. One can check that
\[p^2\cdot f_c(p)\cdot f_c^3(p)=\frac{2(r^3+r^2-r+1)(r^2+1)(r^2+2r-1)}{(r^2-1)^4},\]
so in fact the condition $p\cdot f_c(p)\cdot f_c^3(p)\ne 0$ is automatically satisfied since $r\in\Q$.
\end{proof}


\subsection{Graph 8(4)}\label{8_4_section}
\begin{lem}\label{8_4_curve}
Let $C/\Q$ be the affine curve of genus 2 defined by the equation
		\[y^2 = F_{16}(x) := -x(x^2 + 1)(x^2 - 2x - 1).\]
Consider the rational map $\phi : C \dashrightarrow \mathbb{A}^2 = \Spec \Q[p,c]$ given by
\begin{equation}\label{8_4_pc}
p = \frac{x-1}{2(x+1)} + \frac{y}{2x(x-1)}\;\;,\;\;c = \frac{(x^2 - 4x - 1)(x^4 + x^3 + 2x^2 - x + 1)}{4x(x-1)^2(x+1)^2}.
\end{equation}
For every number field $K$, the map $\phi$ induces a bijection from the set $\{ (x,y) \in C(K) : y(x^2-1) \ne 0 \}$
to the set of all pairs $(p,c) \in K^2$ such that $p$ is a point of period 4 for the map $f_c$.
\end{lem}
\begin{figure}[h!]
\begin{center}\includegraphics[scale=0.5]{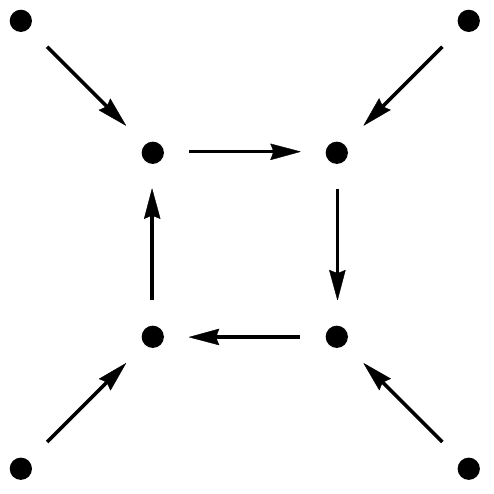}\end{center}
\caption{Graph type 8(4)}
\end{figure}
\begin{proof}
Fix a number field $K$ and suppose that $(x,y) \in C(K)$ satisfies $x(x^2-1) \ne 0$. Defining $p$ and $c$ as in the lemma, it is straightforward to verify the relations
\begin{equation}\label{8_4_relations}
f_c^4(p) = p\;\;,\;\;p-f_c^2(p)=\frac{y}{x(x-1)}.
\end{equation}
The condition that $y \ne 0$ thus implies that $p$ has period 4. Hence, $\phi$ gives a well-defined map. The fact that $\phi$ is surjective follows immediately from Proposition \ref{cycles_prop}.  To see that $\phi$ is injective, one can verify that if $\phi(x,y)=(p,c)$, then
\[x=\frac{1+f_c^2(p)+p}{1-f_c^2(p)-p}\;\;,\;\;y=\frac{2px(x^2-1)-x(x-1)^2}{x+1}.\qedhere\]
\end{proof}

\begin{rem} As noted in \S\ref{modular_genus2_section}, the curve $C$ is birational over $\Q$ to the modular curve $X_1(16)$.
\end{rem}

\begin{prop}\label{infinite_four_cycles}
There are infinitely many real (resp. imaginary) quadratic fields $K$ containing an element $c$ for which $G(f_c,K)$ admits a subgraph of type {\rm 8(4)}.
\end{prop}

\begin{proof}
Applying Lemma \ref{poly_lemma} to the polynomial $F_{16}(x)$ we obtain infinitely many real (resp. imaginary) quadratic fields of the form $K_r=\Q(\sqrt{F_{16}(r)})$ with $r\in\Q^{\ast}$. For every such field there is a point $(r,\sqrt{F_{16}(r)})\in C(K_r)$ which necessarily satisfies $(r^2-1)\sqrt{F_{16}(r)}\ne 0$; hence, by Lemma \ref{8_4_curve} there is an element $c\in K_r$ such that $f_c$ has a point $p\in K_r$ of period 4. In order to conclude that $G(f_c, K_r)$ contains a subgraph of type 8(4) we need the additional condition that every point in the orbit of $p$ is nonzero, so that each point has two distinct preimages. Using Proposition \ref{cycles_prop} (4) we find that if $p\cdot f_c(p)\cdot f_c^2(p)\cdot f_c^3(p)=0$, then one of the following equations will be satisfied:
\begin{align*}
 r(r-1)^4+(r^2+1)(r+1)^2(r^2-2r-1) &=0, \\
 r(r+1)^4+(r^2+1)(r-1)^2(r^2-2r-1)&=0. \\
 \end{align*}
 However, one can verify that neither of these equations has a rational solution; therefore, every point in the orbit of $p$ is necessarily nonzero.
\end{proof}

\begin{rem}
Our search in \S\ref{quad_prep_comp} failed to produce an example of a graph of type 8(4) over an imaginary quadratic field, though the previous proposition suggests that there are infinitely many such examples. The reason for this is that, even for rational parameters $x$ of moderate size, the discriminant of the field $K=\Q(\sqrt{F_{16}(x)})$ may be large, and the complexity of the rational function defining $c$ forces the height of $c$ to be large as well, thus placing the pair $(K,c)$ outside of our search range. We obtain an instance of this graph by taking $x = 5$ in \eqref{8_4_pc}; this leads to the pair $(K,c) = (\Q(\sqrt{-455}), 199/720)$. A computation of preperiodic points using the algorithm developed in \S\ref{prep_algorithm} shows that, indeed, the graph $G(f_c, K)$ for this pair $(K,c)$ is of type 8(4).
\end{rem}

We end our discussion of the graph type 8(4) by stating explicitly how to obtain all pairs $(K,c)$ consisting of a quadratic field $K$ and an element $c\in K$ for which $G(f_c,K)$ is of this type.

\begin{thm} Let $K$ be a quadratic field. Suppose that there exists an element $c\in K$ such that $G(f_c,K)$ is of type {\rm 8(4)}. Then there is a rational number $x\notin\{0,\pm 1\}$ such that 
\begin{equation}\label{8_4_c}
c = \frac{(x^2 - 4x - 1)(x^4 + x^3 + 2x^2 - x + 1)}{4x(x-1)^2(x+1)^2}
\end{equation} 
and $K=\Q(\sqrt{-x(x^2 + 1)(x^2 - 2x - 1)})$.
\end{thm}

\begin{proof}
By Lemma \ref{8_4_curve} there exists a point $(x,y) \in C(K)$ with $y(x^2-1) \ne 0$ such that $c$ is given by \eqref{8_4_c}. We claim that $(x,y)$ cannot be a rational point on $C$. Indeed, $C$ is an affine model of the curve $X_1(16)$, which has exactly six rational points. Therefore, $C$ has five rational points, and it is easy to see that they are $(0,0)$ and $(\pm 1,\pm 2)$.  Thus, if $(x,y)\in C(\Q)$, then either $x=0$ or $x=\pm 1$; however, both possibilities are precluded by the fact that $y(x^2-1)\ne 0$. Therefore, $(x,y)$ is a quadratic point on $C$.  Since $y\ne 0$, it follows from Theorem \ref{131618_quad} that $x$ must be a rational number. In particular, since $(x,y)$ is a quadratic point, then $K=\Q(x,y)=\Q(y)=\Q(\sqrt{-x(x^2 + 1)(x^2 - 2x - 1)})$.
\end{proof}


\subsection{Graph 10(1,1)a}\label{10_11a_section}
Our search described in \S\ref{quad_prep_comp} produced the pair \[(K,c)=\left(\Q(\sqrt{-7}), \frac{3}{16}\right)\] for which the graph $G(f_c,K)$ is of type 10(1,1)a. We show here that this is the only such pair $(K,c)$ with $K$ a quadratic number field and $c\in K$.
\begin{lem}\label{10_11a_curve} Let $C/\Q$ be the affine curve of genus 4 defined by the equations
\begin{equation}\label{10_11a_curve_equations}
\begin{cases}
& y^2 = 2(x^3+x^2-x+1)\\
& z^2 = -2(x^3-x^2-x-1).
\end{cases}
\end{equation}

Consider the rational map $\varphi: C \dashrightarrow \mathbb{A}^3 = \Spec \Q[a,b,c]$ given by 
\[a=\frac{y}{x^2-1}\;\;,\;\;b=\frac{z}{x^2-1}\;\;,\;\;c=\frac{-2(x^2 + 1)}{(x^2 - 1)^2}.\]
For every number field $K$, the map $\phi$ induces a bijection from the set $\{(x,y,z)\in C(K): x(x^2-1)\ne 0\}$ to the set of all triples $(a,b,c)\in K^3$ such that $a$ and $b$ are points of type $1_3$ for the map $f_c$ satisfying $f_c^2(a)=f_c^2(b)$ and $f_c(a)\ne f_c(b)$. 
\end{lem}
\begin{figure}[h!]
\begin{center}\includegraphics[scale=0.5]{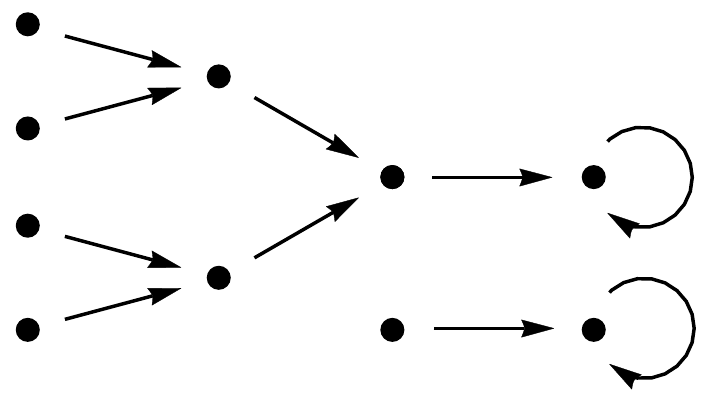}\end{center}
\caption{Graph type 10(1,1)a}
\end{figure}
\begin{proof} Fix a number field $K$ and suppose that $(x,y,z)\in C(K)$ satisfies $x^2\ne 1$. Defining $a,b,c$ as in the lemma, it is a routine calculation to verify that $a$ and $b$ are points of type $1_3$ for the map $f_c$ satisfying $f_c^2(a)=f_c^2(b)$; moreover, we have the relation
\begin{equation}\label{10_11a_relation}
f_c(a)-f_c(b)=\frac{4x}{x^2-1}.
\end{equation}
It follows that if $x\ne 0$, then $f_c(a)\ne f_c(b)$. Hence, $\phi$ gives a well-defined map. 

To see that $\phi$ is surjective, suppose that $a,b,c\in K$ are such that $a$ and $b$ are points of type $1_3$ for the map $f_c$ satisfying $f_c^2(a)=f_c^2(b)$ and $f_c(a)\ne f_c(b)$. Then an argument given in \cite[22]{poonen_prep} shows that there is an element $x\in K\setminus\{\pm 1\}$ such that 
\begin{equation}\label{10_type1_3} c = \frac{-2(x^2 + 1)}{(x^2 - 1)^2}\;\mathrm{\;and\;}\; a^2 = \frac{2(x^3+x^2-x+1)}{(x^2 - 1)^2}.\end{equation} 
Since $a^2+c=f_c(a)=-f_c(b)=-b^2-c$, then using \eqref{10_type1_3} we obtain \[b^2=\frac{-2(x^3-x^2-x-1)}{(x^2-1)^2}.\] Letting $y=a(x^2-1)$ and $z=b(x^2-1)$ we then have a point $(x,y,z)\in C(K)$ with $x^2\ne 1$ such that $\phi(x,y,z)=(a,b,c)$. Furthermore, the relation \eqref{10_11a_relation} implies that $x\ne 0$.  To see that $\phi$ is injective, one can verify that if $\phi(x,y,z)=(a,b,c)$, then
\[x=\frac{f_c(a)}{f_c^2(a)}\;\;,\;\;y=a(x^2-1)\;\;,\;\;z=b(x^2-1).\qedhere\]
\end{proof}

\begin{thm}\label{10_11a_points} With $C$ as in Lemma \ref{10_11a_curve} we have the following:
\begin{enumerate}
\item $C(\Q)=\{(\pm 1,\pm 2,\pm 2)\}$.
\item If $K$ is a quadratic field different from $\Q(\sqrt{2})$ and $\Q(\sqrt{-7})$, then $C(K)=C(\Q)$.
\item For $K=\Q(\sqrt{2})$, $C(K)\setminus C(\Q)=\{(0,\pm\sqrt 2,\pm\sqrt 2)\}$.
\item For $K=\Q(\sqrt{-7})$, $C(K)\setminus C(\Q)$ consists of the points $(x,\pm(2x-4),\pm(2x+4))$ with $x^2+7=0$.
\end{enumerate}
\end{thm}

\begin{proof} The equation $y^2=2(x^3 + x^2 - x + 1)$ defines the elliptic curve with Cremona label 11a3. This curve has rank 0 and torsion order 5; the affine rational points are $(\pm 1, \pm 2)$. The equation $z^2 = -2(x^3-x^2-x-1)$ defines the same elliptic curve, and the affine rational points are again $(\pm 1, \pm 2)$. Therefore, $C(\Q)=\{(\pm 1,\pm 2,\pm 2)\}$. 

Suppose now that $(x,y,z)\in C(\overline\Q)$ is a point with $[\Q(x,y,z):\Q]=2$, and let $K=\Q(x,y,z)$.

{\bf Case 1:} $x\in\Q$. We cannot have $x=\pm 1$, since this would imply that $y=\pm 2$ and $z=\pm 2$, contradicting the assumption that $(x,y,z)$ is a quadratic point on $C$; hence, $x\ne\pm 1$. It follows that $y\notin \Q$, since having $x,y\in\Q$ would imply that $x=\pm 1$. By the same argument, $z\notin\Q$. Therefore, $K=\Q(y)=\Q(z)$, so the rational numbers $2(x^3+x^2-x+1)$ and $-2(x^3-x^2-x-1)$ have the same squarefree part, and thus their product is a square. Hence, there is a rational number $w$ such that 
\begin{equation}\label{10_11a_hyper} 
w^2= -x^6 + 3x^4 + x^2 + 1.
\end{equation} 
Let $X$ be the hyperelliptic curve of genus 2 defined by \eqref{10_11a_hyper}. We claim that the following is a complete list of rational points on $X$: \[X(\Q)=\{(\pm 1, \pm 2), (0,\pm 1)\}.\] To see this, note that $X$ has a nontrivial involution given by $(x,w)\mapsto (-x,-w)$. The quotient of $X$ by this involution is the elliptic curve $E$ defined by the equation $v^2 = u^3 + u^2 + 3u - 1$; this is the curve 44a1 in Cremona's tables. The quotient map $X\to E$ of degree 2 is given by $(x,w)\mapsto (1/x^2, w/x^3)$. The curve $E$ has exactly 3 rational points, namely the point at infinity and the points $(1,\pm 2)$. It follows that $X$ has at most six rational points, and since we have already listed six points, these must be all.

Returning to \eqref{10_11a_hyper}, we have a point $(x,w)\in X(\Q)$ with $x\neq\pm 1$, so we conclude that $x=0$. By \eqref{10_11a_curve_equations} we then have $y^2=z^2=2$. Thus, we have shown that the only quadratic points on $C$ having a rational $x$-coordinate are the points $(x,y,z)=(0,\pm\sqrt 2,\pm\sqrt 2)$.

{\bf Case 2:} $x$ is quadratic. By Lemma \ref{ellcrv_quad_pts} applied to the equation $y^2 = 2(x^3+x^2-x+1)$, there exist a rational number $v$ and a point $(x_0,y_0)\in \{(\pm1 ,\pm 2)\}$ such that 
\begin{equation}\label{10_11a_xpoly1} x^2+\frac{2x_0-v^2+2}{2}x+\frac{2x_0^2+v^2x_0+2x_0-2y_0v-2}{2}=0.\end{equation} 
Similarly, applying Lemma \ref{ellcrv_quad_pts} to the equation $z^2 = -2(x^3-x^2-x-1)$ we see that there exist a rational number $w$ and a point $(x_1,z_1)\in \{(\pm1 ,\pm 2)\}$ such that 
\begin{equation}\label{10_11a_xpoly2} x^2+\frac{2x_1+w^2-2}{2}x+\frac{2x_1^2-w^2x_1-2x_1+2z_1w-2}{2}=0.\end{equation}
Comparing \eqref{10_11a_xpoly1} and \eqref{10_11a_xpoly2} we obtain the system

\[
\begin{cases} & 2x_0-v^2+4 = 2x_1+w^2 \\ 
& 2x_0^2+v^2x_0+2x_0-2y_0v = 2x_1^2-w^2x_1-2x_1+2z_1w. 
\end{cases}
\]

For each choice of points $(x_0,y_0), (x_1,z_1)$ the above system defines a zero-dimensional scheme $S$ in the $(v,w)$ plane over $\Q$, and hence all its rational points may be determined. There are a total of 16 choices of pairs $(x_0,y_0), (x_1,z_1)$, leading to 16 schemes $S$. Using the Magma function \texttt{RationalPoints} we find all the rational points $(v,w)$ on these schemes, and in every case check whether the polynomial \eqref{10_11a_xpoly1} is irreducible. This occurs for four of these schemes, and for all of these \eqref{10_11a_xpoly1} becomes $x^2+7=0$. The equations \eqref{10_11a_curve_equations} now imply that $y^2=(2x-4)^2$ and $z^2=(2x+4)^2$. Thus, we have shown that the only quadratic points on $C$ having a quadratic $x$-coordinate are those of the form $\left(x,\pm(2x-4),\pm(2x+4)\right)$ with $x^2+7=0$.

In conclusion, the quadratic points on $C$ are the points $(0,\pm\sqrt 2,\pm\sqrt 2)$ defined over the field $\Q(\sqrt 2)$ and the points $\left(x,\pm(2x-4),\pm(2x+4)\right)$ where $x^2+7=0$, defined over the field $\Q(\sqrt{-7})$; the theorem now follows immediately.
\end{proof}

\begin{cor}\label{10_11a_pairs} Let $K$ be a quadratic field and let $c\in K$. Suppose that $G(f_c,K)$ contains a graph of type \rm{10(1,1)a}. Then $c=3/16$ and $K=\Q(\sqrt{-7})$.
\end{cor}

\begin{proof} By Lemma \ref{10_11a_curve} there is a point $(x,y,z)\in C(K)$ with $x(x^2-1)\ne 0$ such that 
\begin{equation}\label{10_11a_cvalue}c = \frac{-2(x^2 + 1)}{(x^2 - 1)^2}.
\end{equation} 
It follows from Theorem \ref{10_11a_points} that $(x,y,z)$ is a quadratic point on $C$ with $x\ne 0$. Therefore, $K=\Q(\sqrt{-7})$ and $x^2+7=0$. From \eqref{10_11a_cvalue} we then obtain $c =3/16$.
\end{proof}


\subsection{Graph 10(2,1,1)a}\label{10_211a_section}
\begin{lem}\label{10_211a_curve} Let $C/\Q$ be the affine curve of genus 1 defined by the equation
 \[y^2 = 5x^4-8x^3+6x^2+8x+5.\]
Consider the rational map $\varphi: C \dashrightarrow \mathbb{A}^3 = \Spec \Q[a,b,c]$ given by 
\[a= \frac{3x^2+1}{2(x^2 -1)} \;\;,\;\;b=\frac{y}{2(x^2-1)}\;\;,\;\;c=-\frac{3x^4+10x^2+3}{4(x^2 - 1)^2}.\]
For every number field $K$, the map $\phi$ induces a bijection from the set $\{(x,y)\in C(K): x(x^2-1)(x^2-4x-1)\ne 0\}$ to the set of all triples $(a,b,c)\in K^3$ such that $a$ is a fixed point and $b$ is a point of type $2_2$ for the map $f_c$. 
\end{lem}
\begin{figure}[h!]
\begin{center}\includegraphics[scale=0.5]{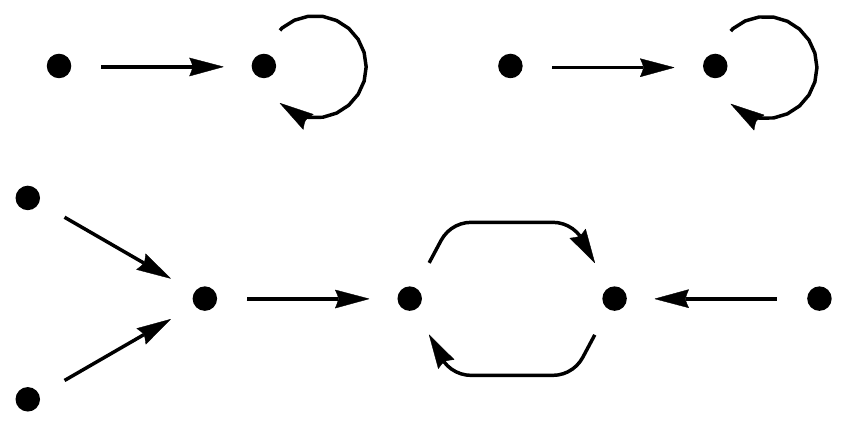}\end{center}
\caption{Graph type 10(2,1,1)a}
\end{figure}
\begin{proof}
Fix a number field $K$ and suppose that $(x,y)\in C(K)$ satisfies $x^2\ne 1$. Defining $a,b,c$ as in the lemma, it is straightforward to verify that $a$ is a fixed point for $f_c$; that $f_c^2(b)=f_c^4(b)$, and that the following relations hold:
\begin{equation}\label{10_211a_relations}
f_c^3(b)-f_c^2(b)=\frac{4x}{x^2-1}\;\;,\;\;f_c(b)-f_c^3(b)=\frac{x^2-4x-1}{x^2-1}.
\end{equation}
From these relations it follows that if $x(x^2-4x-1)\ne 0$, then $b$ is of type $2_2$ for $f_c$. Hence, $\phi$ gives a well-defined map.

To see that $\phi$ is surjective, suppose that $a,b,c\in K$ are such that $a$ is a fixed point and $b$ is a point of type $2_2$ for the map $f_c$. Then an argument given in \cite[21-22]{poonen_prep} shows that there exists a point $(x,y) \in C(K)$ with $x(x^2-1)\ne 0$ such that $\phi(x,y)=(a,b,c)$. Furthermore, the relations \eqref{10_211a_relations} imply that we must have $x^2-4x-1\ne 0$.  To see that $\phi$ is injective, one can verify that if $\phi(x,y)=(a,b,c)$, then
\[x=\frac{1+2f_c^2(b)}{3-2a}\;\;,\;\;y=2b(x^2-1).\qedhere\]
\end{proof}

\begin{rem} As shown in \cite[22]{poonen_prep}, the curve $C$ is birational over $\Q$ to the elliptic curve 17a4 in Cremona's tables \cite{cremona}.
\end{rem} 
	
\begin{prop}\label{10_211_charac}
There are infinitely many real (resp. imaginary) quadratic fields $K$ containing an element $c$ for which $G(f_c,K)$ admits a subgraph of type {\rm 10(2,1,1)a}. 
\end{prop}
\begin{proof}
Let $p(x)=5x^4-8x^3+6x^2+8x+5\in\Q[x]$. Applying Lemma \ref{poly_lemma} to the polynomial $p(x)$ we obtain infinitely many quadratic fields of the form $K_r=\Q(\sqrt{p(r)})$ with $r\in\Q$. For every such field $K_r$ with $r\ne 0$ there is a point $(r,\sqrt{p(r)})\in C(K_r)$ which necessarily satisfies $r(r^2-1)(r^2-4r-1)\ne 0$; hence, by Lemma \ref{10_211a_curve} there is an element $c\in K_r$ such that $f_c$ has a fixed point $a\in K_r$ and a point $b\in K_r$ of type $2_2$. In order to guarantee that $G(f_c,K_r)$ contains a subgraph of type 10(2,1,1)a we need the additional condition that $b\cdot f_c^2(b)\cdot c(4c-1)\ne 0$. Indeed, the condition $b\cdot f_c^2(b)\ne 0$ ensures that $f_c(b)$ and $f_c^3(b)$ each have two distinct preimages, and the condition $c(4c-1)\ne 0$ ensures that $f_c$ has two distinct fixed points, each of which has a preimage different from itself. Now, one can check that
\begin{equation}\label{10_211a_exceptions}
-b^2\cdot f_c^2(b)\cdot c(4c-1)=\frac{(5r^4-8r^3+6r^2+8r+5)(3r^2 + 1)(r^2+4r-1)(r^2+3)(r^2+1)^2}{8(r^2-1)^7},
\end{equation}
so the condition $b\cdot f_c^2(b)\cdot c(4c-1)\ne 0$ is automatically satisfied since $r\in\Q$.

Note that the polynomial function $p:\R\to\R$ induced by $p(x)$ only takes positive values, so that all fields $K_r$ are real quadratic fields. Hence, this argument proves the statement only for real quadratic fields. To prove the statement for imaginary quadratic fields, we first obtain a Weierstrass equation for the elliptic curve birational to $C$. The following are inverse rational maps between $C$ and the elliptic curve $E$ with equation $Y^2=X^3-11X+6:$
\begin{align*}
(x,y) \mapsto & \left(\frac{2y+x^2+2x+5}{(x-1)^2},\frac{2y(x+3)+6x^3 - 6x^2 + 18x + 14}{(x-1)^3}\right), \\
(X,Y)\mapsto & \left(\frac{4Y+(X+1)^2}{X^2-2X-19},\frac{8Y(X^2+10X+9)+4(X^4+2X^3+24X^2+46X-233)}{(X^2-2X-19)^2}\right).
\end{align*}
Let $q(X)=X^3-11X+6\in\Q[X]$. Applying Lemma \ref{poly_lemma} to the polynomial $q(X)$ we obtain infinitely many imaginary quadratic fields of the form $K_R=\Q(\sqrt{q(R)})$ with $R\in\Q$. For every such field there is a point $(R,\sqrt{q(R)})\in E(K_R)$; applying the change of variables above we obtain a point $(r,s)\in C(K_R)$ with
 \begin{equation}\label{10_211a_rR}
 r=\frac{4\sqrt{q(R)}+(R+1)^2}{R^2-2R-19}.
 \end{equation}
In particular, $r$ must satisfy $r(r^2-1)(r^2-4r-1)\ne 0$, since otherwise $r$ would be rational or would generate a real quadratic field. We can now apply Lemma \ref{10_211a_curve} to see that there is an element $c\in K_R$ such that $f_c$ has a fixed point $a\in K_R$ and a point $b\in K_R$ of type $2_2$. Once again, we must check the condition $b\cdot f_c^2(b)\cdot c(4c-1)\ne 0$. From \eqref{10_211a_exceptions} and \eqref{10_211a_rR} it follows that there are only finitely many values of $R\in\Q$ for which we might have $b\cdot f_c^2(b)\cdot c(4c-1)=0$. Hence, for all but finitely many $R$, the above construction yields a graph $G(f_c,K_R)$ containing a subgraph of type 10(2,1,1)a.
\end{proof}


\subsection{Graph 10(2,1,1)b}\label{10_211b_section}
\begin{lem}\label{10_211b_curve} Let $C/\Q$ be the affine curve of genus 1 defined by the equation
 \[y^2 = (5x^2-1)(x^2+3).\]
Consider the rational map $\varphi: C \dashrightarrow \mathbb{A}^3 = \Spec \Q[a,b,c]$ given by 
\[a= \frac{y}{2(x^2 -1)} \;\;,\;\;b=-\frac{x^2-4x-1}{2(x^2-1)}\;\;,\;\;c=-\frac{3x^4+10x^2+3}{4(x^2 - 1)^2}.\]
For every number field $K$, the map $\phi$ induces a bijection from the set $\{(x,y)\in C(K): x(x^2-1)(x^2+3)\ne 0\}$ to the set of all triples $(a,b,c)\in K^3$ such that $a$ is a point of type $1_2$ and $b$ is a point of period 2 for the map $f_c$. 
\end{lem}
\begin{figure}[h!]
\begin{center}\includegraphics[scale=0.5]{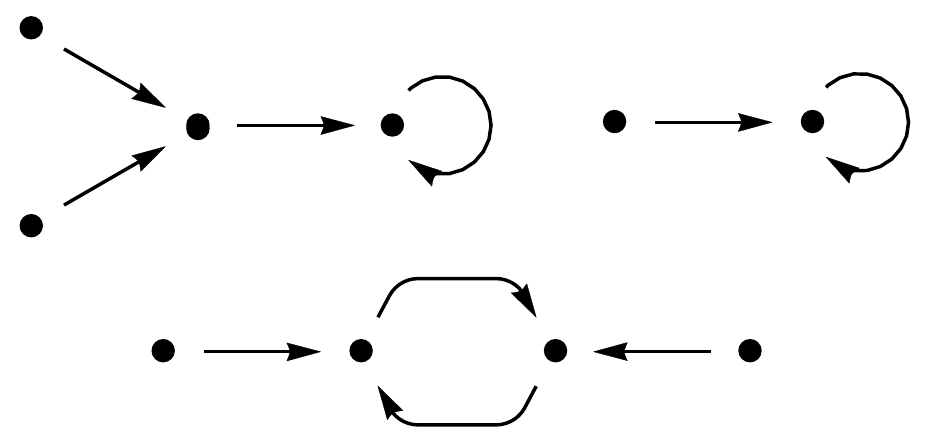}\end{center}
\caption{Graph type 10(2,1,1)b}
\end{figure}
\begin{proof}
Fix a number field $K$ and suppose that $(x,y)\in C(K)$ satisfies $x^2\ne 1$. Defining $a,b,c$ as in the lemma, it is straightforward to verify that $f_c^2(a)$ is fixed by $f_c$; that $b=f_c^2(b)$; and that 
\begin{equation}\label{10_211b_relations}
b-f_c(b)=\frac{4x}{x^2-1}\;\;,\;\;f_c(a)-f_c^2(a)=\frac{x^2+3}{x^2-1}.
\end{equation}
From these relations it follows that if $x(x^2+3)\ne 0$, then $b$ is a point of period 2 and $a$ is of type $1_2$ for $f_c$. Hence, $\phi$ gives a well-defined map.

To see that $\phi$ is surjective, suppose that $a,b,c\in K$ are such that $a$ is a point of type $1_2$ and $b$ is a point of period 2 for the map $f_c$. Then an argument given in \cite[21]{poonen_prep} shows that there exists a point $(x,y) \in C(K)$ with $x(x^2-1)\ne 0$ such that $\phi(x,y)=(a,b,c)$. Furthermore, the relations \eqref{10_211b_relations} imply that we must have $x^2+3\ne 0$.  To see that $\phi$ is injective, one can verify that if $\phi(x,y)=(a,b,c)$, then
\[x=-\frac{1+2f_c^2(b)}{1+2f_c^2(a)}\;\;,\;\;y=2a(x^2-1).\qedhere\]
\end{proof}

\begin{rem} As shown in \cite[21]{poonen_prep}, the curve $C$ is birational over $\Q$ to the elliptic curve 15a8 in Cremona's tables \cite{cremona}.
\end{rem} 
	
\begin{prop}
There are infinitely many real (resp. imaginary) quadratic fields $K$ containing an element $c$ for which $G(f_c,K)$ admits a subgraph of type {\rm 10(2,1,1)b}. 
\end{prop}
\begin{proof}
Let $p(x)=(5x^2-1)(x^2+3)\in\Q[x]$. Applying Lemma \ref{poly_lemma} to the polynomial $p(x)$ we obtain infinitely many real (resp. imaginary) quadratic fields of the form $K_r=\Q(\sqrt{p(r)})$ with $r\in\Q$. For every such field $K_r$ with $r\ne 0$ there is a point $(r,\sqrt{p(r)})\in C(K_r)$ which necessarily satisfies $r(r^2-1)(r^2+3)\ne 0$; hence, by Lemma \ref{10_211b_curve} there is an element $c\in K_r$ such that $f_c$ has a point $a\in K_r$ of type $1_2$ and a point $b\in K_r$ of period 2. In order to conclude that $G(f_c,K_r)$ contains a subgraph of type 10(2,1,1)b we need the additional condition that $ac(4c-1)b\cdot f_c(b)\ne 0$. Indeed, the condition $a\ne 0$ ensures that $f_c(a)$ has two distinct preimages, while the condition $c(4c-1)\ne 0$ guarantees that $f_c$ has two distinct fixed points, each of which has a preimage different from itself; the condition $b\cdot f_c(b)\ne 0$ ensures that $b$ and $f_c(b)$ each have two distinct preimages. Now, one can check that 
\[a^2c(4c-1)b\cdot f_c(b)=\frac{(5r^2-1)(r^2+3)^2(3r^2+1)(r^2+1)^2(r^2-4r-1)(r^2+4r-1)}{16(r^2-1)^8},\]
so the condition $ac(4c-1)b\cdot f_c(b)\ne 0$ is automatically satisfied since $r\in\Q$.
\end{proof}


\subsection{Graph 10(3,1,1)}\label{10_311_section}
\begin{lem}\label{10_311_curve} Let $C/\Q$ be the affine curve of genus 2 defined by the equation
\[y^2 = F_{18}(x):=x^6+2x^5+5x^4+10x^3+10x^2+4x+1.\]
Consider the rational map $\phi: C \dashrightarrow \mathbb{A}^3 = \Spec \Q[a,b,c]$ given by
\[a=\frac{x^3+2x^2+x+1}{2x(x+1)}\;\;,\;\;b=\frac{1}{2} + \frac{y}{2x(x+1)}\;\;,\;\;c=-\frac{ x^6 + 2x^5 + 4x^4 + 8x^3 + 9x^2 + 4x + 1}{4x^2(x+1)^2}.\] 
For every number field $K$, the map $\phi$ induces a bijection from the set $\{(x,y)\in C(K): x(x+1)(x^2+x+1)\neq 0\}$ to the set of all triples $(a,b,c)\in K^3$ such that $a$ and $b$ are points of periods $3$ and $1$, respectively, for the map $f_c$.
\end{lem}
\begin{figure}[h!]
\begin{center}\includegraphics[scale=0.5]{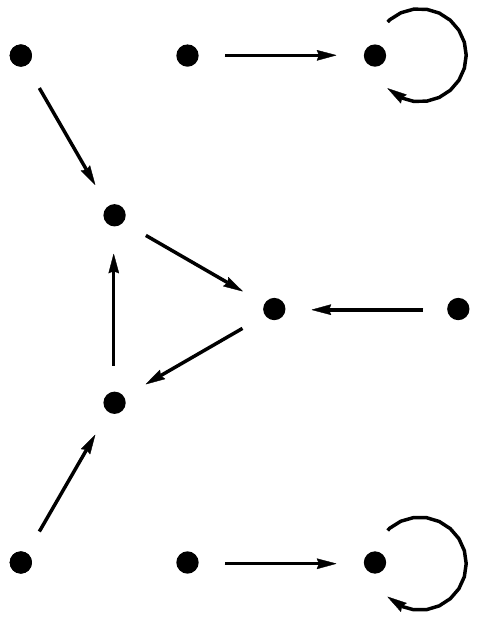}\end{center}
\caption{Graph type 10(3,1,1)}
\end{figure}
\begin{proof} Fix a number field $K$ and suppose that $(x,y)\in C(K)$ satisfies $x(x+1)\ne 0$. Defining $a,b,c$ as in the lemma, it is a routine calculation to verify that $b$ is a fixed point for $f_c$; that $f_c^3(a)=a$; and that
\begin{equation}\label{10_311_relations}
a-f_c(a)=\frac{x^2+x+1}{x(x+1)}.
\end{equation}
It follows from these relations that if $x^2+x+1\ne 0$, then $a$ has period 3 under $f_c$. Hence, $\phi$ gives a well-defined map.

To see that $\phi$ is surjective, suppose that $a,b,c\in K$ are such that $a$ and $b$ are points of periods $3$ and $1$, respectively, for the map $f_c$. Applying Proposition \ref{cycles_prop} we see that there are elements $\rho\in K$ and $x\in K\setminus\{0,-1\}$ such that 
\begin{equation}\label{10_311_bc}
b=1/2+\rho\;\;,\;\; c=1/4-\rho^2
\end{equation}
and 
\begin{equation}\label{10_311_ac}
a=\frac{x^3+2x^2+x+1}{2x(x+1)}\;\;,\;\;c=-\frac{ x^6 + 2x^5 + 4x^4 + 8x^3 + 9x^2 + 4x + 1}{4x^2(x+1)^2}.
\end{equation}
Equating the two expressions for $c$ given in \eqref{10_311_bc} and \eqref{10_311_ac} and letting $y= 2x(x+1)\rho$ we obtain \[y^2 = x^6+2x^5+5x^4+10x^3+10x^2+4x+1.\] Thus, we have a point $(x,y)\in C(K)$ with $x(x+1)\ne 0$ such that $\phi(x,y)=(a,b,c)$. Furthermore, the relations \eqref{10_311_relations} imply that $x^2+x+1\ne 0$.  To see that $\phi$ is injective, one can verify that if $\phi(x,y)=(a,b,c)$, then
\[x=a^2+a+c\;\;,\;\;y=x(x+1)(2b-1).\qedhere\]
\end{proof}

\begin{rem} As noted in \S\ref{modular_genus2_section}, the curve $C$ is birational over $\Q$ to the modular curve $X_1(18)$.
\end{rem}

\begin{thm} Let $K$ be a quadratic field. Suppose that there exists an element $c\in K$ such that $G(f_c,K)$ has a subgraph of type {\rm 10(3,1,1)}. Then there is a rational number $x\notin\{0,-1\}$ such that 
\begin{equation}\label{10_311_c} c=-\frac{ x^6 + 2x^5 + 4x^4 + 8x^3 + 9x^2 + 4x + 1}{4x^2(x+1)^2}
\end{equation} and $K=\Q(\sqrt{F_{18}(x)})$. In particular, $K$ is a real quadratic field.
\end{thm}

\begin{proof} By Lemma \ref{10_311_curve} there is a point $(x,y)\in C(K)$ with $x(x+1)(x^2+x+1)\ne 0$ such that $c$ is given by \eqref{10_311_c}. We claim that $(x,y)$ cannot be a rational point on $C$. Indeed, $C$ is an affine model of the curve $X_1(18)$, which has exactly six rational points (two of them at infinity). Therefore, $C$ has four rational points, and it is easy to see that they are $(0,\pm 1)$ and $(-1,\pm 1)$.  Thus, if $(x,y)\in C(\Q)$, then either $x=0$ or $x=-1$; however, this contradicts the assumption that $x(x+1)\ne 0$. Therefore, $(x,y)$ is a quadratic point on $C$. Since $x^2+x+1\ne 0$,  it follows from Theorem \ref{131618_quad} that $x$ must be a rational number. In particular, since $(x,y)$ is a quadratic point, then $K=\Q(x,y)=\Q(y)=\Q(\sqrt{F_{18}(x)})$. The fact that $K$ is a real quadratic field now follows from the observation that the polynomial function $F_{18}:\R\to\R$ induced by $F_{18}(x)$ only takes positive values.
\end{proof}

\begin{prop}\label{10_311_real} There are infinitely many (real) quadratic fields $K$ containing an element $c$ for which $G(f_c,K)$ admits a subgraph of type {\rm 10(3,1,1)}.
\end{prop}

\begin{proof}
Applying Lemma \ref{poly_lemma} to the polynomial $F_{18}(x)$ we obtain infinitely many (real) quadratic fields of the form $K_r=\Q(\sqrt{F_{18}(r)})$ with $r\in\Q$. For every such field there is a point $(r,\sqrt{F_{18}(r)})\in C(K_r)$ which necessarily satisfies $r(r+1)(r^2+r+1)\ne 0$; hence, by Lemma \ref{10_311_curve} there is an element $c\in K_r$ such that $f_c$ has points $a, b\in K_r$ of periods 3 and 1, respectively. In order to conclude that $G(f_c, K_r)$ contains a subgraph of type 10(3,1,1) we need the additional condition that $a\cdot f_c(a)\cdot f_c^2(a)\cdot c(4c-1)\ne 0$. Indeed, the condition $c(4c-1)\ne 0$ guarantees that $f_c$ has two distinct fixed points, each of which has a preimage different from itself; and the condition $a\cdot f_c(a)\cdot f_c^2(a)\ne 0$ ensures that every point in the orbit of $a$ has two distinct preimages. The expression for $c$ given in Lemma \ref{10_311_curve} implies that $c<0$, so in fact the condition $c(4c-1)\ne 0$ is automatically satisfied. Furthermore, one can verify that
\[-a\cdot f_c(a)\cdot f_c^2(a)=\frac{(r^3+2r^2+r+1)(r^3-r-1)(r^3+2r^2+3r+1)}{8r^3(r+1)^3},\]

so $a\cdot f_c(a)\cdot f_c^2(a)\ne 0$ since $r\in\Q$.
\end{proof}

\begin{rem}
Our search in \S\ref{quad_prep_comp} failed to produce an example of a graph of type 10(3,1,1), though the previous proposition suggests that there are infinitely many such examples. The reason for this is that, even for rational parameters $x$ of moderate size, the discriminant of the field $K=\Q(\sqrt{F_{18}(x)})$ may be large, and the complexity of the rational function defining $c$ forces the height of $c$ to be large as well, thus placing the pair $(K,c)$ outside of our search range. However, we expected this graph to occur infinitely times over quadratic fields, since Poonen \cite[15]{poonen_prep} had showed that the curve parametrizing maps with a fixed point and a point of period 3 is the modular curve $X_1(18)$, which is hyperelliptic and therefore has infinitely many quadratic points. We obtain one instance of the graph by taking $x = 2$ in \eqref{10_311_c}; this leads to the pair $(K,c) = (\Q(\sqrt{337}), -301/144)$. A computation of preperiodic points using the algorithm developed in \S\ref{prep_algorithm} shows that, indeed, the graph $G(f_c,K)$ for this pair $(K,c)$ is of type 10(3,1,1).
\end{rem}


\subsection{Graph 10(3,2)}\label{10_32_section}
\begin{lem}\label{10_32_curve} Let $C/\Q$ be the affine curve of genus 2 defined by the equation
\[y^2 = F_{13}(x) := x^6 + 2x^5 + x^4 + 2x^3 + 6x^2 + 4x +1.\]
Consider the rational map $\phi: C \dashrightarrow \mathbb{A}^3 = \Spec \Q[a,b,c]$ given by
\[a=\frac{x^3+2x^2+x+1}{2x(x+1)}\;\;,\;\;b=-\frac{1}{2} + \frac{y}{2x(x+1)}\;\;,\;\;c=-\frac{ x^6 + 2x^5 + 4x^4 + 8x^3 + 9x^2 + 4x + 1}{4x^2(x+1)^2}.\] 
For every number field $K$, the map $\phi$ induces a bijection from the set \[\{(x,y)\in C(K): xy(x+1)(x^2+x+1)\ne 0\}\] to the set of all triples $(a,b,c)\in K^3$ such that $a$ and $b$ are points of periods $3$ and $2$, respectively, for the map $f_c$.
\end{lem}
\begin{figure}[h!]
\begin{center}\includegraphics[scale=0.5]{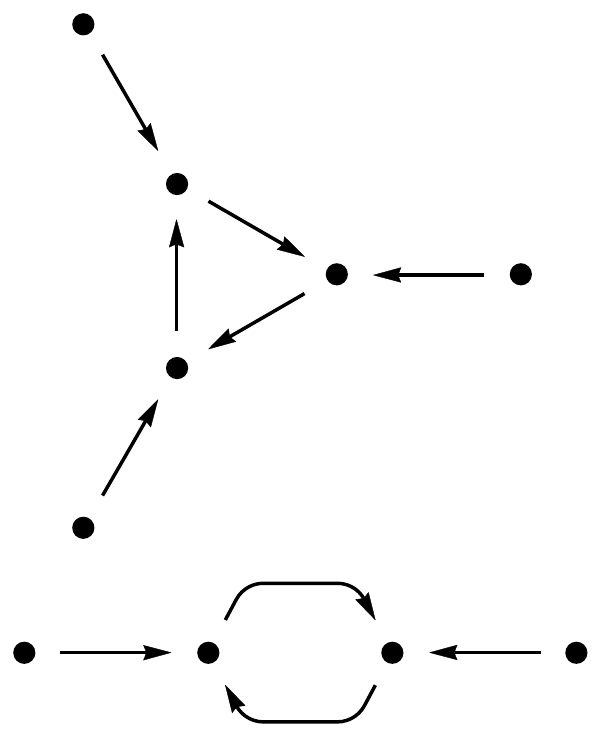}\end{center}
\caption{Graph type 10(3,2)}
\end{figure}
\begin{proof} Fix a number field $K$ and suppose that $(x,y)\in C(K)$ satisfies $x(x+1)\ne 0$. Defining $a,b,c$ as in the lemma, it is straightforward to verify that $f_c^3(a)=a$, $f_c^2(b)=b$, and
\begin{equation}\label{10_32_relations}
a-f_c(a)=\frac{x^2+x+1}{x(x+1)}\;\;,\;\;b-f_c(b)=\frac{y}{x(x+1)}.
\end{equation}
It follows from these relations that if $y(x^2+x+1)\ne 0$, then $a$ has period 3 under $f_c$ and $b$ has period 2. Hence, $\phi$ gives a well-defined map.

To see that $\phi$ is surjective, suppose that $a,b,c\in K$ are such that $a$ and $b$ are points of periods $3$ and $2$, respectively, for the map $f_c$. 
Applying Proposition \ref{cycles_prop} we see that there are elements $\sigma\in K$ and $x\in K\setminus\{0,-1\}$ such that 
 \begin{equation}\label{10_32_bc}
 c=-3/4-\sigma^2\;\;,\;\; b=-1/2+\sigma
 \end{equation}
and
\begin{equation}\label{10_32_ac} a=\frac{x^3+2x^2+x+1}{2x(x+1)}\;\;,\;\;c=-\frac{ x^6 + 2x^5 + 4x^4 + 8x^3 + 9x^2 + 4x + 1}{4x^2(x+1)^2}.
\end{equation}
Equating the two expressions for $c$ given in \eqref{10_32_bc} and \eqref{10_32_ac} and letting $y= 2x(x+1)\sigma$ we obtain \[y^2 = x^6 + 2x^5 + x^4 + 2x^3 + 6x^2 + 4x +1.\] Thus, we have a point $(x,y)\in C(K)$ with $x(x+1)\ne 0$ such that $\phi(x,y)=(a,b,c)$. Furthermore, the relations \eqref{10_32_relations} imply that $y(x^2+x+1)\ne 0$.  To see that $\phi$ is injective, one can verify that if $\phi(x,y)=(a,b,c)$, then
\[x=a^2+a+c\;\;,\;\;y=x(x+1)(2b+1).\qedhere\]
\end{proof}

\begin{rem} As noted in \S\ref{modular_genus2_section}, the curve $C$ is birational over $\Q$ to the modular curve $X_1(13)$.
\end{rem}

\begin{thm}\label{10_32_charac} Let $K$ be a quadratic field. Suppose that there exists an element $c\in K$ such that $G(f_c,K)$ has a subgraph of type {\rm 10(3,2)}. Then there is a rational number $x\notin\{0,-1\}$ such that 
\begin{equation}\label{10_32_c}c=-\frac{ x^6 + 2x^5 + 4x^4 + 8x^3 + 9x^2 + 4x + 1}{4x^2(x+1)^2}
\end{equation} and $K=\Q(\sqrt{F_{13}(x)})$. In particular, $K$ is a real quadratic field.
\end{thm}

\begin{proof} By Lemma \ref{10_32_curve} there is a point $(x,y)\in C(K)$ with $xy(x+1)(x^2+x+1)\ne 0$ such that $c$ is given by \eqref{10_32_c}. We claim that $(x,y)$ cannot be a rational point on $C$. Indeed, $C$ is an affine model of the curve $X_1(13)$, which has exactly six rational points (two of them at infinity). Therefore, $C$ has four rational points, and it is easy to see that they are $(0,\pm 1)$ and $(-1,\pm 1)$.  Thus, if $(x,y)\in C(\Q)$, then either $x=0$ or $x=-1$; however, this contradicts the assumption that $x(x+1)\ne 0$. Therefore, $(x,y)$ is a quadratic point on $C$, so it follows from Theorem \ref{131618_quad} that $x$ must be a rational number. In particular, since $(x,y)$ is a quadratic point, then $K=\Q(x,y)=\Q(y)=\Q(\sqrt{F_{13}(x)})$. The fact that $K$ is a real quadratic field now follows from the observation that the polynomial function $F_{13}:\R\to\R$ induced by $F_{13}(x)$ only takes positive values.
\end{proof}

\begin{prop}\label{10_32_real} There are infinitely many (real) quadratic fields $K$ containing an element $c$ for which $G(f_c,K)$ admits a subgraph of type {\rm 10(3,2)}.
\end{prop}

\begin{proof}
Applying Lemma \ref{poly_lemma} to the polynomial $F_{13}(x)$ we obtain infinitely many (real) quadratic fields of the form $K_r=\Q(\sqrt{F_{13}(r)})$ with $r\in\Q$. For every such field there is a point $(r,\sqrt{F_{13}(r)})\in C(K_r)$ which necessarily satisfies $r(r+1)(r^2+r+1)F_{13}(r)\ne 0$; hence, by Lemma \ref{10_32_curve} there is an element $c\in K_r$ such that $f_c$ has points $a, b\in K_r$ of periods 3 and 2, respectively. In order to conclude that $G(f_c, K_r)$ contains a subgraph of type 10(3,2) we need the additional condition that all points in the orbits of $a$ and $b$ are nonzero, as this will ensure that every such point has two distinct preimages. One can check that
\[-a\cdot f_c(a)\cdot f_c^2(a)=\frac{(r^3+2r^2+r+1)(r^3-r-1)(r^3+2r^2+3r+1)}{8r^3(r+1)^3},\]
so all points in the orbit of $a$ are nonzero since $r\in\Q$. Now, using Lemma \ref{10_32_curve} we see that if $b\cdot f_c(b)=0$, then
\[r^6 + 2r^5 + r^4 + 2r^3 + 6r^2 + 4r +1=r^2(r+1)^2.\]

However, one can verify that this equation has no rational solution. Therefore, both points in the orbit of $b$ must be nonzero.
\end{proof}


\subsection{Graph 12(2)}\label{12_2_section}
Our search described in \S\ref{quad_prep_comp} produced the pairs \[(K,c)=\left(\Q(\sqrt{2}), -\frac{15}{8}\right)\;\;\text{and}\;\; (K,c)=\left(\Q(\sqrt{57}),-\frac{55}{48}\right)\] for which the graph $G(f_c,K)$ is of type 12(2). We will show here that, in addition to these known pairs, there are at most four other pairs $(K,c)$ giving rise to this graph structure.
\begin{lem}\label{12_2_curve} Let $C/\Q$ be the affine curve of genus 5 defined by the equations
\begin{equation}\label{12_2_curve_equations}
\begin{cases}
& y^2 =  2(x^4+2x^3-2x+1)\\
& z^2 = 2(x^3+x^2-x+1).
\end{cases}
\end{equation}
Consider the rational map $\phi: C \dashrightarrow \mathbb{A}^3 = \Spec \Q[a,s,c]$ given by
\begin{equation*} a=\frac{z}{x^2-1}\;,\;s=\frac{y}{x^2-1}\;,\;c=-\frac{x^4+2x^3+2x^2-2x+1}{(x^2-1)^2}\;. \end{equation*}
For every number field $K$, the map $\phi$ induces a bijection from the set \[\{(x,y,z)\in C(K):x(x^2-1)(x^2+4x-1)(x^2 + 2x - 1)\ne 0\}\] to the set of all triples $(a,s,c)\in K^3$ such that $a$ is a point of type $2_3$ for the map $f_c$ and $s$ is a point of type $2_2$ satisfying $f_c^2(s)\ne f_c^3(a)$.
\end{lem}
\begin{figure}[h!]
\begin{center}\includegraphics[scale=0.5]{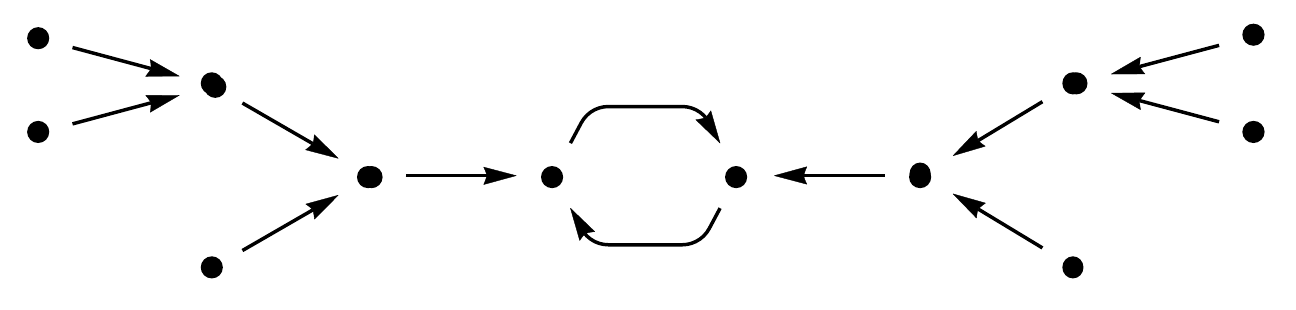}\end{center}
\caption{Graph type 12(2)}
\end{figure}
\begin{proof} Fix a number field $K$ and suppose that $(x,y,z)\in C(K)$ is a point with $x^2\ne 1$. Defining $a,s,c\in K$ as in the lemma, it is a straightforward calculation to verify that $f_c^3(a)=f_c^5(a)$, $f^4_c(a)=f_c^2(s)$, $f_c(s)-f_c^2(s)=1$, and
\begin{equation}\label{12_2_relations}
f_c^4(a)-f_c^3(a)=\frac{x^2+4x-1}{x^2-1}\;\;,\;\;f_c^4(a)-f_c^2(a)=\frac{4x}{x^2-1}\;\;,\;\;f_c(s)-f_c^3(s)=\frac{2(x^2+2x-1)}{x^2-1}.
\end{equation}
It follows from these relations that if $x(x^2+4x-1)(x^2 + 2x - 1)\ne 0$, then $a$ is a point of type $2_3$ and $s$ is of type $2_2$ for $f_c$, and moreover that $f_c^2(s)\ne f_c^3(a)$. Hence, $\phi$ gives a well-defined map.

To see that $\phi$ is surjective, suppose that $a,s,c\in K$ are such that $a$ is a point of type $2_3$ for the map $f_c$ and $s$ is a point of type $2_2$ satisfying $f_c^2(s)\ne f_c^3(a)$. Let $r=f_c(a)$, so that $r$ is of type $2_2$. Since $f_c$ can have only one 2-cycle, the points $f_c^2(r)$ and $f_c^2(s)$ must form a 2-cycle. Then the argument given in \cite[20]{poonen_prep} shows that there is an element $x\in K\setminus\{0, \pm 1\}$ such that

\begin{equation}\label{12_2_crsvalues} c=-\frac{x^4+2x^3+2x^2-2x+1}{(x^2-1)^2}\;,\; r=-\frac{x^2+1}{x^2-1}\;,\;s^2=\frac{2(x^4+2x^3-2x+1)}{(x^2-1)^2}\;.
\end{equation}
Since $a^2=r-c$, this implies that \[a^2=\frac{2(x^3+x^2-x+1)}{(x^2-1)^2}.\]
Letting $y=s(x^2-1)$ and $z=a(x^2-1)$ we obtain a point $(x,y,z)\in C(K)$ with $x(x^2-1)\ne 0$ and $\phi(x,y,z)=(a,s,c)$. Furthermore, the relations \eqref{12_2_relations} imply that $(x^2+4x-1)(x^2+2x-1)\ne 0$.  To see that $\phi$ is injective, one can verify that if $\phi(x,y,z)=(a,s,c)$, then
\[x=\frac{f_c(a)-1}{f_c^2(a)}\;\;,\;\;y=s(x^2-1)\;\;,\;\;z=a(x^2-1).\qedhere\]
\end{proof}

\begin{lem}\label{12_2_aux_curve_pts} Let $X/\Q$ be the hyperelliptic curve of genus 3 defined by the equation
\[w^2=x^7 + 3x^6 + x^5 - 3x^4 + x^3 + 3x^2 - 3x + 1.\] 
Then $X(\Q)$ contains the points $\infty, (-1, \pm 2), (0,\pm 1), (1,\pm 2)$ and at most 8 other points.
\end{lem}

\begin{proof} Using the Magma function \texttt{Points} we search for rational points on $X$ of height at most $10^5$ and obtain the points listed above. Using the \texttt{RankBounds} function we find that the group $\Jac(X)(\Q)$ has rank 2; we may therefore apply the method of Chabauty and Coleman to bound the number of rational points on $X$. The prime 3 is of good reduction for $X$, and Magma's \texttt{Points} function yields $\#X(\F_{3})=7.$ Applying the Lorenzini-Tucker bound (Theorem \ref{lorenzini_tucker_bound}) with $p=3$, $d=2$ we obtain $\#X(\Q) \le 15$. Since we have already listed 7 rational points, we conclude that there are at most 8 additional rational points on $X$.
\end{proof}

\begin{thm}\label{12_2_points} With $C$ as in Lemma \ref{12_2_curve} we have the following:
\begin{enumerate}
\item $C(\Q)=\{(\pm 1,\pm 2,\pm 2)\}$.
\item If $(x,y,z)$ is a quadratic point on $C$ with $x\in\Q$, then there exists $w\in\Q$ such that $(x,w)\in X(\Q)$, where $X$ is the curve defined in Lemma \ref{12_2_aux_curve_pts}. Moreover, $x\ne\pm 1$.
\item The quadratic points $(x,y,z)$ on $C$ with $x\notin\Q$ are the points $(x,\pm(4x+10),\pm(2x+4))$ with $x^2-2x-7=0$, which are defined over the field $\Q(\sqrt 2)$, and the points $(x,\pm(24x-10),\pm(6x-4))$ with $x^2-16x+7=0$, which are defined over the field $\Q(\sqrt{57})$.
\end{enumerate}
\end{thm}

\begin{proof} Let $C_1$ be the hyperelliptic curve of genus 1 defined by equation $y^2 = 2(x^4+2x^3-2x+1)$. Since $C_1$ has a rational point --- for instance, the point (1,2) --- it is an elliptic curve; in fact, as mentioned in \S\ref{8_2a_section}, $C_1$ is the elliptic curve with Cremona label 40a3, which has exactly four rational points. Hence, $\#C_1(\Q)=4$, and we easily find the four points: $C_1(\Q)=\{(\pm 1,\pm 2)\}$. Similarly, let $C_2$ be the hyperelliptic curve defined by the equation $z^2 = 2(x^3+x^2-x+1)$. Then $C_2$ is the elliptic curve 11a3, which has exactly five rational points. Hence, $C_2(\Q)$ consists of the point at infinity and four other points, which are easily found to be the points $(\pm 1, \pm 2)$. Knowing $C_1(\Q)$ and $C_2(\Q)$ we conclude, in particular, that $C(\Q)=\{(\pm 1,\pm 2,\pm 2)\}$.

Suppose now that $(x,y,z)\in C(\qbar)$ satisfies $[\Q(x,y,z):\Q]=2$, and let $K=\Q(x,y,z)$.

{\bf Case 1:} $x\in\Q$. We cannot have $x=\pm 1$, since this would imply that $y=\pm 2$ and $z=\pm 2$, contradicting the assumption that $(x,y,z)$ is a quadratic point on $C$. It follows that $y\notin \Q$, since having $x,y\in\Q$ would imply that $x=\pm 1$. By the same argument, $z\notin\Q$. Since both $y$ and $z$ are quadratic, then $K=\Q(y)=\Q(z)$, so the numbers $2(x^4+2x^3-2x+1)$ and $2(x^3 + x^2 - x + 1)$ must have the same squarefree part; hence their product is a square, so there is a rational number $w$ such that 
\[w^2= x^7 + 3x^6 + x^5 - 3x^4 + x^3 + 3x^2 - 3x + 1.\] This proves part (2) of the theorem.

{\bf Case 2:} $x$ is quadratic. By Lemma \ref{ellcrv_quad_pts} applied to the equation $z^2 = 2(x^3 + x^2 - x + 1)$, there exist a rational number $w$ and a point $(x_0,z_0)\in\{(\pm 1,\pm 2)\}$ such that 
\begin{equation}\label{12_2_xpoly1} x^2 + \frac{2x_0-w^2+2}{2}x + \frac{2x_0^2+w^2x_0+2x_0-2z_0w-2}{2}= 0.
\end{equation}

We will now deduce a different expression for the minimal polynomial of $x$. Making the change of variables 

\begin{equation}\label{12_2_chv}
X = \frac{2x^2 + y}{(x - 1)^2}\;,\;Y = \frac{3x^3 + 3x^2 + 2xy - 3x + 1}{(x - 1)^3}
\end{equation}
satisfying
\begin{equation}\label{12_2_X} x=\frac{X^2 + 2Y}{X^2 - 4X + 2},
\end{equation}
the system \eqref{12_2_curve_equations} becomes
\[
\begin{cases}
& Y^2 = X^3-2X+1\\
 & z^2 = 2(x^3 + x^2 - x + 1).
 \end{cases}
\]

Suppose that $X\in\Q$. Substituting $y=X(x-1)^2-2x^2$, the equation $y^2 - 2(x^4 + 2x^3 - 2x + 1)=0$ becomes
\[(x-1)^2\left((X^2 - 4X + 2)x^2 - 2X^2x + X^2 - 2\right)=0,\] so we must have
\begin{equation}\label{12_2_xpoly2}
x^2 - \frac{2X^2}{X^2 - 4X + 2}x + \frac{X^2 - 2}{X^2 - 4X + 2} = 0.
\end{equation}

Comparing \eqref{12_2_xpoly1} and \eqref{12_2_xpoly2} we arrive at the system
\[
\begin{cases}
& (2x_0-w^2+2)(X^2-4X+2)=-4X^2\\
 & (2x_0^2+w^2x_0+2x_0-2z_0w-2)(X^2-4X+2) = 2(X^2-2).
 \end{cases}
\] 

For every choice of point $(x_0,z_0)\in\{(\pm 1,\pm 2)\}$, the above system defines a 0-dimensional scheme in the $(X,w)$ plane over $\Q$, whose rational points we compute using the Magma function \texttt{RationalPoints}. For every solution $(X,w)$ we then check whether the polynomial \eqref{12_2_xpoly2} is irreducible. With $(x_0,z_0) =  (-1, 2)$ and $(X,w) =  (1/2, 2)$ we obtain $x^2 - 2x - 7=0$; with $(x_0,z_0) =  (1, 2)$ and $(X,w) =  (4, 6)$ we obtain $x^2 - 16x + 7=0$. All other solutions lead to one of these two equations.

Suppose now that $X\notin\Q$. By Lemma \ref{ellcrv_quad_pts} applied to the equation $Y^2 = X^3 - 2X + 1$, there exist a rational number $v$ and a point $(X_0,Y_0)\in\{(0,\pm 1), (1,0)\}$ such that
\begin{equation}\label{12_2_Xpoly} X^2+(X_0-v^2)X + X_0^2+v^2X_0-2Y_0v-2= 0\;\;\mathrm{and}\;\; Y=Y_0+v(X-X_0).\end{equation}

The point $(0,-1)$ is excluded in this case because \eqref{12_2_X} would imply that $x=\frac{v}{v-2}\in\Q$. If $(X_0,Y_0)=(0,1)$, then \eqref{12_2_X} and \eqref{12_2_Xpoly} imply that 
\begin{equation}\label{12_2_xpoly_01}x^2-\frac{4(v+1)}{v^2-2}x-1=0,\end{equation} and if $(X_0,Y_0)=(1,0)$, then \begin{equation}\label{12_2_xpoly_10}x^2 - \frac{4x}{v^2 - 2v - 1} - \frac{v^2 + 2v - 1}{v^2 - 2v - 1}=0.\end{equation}

Suppose that $(X_0,Y_0)=(0,1)$. Comparing \eqref{12_2_xpoly1} and \eqref{12_2_xpoly_01} we obtain the system
\[
\begin{cases}
& (2x_0-w^2+2)(v^2-2)=-8(v+1)\\
 & 2x_0^2+w^2x_0+2x_0-2z_0w= 0.
 \end{cases}
\]

For each choice of point $(x_0,z_0)\in\{(\pm 1,\pm 2)\}$ we solve the above system for $v$ and $w$, and find that $v=-1$.  But then \eqref{12_2_xpoly_01} implies that $x=\pm 1$, which is a contradiction.

Now suppose that $(X_0,Y_0)=(1,0)$. Comparing \eqref{12_2_xpoly1} and \eqref{12_2_xpoly_10} we obtain the system
\[
\begin{cases}
& (2x_0-w^2+2)(v^2-2v-1)=-8\\
 & (2x_0^2+w^2x_0+2x_0-2z_0w-2)(v^2-2v-1) = -2(v^2+2v-1).
 \end{cases}
\]

For each choice of point $(x_0,z_0)\in\{(\pm 1,\pm 2)\}$ we consider the above system and find that it has no rational solutions. 

The assumption that $X\notin\Q$ has led in every case to a contradiction, so we conclude that $X$ must be rational. Hence, the analysis done earlier shows that either $x^2 - 2x - 7=0$ or $x^2 - 16x + 7=0$. Assuming $x^2 - 2x - 7=0$, the system \eqref{12_2_curve_equations} can be solved to obtain $y=\pm(4x+10)$, $z=\pm(2x+4)$. Thus, we have found the points $(x,\pm(4x+10),\pm(2x+4))\in C(K)$, where $K=\Q(x)=\Q(\sqrt{2})$. Now assume that $x^2 - 16x + 7=0$. Then we solve the system \eqref{12_2_curve_equations} and find that $y=\pm(24x-10)$, $z=\pm(6x-4)$. Thus, we have found the points $(x,\pm(24x-10),\pm(6x-4))\in C(K)$, where $K=\Q(x)=\Q(\sqrt{57})$. This proves part (3) of the theorem.
\end{proof}

\begin{cor} In addition to the known pairs $(\Q(\sqrt{2}), -15/8)$ and $(\Q(\sqrt{57}),-55/48)$ there are at most four pairs $(K,c)$, with $K$ a quadratic number field and $c\in K$, for which $G(f_c,K)$ contains a graph of type {\rm 12(2)}. Moreover, for every such pair we must have $c\in\Q$.
\end{cor}

\begin{proof} Suppose that $(K,c)$ is such a pair. By Lemma \ref{12_2_curve} there is a point $(x,y,z)\in C(K)$ with $x(x^2-1)\ne 0$ such that 
\begin{equation}\label{12_2_c_param}
c=-\frac{x^4+2x^3+2x^2-2x+1}{(x^2-1)^2}.
\end{equation} 
We cannot have $(x,y,z)\in C(\Q)$ since this would imply, by Theorem \ref{12_2_points} part (1), that $x\in\{\pm 1\}$. Hence, $(x,y,z)$ is a quadratic point on $C$. Suppose that $x$ is quadratic. Then Theorem \ref{12_2_points} implies that either $x^2-2x-7=0$ or $x^2-16x+7=0$. In the former case we have $K=\Q(x)=\Q(\sqrt 2)$, and \eqref{12_2_c_param} implies that $c=-15/8$; in the latter case have $K=\Q(x)=\Q(\sqrt {57})$ and \eqref{12_2_c_param} implies that $c=-55/48$. Thus, we recover the two pairs listed in the corollary. Note also that $c\in\Q$ in this case.

Consider now the case where $x\in\Q$ (and hence $c\in\Q$, by \eqref{12_2_c_param}). By Theorem \ref{12_2_points} there is a rational number $w$ such that $(x,w)\in X(\Q)$. Since $x\notin\{0,\pm 1\}$, then Lemma \ref{12_2_aux_curve_pts} implies that there are at most four options for $x$. Each value of $x$ determines the number $c$ by \eqref{12_2_c_param} and the field $K$ by \eqref{12_2_curve_equations}. This gives at most four options for the pair $(K,c)$.  
\end{proof}


\subsection{Graph 12(2,1,1)a}\label{12_211a_section}
Our search described in \S\ref{quad_prep_comp} produced the pair \[(K,c)=\left(\Q(\sqrt{17}), -\frac{13}{16}\right)\] for which the graph $G(f_c,K)$ is of type 12(2,1,1)a. We will show here that this is the only such pair $(K,c)$ with $K$ a quadratic number field and $c\in K$.

\begin{lem}\label{12_211a_curve} Let $C/\Q$ be the affine curve of genus 5 defined by the equations
\begin{equation}\label{12_211a_curve_equations}
\begin{cases}
& y^2 =  2(x^4+2x^3-2x+1)\\
& z^2 = 5x^4+8x^3+6x^2-8x+5.
\end{cases}
\end{equation}
Consider the rational map $\phi: C \dashrightarrow \mathbb{A}^4 = \Spec \Q[r,s,p,c]$ given by
\begin{equation*} r=-\frac{x^2+1}{x^2-1}\;,\;s=\frac{y}{x^2-1}\;,\;p=\frac{1}{2}+\frac{z}{2(x^2-1)}\;,\;c=-\frac{x^4+2x^3+2x^2-2x+1}{(x^2-1)^2}\;. \end{equation*}
For every number field $K$, the map $\phi$ induces a bijection from the set \[\{(x,y,z)\in C(K):x(x^2-1)(x^2+4x-1)(x^2+2x-1)\ne 0\}\] to the set of all tuples $(r,s,p,c)\in K^4$ such that $p$ is a fixed point of the map $f_c$ and $r,s$ are points of type $2_2$ for $f_c$ satisfying $f_c^2(r)\ne f_c^2(s)$.
\end{lem}
\begin{figure}[h!]
\begin{center}\includegraphics[scale=0.5]{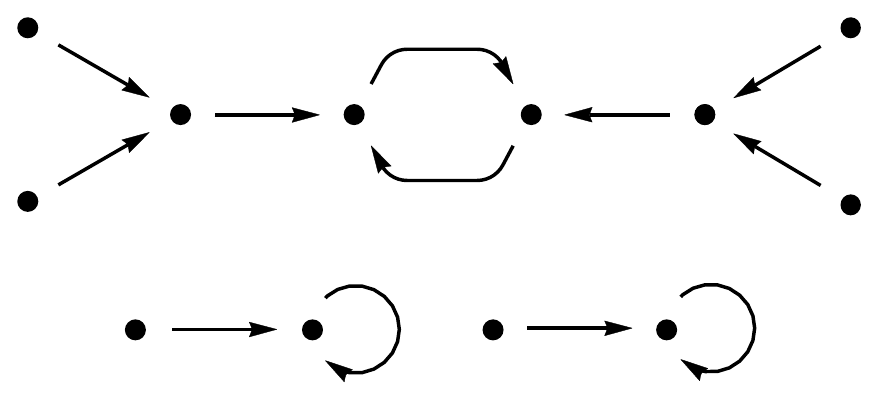}\end{center}
\caption{Graph type 12(2,1,1)a}
\end{figure}
\begin{proof} Fix a number field $K$ and suppose that $(x,y,z)\in C(K)$ is a point satisfying $x^2\ne 1$. Defining $r,s,p,c\in K$ as in the lemma, it is a simple calculation to verify that $p$ is a fixed point of the map $f_c$; moreover, Lemma \ref{8_2a_curve}  implies that if $x(x^2+4x-1)(x^2+2x-1)\ne 0$, then $r$ and $s$ are points of type $2_2$ for $f_c$ with $f_c^2(r)\ne f_c^2(s)$. Hence, $\phi$ gives a well-defined map.

To see that $\phi$ is surjective, suppose that $r,s,p,c\in K$ are such that $p$ is a fixed point of the map $f_c$ and $r,s$ are points of type $2_2$ for $f_c$ satisfying $f_c^2(r)\ne f_c^2(s)$. By Lemma \ref{8_2a_curve} there are elements $x,y\in K$ satisfying $y^2 =  2(x^4+2x^3-2x+1)$ such that $x(x^2-1)(x^2+4x-1)(x^2+2x-1)\ne 0$ and
\[r= -\frac{x^2+1}{x^2-1} \;\;,\;\; s= \frac{y}{x^2-1} \;\;,\;\; c=- \frac{x^4 + 2x^3 + 2x^2 -2x + 1}{(x^2 - 1)^2}.\]

Now, since $p$ is a fixed point for $f_c$, Proposition \ref{cycles_prop} implies that there is an element $\rho\in K$ such that
\[c=1/4-\rho^2 \;\;\mathrm{and}\;\; p=1/2 +\rho.\]
Equating the two expressions for $c$ given above and letting $z=2\rho(x^2-1)$, we obtain the relation $z^2 = 5x^4+8x^3+6x^2-8x+5$. Thus, we have a point $(x,y,z)\in C(K)$ with $x(x^2-1)(x^2+4x-1)(x^2+2x-1)\ne 0$ and $\phi(x,y,z)=(r,s,p,c)$.  To see that $\phi$ is injective, one can verify that if $\phi(x,y,z)=(r,s,p,c)$, then
\[x=\frac{r-1}{r^2+c}\;\;,\;\;y=s(x^2-1)\;\;,\;\;z=(2p-1)(x^2-1).\qedhere\]
\end{proof}

\begin{thm}\label{12_211a_points} With $C$ as in Lemma \ref{12_211a_curve} we have the following:
\begin{enumerate}
\item $C(\Q)=\{(\pm 1,\pm 2,\pm 4)\}$.
\item If $K$ is a quadratic field different from $\Q(\sqrt{5})$ and $\Q(\sqrt{17})$, then $C(K)=C(\Q)$.
\item For $K=\Q(\sqrt 5)$, $C(K)\setminus C(\Q)=\{(x,\pm(4x-2),\pm 8x): x^2+4x-1=0\}$.
\item For $K=\Q(\sqrt{17})$, $C(K)\setminus C(\Q)$ consists of the points $(-3,\pm 2\sqrt{17},\pm 4\sqrt{17}), (1/3,\pm 2\sqrt{17}/9,\pm 4\sqrt{17}/9),$ and the points $(x,\pm 10x,\pm(16x-4))$ with $x^2+8x-1=0$.
\end{enumerate}
\end{thm}

\begin{proof} As shown in the proof of Theorem \ref{12_2_points}, the only rational solutions to the equation $y^2 = 2(x^4+2x^3-2x+1)$ are $(x,y)=(\pm 1,\pm 2)$. Let $C_2$ be the hyperelliptic curve of genus 1 defined by the equation  $z^2 = 5x^4+8x^3+6x^2-8x+5$. Making the change of variables $x\mapsto -x$, it follows from the proof of Proposition \ref{10_211_charac} that $C$ is isomorphic to the elliptic curve 17a4 in Cremona's tables, which has exactly four rational points; hence, $C_2$ has four rational points, and these are easily found to be $(\pm 1, \pm 4)$. Knowing all rational solutions to each of the defining equations for $C$, we deduce that $C(\Q)=\{(\pm 1,\pm 2,\pm 4)\}$.

Suppose now that $(x,y,z)\in C(\qbar)$ satisfies $[\Q(x,y,z):\Q]=2$, and let $K=\Q(x,y,z)$.

{\bf Case 1:} $x\in\Q$. We cannot have $x=\pm 1$, since this would imply that $y=\pm 2$ and $z=\pm 4$, contradicting the assumption that $(x,y,z)$ is a quadratic point on $C$. It follows that $y\notin \Q$, since having $x,y\in\Q$ would imply that $x=\pm 1$. By a similar argument, $z\notin\Q$. Since $y$ and $z$ are both quadratic, we must have $K=\Q(y)=\Q(z)$, so the numbers $2(x^4+2x^3-2x+1)$ and $5x^4+8x^3+6x^2-8x+5$ must have the same squarefree part, and hence their product is a square; thus, there is a rational number $w$ such that
\begin{equation}\label{12_211a_hyper} w^2=10x^8 + 36x^7 + 44x^6 - 12x^5 - 44x^4 + 12x^3 + 44x^2 - 36x + 10.
\end{equation} 
Let $X$ be the hyperelliptic curve of genus 3 defined by \eqref{12_211a_hyper}. Searching for rational points on $X$ using Magma's \texttt{Points} function we obtain the 8 points $(\pm1, \pm 8), (-3,\pm 136), (1/3,\pm 136/81)$; we will show that these are all the rational points on $X$. To see this, note that $X$ has an involution given by $(x,w)\mapsto (-1/x,w/x^4)$. The quotient of $X$ by this involution is the elliptic curve \[E :\; v^2 + uv + v = u^3 - u^2 - 6u - 4.\] The quotient map $X\to E$ of degree 2 is given by 
\begin{align*}
u &= \frac{2x^4+6x^3+4x^2-6x-w+2}{(x^2-1)^2},\\
v &= \frac{-6x^6 - 24x^5 - 22x^4 + 16x^3 + 2x^2w + 22x^2 + 4xw - 24x - 2w + 6}{x^6 - 3x^4 + 3x^2 - 1}.
\end{align*}
The curve $E$ is the elliptic curve 17a2 in Cremona's tables, and has exactly four rational points. It follows that $X$ can have at most 8 rational points, and since we have already listed 8 points in $X(\Q)$, these must be all.

Returning to \eqref{12_211a_hyper}, we have a point $(x,w)\in X(\Q)$ with $x\ne\pm 1$; we must then have $x=-3$ or $x=1/3$. Taking $x=-3$ we obtain by \eqref{12_211a_curve_equations} that $y^2=68$ and $z^2=272$; if $x=1/3$, then $y^2=68/81$ and $z^2=272/81$. Thus, we have shown that the only quadratic points on $C$ having a rational $x$-coordinate are the points $(-3,\pm 2\sqrt{17},\pm 4\sqrt{17})$ and $(1/3,\pm 2\sqrt{17}/9,\pm 4\sqrt{17}/9)$.

{\bf Case 2:} $x$ is quadratic. Starting with the system \eqref{12_211a_curve_equations}, we make the change of variables 

\begin{align}\label{12_11a_chv}
X &= \frac{2x^2 + y}{(x - 1)^2}\;,\;Y = \frac{3x^3 + 3x^2 + 2xy - 3x + 1}{(x - 1)^3}\\
S &= \frac{5x^2 + 2x + 2z + 1}{(x - 1)^2}\;,\; T = \frac{2(7x^3 + 9x^2 + 3xz - 3x + z + 3)}{(x - 1)^3} 
\end{align}

satisfying

\begin{equation}\label{12_211a_XS}\frac{X^2 + 2Y}{X^2 - 4X + 2} = x = \frac{S^2 + 2S + 4T + 1}{S^2 - 10S + 5} \end{equation}

to obtain the equations
\[
\begin{cases}
& Y^2 = X^3 - 2X + 1\\
 &T^2 = S^3 - 11S + 6.
 \end{cases}
\]

The idea of the proof is to use these two equations together with Lemma \ref{ellcrv_quad_pts} to find the minimal polynomials of $X$ and $S$, and then using \eqref{12_211a_XS} to find the minimal polynomial of $x$. It may occur that $X$ or $S$ are rational rather than quadratic, so we must consider this possibility. If $X\in\Q$, then using the relation $y=X(x-1)^2-2x^2$, the equation
\[y^2 -  2(x^4+2x^3-2x+1)=0\] becomes
\[(x-1)^2(X^2-4X+2)\left(x^2 - \frac{2X^2}{X^2 - 4X + 2}x + \frac{X^2 - 2}{X^2 - 4X + 2}\right)=0.\]

Since $x\notin\Q$ and the equation $X^2-4X+2=0$ has no rational solution, this implies that 
\begin{equation}\label{12_211a_xpoly1}x^2 - \frac{2X^2}{X^2 - 4X + 2}x + \frac{X^2 - 2}{X^2 - 4X + 2} = 0.
\end{equation}
Similarly, if $S\in\Q$, then using the relation $z=\frac{1}{2}[S(x-1)^2-5x^2-2x-1]$, the equation
\[z^2 -(5x^4+8x^3+6x^2-8x+5)=0\] becomes
\[\frac{1}{4}(x-1)^2(S^2-10S+5)\left(x^2 - \frac{2S^2 + 4S + 2}{S^2 - 10S + 5}x + \frac{S^2 - 2S - 19}{S^2 - 10S + 5}\right) = 0.\]
Since $x\notin\Q$ and the equation $S^2-10S+5=0$ has no rational solution, this implies that
\begin{equation}\label{12_211a_xpoly2}x^2 - \frac{2S^2 + 4S + 2}{S^2 - 10S + 5}x + \frac{S^2 - 2S - 19}{S^2 - 10S + 5} = 0.\end{equation} 

Now, if $X$ is quadratic, then by Lemma \ref{ellcrv_quad_pts} applied to the equation $Y^2 = X^3 - 2X + 1$, there is a rational number $v$ and a point $(X_0,Y_0)\in\{(0,\pm 1), (1,0)\}$ such that
\begin{equation}\label{12_211a_Xpoly} X^2+(X_0-v^2)X + X_0^2+v^2X_0-2Y_0v-2= 0\;,\; Y=Y_0+v(X-X_0).\end{equation}

The point $(0,-1)$ is excluded in this case because \eqref{12_211a_XS} would imply that $x=\frac{v}{v-2}\in\Q$. If $(X_0,Y_0)=(0,1)$, then \eqref{12_211a_XS} and \eqref{12_211a_Xpoly} imply that 
\begin{equation}\label{12_211a_xpoly_01}x^2-\frac{4(v+1)}{v^2-2}x-1=0.\end{equation} If instead  $(X_0,Y_0)=(1,0)$, then \begin{equation}\label{12_211a_xpoly_10}x^2 - \frac{4x}{v^2 - 2v - 1} - \frac{v^2 + 2v - 1}{v^2 - 2v - 1}=0.\end{equation}   
 
Similarly, if $S\notin\Q$, then by Lemma \ref{ellcrv_quad_pts} applied to the equation $T^2 = S^3 - 11S + 6$, there is a rational number $w$ and a point $(S_0,T_0)\in\{(-1,\pm 4), (3,0)\}$ such that
\begin{equation}\label{12_211a_Spoly} S^2 + (S_0-w^2)S + S_0^2+w^2S_0-2T_0w-11= 0\;,\; T=T_0+w(S-S_0).\end{equation}

The point $(-1,-4)$ is excluded because it would lead to $x=\frac{w+1}{w-3}\in\Q$. If $(S_0,T_0)=(3,0)$, then \eqref{12_211a_XS} and \eqref{12_211a_Spoly} imply that  \begin{equation}\label{12_211a_xpoly_30}x^2 - \frac{8x}{w^2 - 4w - 1} - \frac{w^2 + 4w - 1}{w^2 - 4w - 1}=0,\end{equation} and if $(S_0,T_0)=(-1,4)$, then \begin{equation}\label{12_211a_xpoly_14}x^2 - \frac{8w + 8}{w^2 - 5}x - 1=0.\end{equation}

We now split the proof into four cases, according to whether $X$ and $S$ are rational or quadratic.

{\bf Case 2a:} $X,S\in\Q$. We simultaneously have relations \eqref{12_211a_xpoly1} and \eqref{12_211a_xpoly2}.  Comparing these equations we obtain the system 
\[
\begin{cases}
& X^2(S^2-10S+5)=(X^2-4X+2)(S^2+2S+1)\\
& (X^2-2)(S^2-10S+5)=(X^2-4X+2)(S^2-2S-19),
\end{cases}
\]

whose rational solutions are $(X,S)= (0, -1)$ and $(X,S) = (1, 3)$. However, both solutions lead to $x=\pm 1$ by applying \eqref{12_211a_xpoly1} and \eqref{12_211a_xpoly2}. We conclude that $X$ and $S$ cannot both be rational.

{\bf Case 2b:} $X\in\Q, S\notin\Q$. We have \eqref{12_211a_xpoly1} and \eqref{12_211a_Spoly}. If $(S_0,T_0)=(3,0)$, then we compare \eqref{12_211a_xpoly1} and \eqref{12_211a_xpoly_30} to arrive at the system 
\[
\begin{cases}
& X^2(w^2-4w-1)=4(X^2-4X+2) \\
& (X^2-2)(w^2-4w-1)=-(X^2-4X+2)(w^2+4w-1),
\end{cases}
\]
whose only rational solution is $(X,w)=(1,1)$. However, when $w=1$, \eqref{12_211a_xpoly_30} becomes $(x+1)^2=0$, a contradiction.

 If $(S_0,T_0)=(-1,4)$, then we compare \eqref{12_211a_xpoly1} and \eqref{12_211a_xpoly_14} to conclude that $X=0$ or 2. Then \eqref{12_211a_xpoly1} becomes $x^2-1=0$ (a contradiction) or $x^2+4x-1=0$.
 
{\bf Case 2c:} $X\notin\Q, S\in\Q$. In this case we have \eqref{12_211a_xpoly2} and \eqref{12_211a_Xpoly}. If $(X_0,Y_0)=(0,1)$, then we compare \eqref{12_211a_xpoly2} and \eqref{12_211a_xpoly_01} to conclude that $S=7$ and $x^2+8x-1=0$.

If $(X_0,Y_0)=(1,0)$, then we compare \eqref{12_211a_xpoly2} and \eqref{12_211a_xpoly_10} to arrive at the system
\[
\begin{cases}
& 2(S^2-10S+5)=(v^2-2v-1)(S+1)^2 \\
& -(v^2+2v-1)(S^2-10S+5)=(v^2-2v-1)(S^2-2S-19),
\end{cases}
\]
whose only rational solution is $(S,v)=(3,1)$. However, if $v=1$, then \eqref{12_211a_xpoly_10} becomes $(x+1)^2=0$, a contradiction.

{\bf Case 2d:} $X,S\notin\Q$. We have \eqref{12_211a_Xpoly} and \eqref{12_211a_Spoly}. 

\begin{itemize}
\item If $(X_0,Y_0)=(1,0)$ and $(S_0,T_0)=(3,0)$, then comparing \eqref{12_211a_xpoly_10} and  \eqref{12_211a_xpoly_30} we find that $v=w=\pm 1$. But then \eqref{12_211a_xpoly_10} implies that $x=\pm 1$, a contradiction.
\item If $(X_0,Y_0)=(1,0)$ and $(S_0,T_0)=(-1,4)$, then we compare \eqref{12_211a_xpoly_10} and  \eqref{12_211a_xpoly_14} to see that $ v=0$ and $x^2+4x-1=0$.
\item If $(X_0,Y_0)=(0,1)$ and $(S_0,T_0)=(3,0)$, then we compare \eqref{12_211a_xpoly_01} and  \eqref{12_211a_xpoly_30} to conclude that $w=0$, and therefore $x^2+8x-1=0$.
\item If $(X_0,Y_0)=(0,1)$ and $(S_0,T_0)=(-1,4)$, then comparing \eqref{12_211a_xpoly_01} and  \eqref{12_211a_xpoly_14} we obtain \[(v+1)(w^2-5)=2(w+1)(v^2-2).\] Let $E\subset\P^2$ be the projective closure of the curve defined by this equation. Then $E$ is a nonsingular plane cubic with at least four rational points, namely the affine point $(-1,-1)$ and three points at infinity. Using Magma we find that $E$ is the elliptic curve with Cremona label 17a4, which has exactly 4 rational points. It follows that $(v,w)=(-1,-1)$ is the only affine point on $E$. But then \eqref{12_211a_xpoly_01} becomes $x^2-1=0$, which is a contradiction.
\end{itemize}

In all cases that have not led to a contradiction we concluded that either $x^2+4x-1=0$ or $x^2+8x-1=0$. If $x^2+4x-1=0$, then \eqref{12_211a_curve_equations} implies that $y=\pm (4x-2)$ and $z=\pm 8x$. If $x^2+8x-1=0$, then $y=\pm 10x$ and $z=\pm(16x-4)$. Thus, we have determined all quadratic points on $C$ having a quadratic $x$-coordinate, and now the theorem follows immediately.
\end{proof}

\begin{cor}\label{12_211a} Let $K$ be a quadratic field and let $c\in K$. Suppose that $G(f_c,K)$ contains a graph of type {\rm 12(2,1,1)a}. Then $c=-13/16$ and $K=\Q(\sqrt{17})$.
\end{cor}

\begin{proof} By Lemma \ref{12_211a_curve}, there is a point $(x,y,z)\in C(K)$ with $(x^2-1)(x^2+4x-1)\ne 0$ such that 
\begin{equation}\label{12_211a_cvalue}c=-\frac{x^4+2x^3+2x^2-2x+1}{(x^2-1)^2}.
\end{equation} It follows from Theorem \ref{12_211a_points} that $K=\Q(\sqrt{17})$ and that $x$ is either $-3,1/3,$ or a quadratic number satisfying $x^2+8x-1=0$. In all three cases, \eqref{12_211a_cvalue} implies that $c=-13/16$. 
\end{proof}


\subsection{Graph 12(2,1,1)b}\label{12_211b_section}
Our search described in \S\ref{quad_prep_comp} produced the pairs \[(K,c)=\left(\Q(\sqrt{-7}), -\frac{5}{16}\right)\;\; \text{and}\;\; \left(\Q(\sqrt{33}),-\frac{45}{16}\right)\] for which the graph $G(f_c,K)$ is of type 12(2,1,1)b. We will show here that, in addition to these known pairs, there are at most four other pairs $(K,c)$ giving rise to this graph structure.

\begin{lem}\label{12_211b_curve} Let $C/\Q$ be the affine curve of genus 5 defined by the equations
\begin{equation}\label{12_211b_curve_equations}
\begin{cases}
&y^2 = -3x^4+14x^2+5 \\
&z^2 = 2(x^3 + x^2 - x + 1).
\end{cases}
\end{equation}
Consider the rational map $\phi: C \dashrightarrow \mathbb{A}^3 = \Spec \Q[p,q,c]$ given by
\begin{equation*} 
p = \frac{y+1-x^2}{2(x^2-1)} \;,\; q = \frac{z}{x^2-1} \;,\; c = \frac{-2(x^2+1)}{(x^2-1)^2}.
\end{equation*}
For every number field $K$, the map $\phi$ induces a bijection from the set \[\{(x,y,z)\in C(K):y(x^2-1)\ne 0\}\] to the set of all tuples $(p,q,c)\in K^3$ such that $p$ is a point of period 2 and $q$ is a point of type $1_3$ for the map $f_c$.
\end{lem}
\begin{figure}[h!]
\begin{center}\includegraphics[scale=0.5]{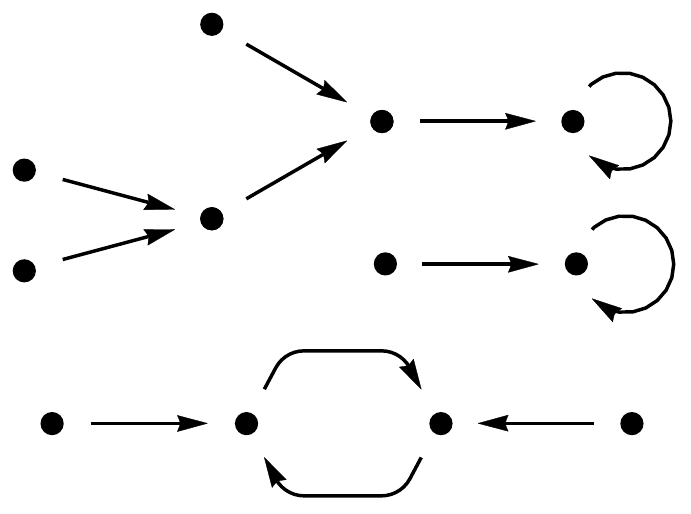}\end{center}
\caption{Graph type 12(2,1,1)b}
\end{figure}
\begin{proof} Fix a number field $K$ and suppose that $(x,y,z)\in C(K)$ is a point with $x^2\ne 1$. Defining $p,q,c\in K$ as in the lemma, it is a simple calculation to verify that $q$ is a point of type $1_3$ for $f_c$ and that
\begin{equation}\label{12_211b_relations}
f_c^2(p)=p\;\;,\;\;p-f_c(p)=\frac{y}{x^2-1}.
\end{equation}
 It follows that if $y\ne 0$, then $p$ has period 2 under $f_c$. Hence, $\phi$ gives a well-defined map.

To see that $\phi$ is surjective, suppose that $p,q,c\in K$ are such that $p$ is a point of period 2 and $q$ is a point of type $1_3$ for $f_c$. By Proposition \ref{cycles_prop}, there is an element $\sigma\in K$ such that
\begin{equation}\label{12_211b_cpvalues}
p = \sigma - 1/2 \;\;\; \mathrm{and} \;\;\; c = -3/4 - \sigma^2.
\end{equation} 
Since $q$ is of type $1_3$, an argument given in \cite[22]{poonen_prep} shows that there is an element $x\in K\setminus\{\pm 1\}$ such that
\begin{equation}\label{12_211b_cqvalues}
q^2 = \frac{2(x^3+x^2-x+1)}{(x^2-1)^2} \;\;\; \mathrm{and} \;\;\; c = \frac{-2(x^2+1)}{(x^2-1)^2}.
\end{equation}
Letting $z=q(x^2-1)$ we then have $z^2 = 2(x^3 + x^2 - x + 1)$. Equating the expressions for $c$ given in \eqref{12_211b_cpvalues} and \eqref{12_211b_cqvalues} we obtain $4\sigma^2(x^2-1)^2=-3x^4+14x^2+5$; hence, letting $y=2\sigma(x^2-1)$ we arrive at the equation $y^2=-3x^4+14x^2+5$. Thus, we have a point $(x,y,z)\in C(K)$ satisfying $x^2\ne 1$ and $\phi(x,y,z)=(p,q,c)$. Moreover, the relations \eqref{12_211b_relations} imply that $y\ne 0$.  To see that $\phi$ is injective, one can verify that if $\phi(x,y,z)=(p,q,c)$, then
\[x=\frac{f_c(q)}{f_c^2(q)}\;\;,\;\;y=(2p+1)(x^2-1)\;\;,\;\;z=q(x^2-1).\qedhere\]
\end{proof}

\begin{lem}\label{12_211b_aux_curve_pts} Let $X/\Q$ be the hyperelliptic curve of genus 3 defined by the equation
\[s^2=-6x^7 - 6x^6 + 34x^5 + 22x^4 - 18x^3 + 38x^2 - 10x + 10.\] 
Then $X(\Q)$ contains the points $\infty, (\pm1, \pm 8), (-3,\pm 56), (1/3,\pm 88/27)$ and at most 8 other points.
\end{lem}

\begin{proof} Using the Magma function \texttt{Points} we search for rational points on $X$ of height at most $10^5$ and obtain the points listed above. Using the \texttt{RankBounds} function we find that the group $\Jac(X)(\Q)$ has rank 1; we may therefore apply the method of Chabauty and Coleman to bound the number of rational points on $X$. The prime 13 is of good reduction for $X$, and we compute $\#X(\F_{13})=16.$ Applying Stoll's bound (Theorem \ref{stoll_bound}) we obtain $\#X(\Q) \le 18$. Since $X$ has no rational point with $s = 0$, the number of rational points must be odd, and therefore $\#X(\Q) \le 17$. Since we have already listed 9 rational points, we conclude that there are at most 8 additional rational points on $X$.
\end{proof}

\begin{thm}\label{12_211b_points} With $C$ as in Lemma \ref{12_211b_curve} we have the following:
\begin{enumerate}
\item $C(\Q)=\{(\pm 1,\pm 4,\pm 2)\}$.
\item If $(x,y,z)$ is a quadratic point on $C$, then $x\in\Q\setminus\{\pm 1\}$. Moreover, there exists $s\in\Q$ such that $(x,s)\in X(\Q)$, where $X$ is the curve defined in Lemma \ref{12_211b_aux_curve_pts}.
\end{enumerate}
\end{thm}

\begin{proof} Let $C_1$ be the hyperelliptic curve of genus 1 defined by the equation $y^2 = -3x^4+14x^2+5$. Making the change of variables $(x,y)\mapsto (1/x,y/x^2)$ we see from \cite[21]{poonen_prep} that $C_1$ is isomorphic to the elliptic curve with Cremona label 15a8, which has exactly four rational points. Therefore, $C_1$ has four rational points, and these are easily found to be $(\pm 1,\pm 4)$. The curve $z^2 = 2(x^3 + x^2 - x + 1)$ is birational to the elliptic curve 11a3, which is the modular curve $X_1(11)$. This curve has five rational points, the nontrivial ones being $(\pm 1, \pm 2)$. Knowing all rational solutions to both of the equations defining $C$, we deduce that $C(\Q)=\{(\pm 1,\pm 4,\pm 2)\}$.

To prove part (2) of the theorem, suppose that $(x,y,z)\in C(\qbar)$ satisfies $[\Q(x,y,z):\Q]=2$, and let $K=\Q(x,y,z)$.

{\bf Case 1:} $x\in\Q$. We cannot have $x=\pm 1$, since this would imply that $y=\pm 4$ and $z=\pm 2$, contradicting the assumption that $(x,y,z)$ is a quadratic point on $C$. It follows that $y\notin \Q$, since having $x,y\in\Q$ would imply that $x=\pm 1$. By a similar argument, $z\notin\Q$. Therefore, $K=\Q(y)=\Q(z)$, so the numbers $ -3x^4+14x^2+5$ and $2(x^3 + x^2 - x + 1)$ must have the same squarefree part, and so their product is a square; hence, there is a rational number $s$ such that \[s^2=-6x^7 - 6x^6 + 34x^5 + 22x^4 - 18x^3 + 38x^2 - 10x + 10.\]
Thus, we have $(x,s)\in X(\Q)$.

{\bf Case 2:} $x$ is quadratic. By Lemma \ref{ellcrv_quad_pts} applied to the equation $z^2 = 2(x^3 + x^2 - x + 1)$, there is a rational number $w$ and a point $(x_0,z_0)\in\{(\pm 1,\pm 2)\}$ such that 
\begin{equation}\label{12_211b_xpoly1} x^2 + \frac{2x_0-w^2+2}{2}x + \frac{2x_0^2+w^2x_0+2x_0-2z_0w-2}{2}= 0.\end{equation}
With $x,y$ as in \eqref{12_211b_curve_equations} we make the change of variables 

\begin{equation}\label{12_11b_chv}
X = \frac{-x^2 + 4x + y + 1}{2(x - 1)^2}\;,\;Y = \frac{-3x^3 + 7x^2 + xy + 7x + 3y + 5}{2(x - 1)^3}
\end{equation} 

satisfying \begin{equation}\label{12_211b_X} x=\frac{Y + X^2 + 2X}{X^2 + X + 1} \end{equation}

to obtain the equations
\[
\begin{cases}
& Y^2 = 4X^3 + 5X^2 + 2X + 1\\
 & z^2 = 2(x^3 + x^2 - x + 1).
 \end{cases}
\]

Suppose that $X\in\Q$. Substituting $y=2X(x-1)^2+x^2-4x-1$ into the equation $y^2 = -3x^4+14x^2+5$ we obtain the following expression for the minimal polynomial of $x$:
\begin{equation}\label{12_211b_xpoly2}x^2 - \frac{2X^2 + 4X}{X^2 + X + 1}x + \frac{X^2 - X - 1}{X^2 + X + 1} = 0.\end{equation} 

Comparing \eqref{12_211b_xpoly1} and \eqref{12_211b_xpoly2} we arrive at the system
\[
\begin{cases}
& (2x_0-w^2+2)(X^2+X+1)=-4X(X+2)\\
 & (2x_0^2+w^2x_0+2x_0-2z_0w-2)(X^2+X+1) = 2(X^2-X-1).
 \end{cases}
\] 

For every choice of point $(x_0,z_0)\in\{(\pm 1,\pm 2)\}$ the above system defines a 0-dimensional scheme over $\Q$, whose rational points we compute using the Magma function \texttt{RationalPoints}. In every case we find that $X=-1$ or $X=0$. But then \eqref{12_211b_xpoly2} implies that $x=\pm 1$, which is a contradiction. Therefore, $X$ cannot be rational, so by Lemma \ref{ellcrv_quad_pts} applied to the equation $Y^2 = 4X^3 + 5X^2 + 2X + 1$, there is a rational number $v$ and a point $(X_0,Y_0)\in\{(0,\pm 1), (-1,0)\}$ such that 
\begin{equation}\label{12_211b_Xpoly} X^2 + \frac{4X_0-v^2+5}{4}X + \frac{4X_0^2+v^2X_0+5X_0-2Y_0v+2}{4}=0 \;,\; Y=Y_0+v(X-X_0).\end{equation}

The point $(X_0,Y_0)=(0,-1)$ is excluded because \eqref{12_211b_X} would imply that $x=\frac{v+3}{v-1}\in\Q$. If $(X_0,Y_0)=(0,1)$, then \eqref{12_211b_X} and \eqref{12_211b_Xpoly} imply that 
\begin{equation}\label{12_211b_xpoly_01} x^2-\frac{v^2+4v-1}{v^2-4v+7}=0,\end{equation} and if $(X_0,Y_0)=(-1,0)$, then 
\begin{equation}\label{12_211b_xpoly_10} x^2 - \frac{8v}{v^2+3}x - \frac{v^2-5}{v^2+3} = 0.\end{equation}

Suppose that $(X_0,Y_0)=(0,1)$. Comparing \eqref{12_211b_xpoly1} and \eqref{12_211b_xpoly_01} we obtain the system
\[
\begin{cases}
& 2x_0-w^2+2=0\\
 & (2x_0^2+w^2x_0+2x_0-2z_0w-2)(v^2-4v+7) = -2(v^2+4v-1).
 \end{cases}
\]

For each choice of point $(x_0,z_0)\in\{(\pm 1,\pm 2)\}$ we solve the above system for $v$ and $w$, and find that $v=1$. But then \eqref{12_211b_xpoly_01} implies that $x=\pm 1$, which is a contradiction.

Suppose now that $(X_0,Y_0)=(-1,0)$. Comparing \eqref{12_211b_xpoly1} and \eqref{12_211b_xpoly_10} we obtain the system
\[
\begin{cases}
& (2x_0-w^2+2)(v^2+3)=-16v\\
 & (2x_0^2+w^2x_0+2x_0-2z_0w-2)(v^2+3) = -2(v^2-5).
 \end{cases}
\]

For each choice of point $(x_0,z_0)\in\{(\pm 1,\pm 2)\}$ we solve the above system for $v$ and $w$, and find that $v=\pm 1$. But then \eqref{12_211b_xpoly_10} implies that $x=\pm 1$, a contradiction.

Since the assumption that $x\notin\Q$ has led in every case to a contradiction, we conclude that $x$ must be rational. The analysis done in Case 1 then proves the theorem.
\end{proof}

\begin{cor} In addition to the known pairs $(\Q(\sqrt{-7}), -5/16)$ and $(\Q(\sqrt{33}),-45/16)$ there are at most four pairs $(K,c)$, with $K$ a quadratic number field and $c\in K$, for which $G(f_c,K)$ contains a graph of type {\rm 12(2,1,1)b}. Moreover, for every such pair we must have $c\in\Q$.
\end{cor}

\begin{proof} Suppose that $(K,c)$ is such a pair. Since $f_c$ has a point of period 2 and a point of type $1_3$ in $K$, then by Lemma \ref{12_211b_curve} there is a point $(x,y,z)\in C(K)$ with $y(x^2-1)\ne 0$ such that 
\begin{equation}\label{12_211b_c_param}
c = \frac{-2(x^2+1)}{(x^2-1)^2}.
\end{equation} 
We cannot have $(x,y,z)\in C(\Q)$ since this would imply, by Theorem \ref{12_211b_points} part (1), that $x\in\{\pm 1\}$. Hence, $(x,y,z)$ is a quadratic point on $C$. By Theorem \ref{12_211b_points} part (2) we must therefore have $x\in\Q$, and thus $c\in\Q$. Moreover, there is a rational number $s$ such that $(x,s)\in X(\Q)$. It follows from Lemma \ref{12_211b_aux_curve_pts} that either $x\in\{-3,1/3\}$ or $x$ belongs to a list of at most 4 other rational numbers.
Setting $x=-3$ we obtain by \eqref{12_211b_c_param} and \eqref{12_211b_curve_equations} that $c=-5/16$ and $y^2=-112$, so $K=\Q(y)=\Q(\sqrt{-7})$. If $x=1/3$, then $c=-45/16$ and $y^2=176/27$, so $K=\Q(y)=\Q(\sqrt{33})$. Thus, we recover the two pairs $(K,c)$ listed in the statement of the corollary. 
If $x\notin\{-3,1/3\}$, then there are at most four options for $x$; each value of $x$ determines the number $c$ by \eqref{12_211b_c_param} and the field $K$ by  \eqref{12_211b_curve_equations}. This gives at most four options for the pair $(K,c)$.  
\end{proof}


\subsection{Graph 12(4)}\label{12_4_section}
Our search described in \S\ref{quad_prep_comp} produced a unique pair \[ (K,c) = \left(\Q(\sqrt{105}), -95/48\right) \] consisting of a quadratic field $K$ and an element $c\in K$ for which the graph $G(f_c, K)$ is of type 12(4). We will show here that, in addition to this known example, there are at most five other such pairs $(K,c)$.

\begin{lem}\label{12_4_curve} Let $C/\Q$ be the affine curve of genus 9 defined by the equations
\begin{equation}\label{12_4_curve_equations}
\begin{cases}
& y^2 = -x(x^2+1)(x^2-2x-1)\\
& z^2 = x(-x^6 + x^5 + 7x^4 + 10x^3 - 7x^2 + 5x + 1) - 2x(x - 1)(x + 1)^2y.
\end{cases}
\end{equation}

Consider the rational map $\varphi: C \dashrightarrow \mathbb{A}^2 = \Spec \Q[p,c]$ given by 
\[p=\frac{z}{2x(x^2-1)}\;\;,\;\;c=\frac{(x^2 - 4x - 1)(x^4 + x^3 + 2x^2 - x + 1)}{4x(x^2 - 1)^2}.\]
For every number field $K$, the map $\phi$ induces a bijection from the set \[\left\{(x,y,z)\in C(K): y(x^2-1)\left(y(x+1)+x(x-1)^2\right)\ne 0\right\}\] to the set of all pairs $(p,c)\in K^2$ such that $p$ is a point of type $4_2$ for the map $f_c$.
\end{lem}
\begin{figure}[h!]
\begin{center}\includegraphics[scale=0.5]{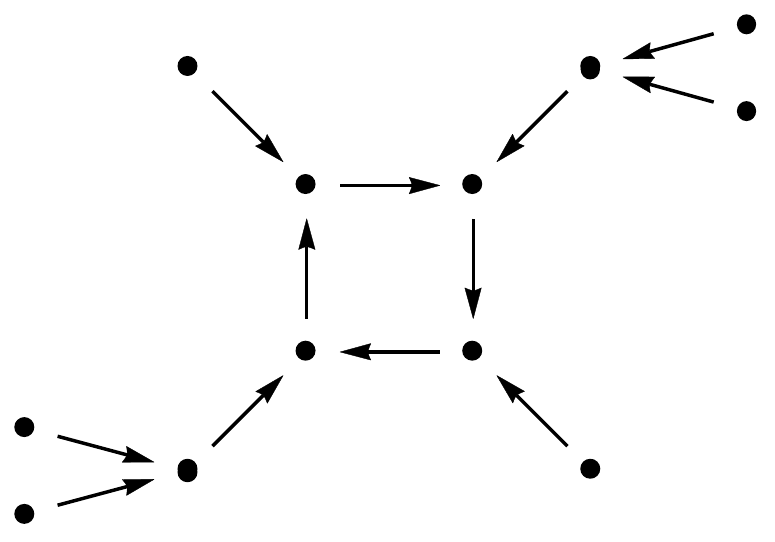}\end{center}
\caption{Graph type 12(4)}
\end{figure}
\begin{proof} Fix a number field $K$, suppose that $(x,y,z)\in C(K)$ is a point satisfying $y(x^2-1)\ne 0$, and define $p,c\in K$ as in the lemma. From Lemma \ref{8_4_curve} it follows that the point \[r:=\frac{x-1}{2(x+1)} + \frac{y}{2x(x-1)}\] has period 4 under $f_c$. Now, it is a simple calculation to verify that $f_c(p)=-r$; in order to conclude that $p$ is of type $4_2$ for $f_c$ it is therefore enough to show that $-r$ is of type $4_1$. Since $r$ has period 4, $-r$ will necessarily be of type $4_1$ as long as $r \ne -r$. We have
\begin{equation}\label{12_4_rzero}
r=0\Longleftrightarrow y(x+1)+x(x-1)^2=0,
\end{equation}
so if we assume that $y(x+1)+x(x-1)^2\ne 0$, then it follows that $p$ is of type $4_2$ for $f_c$. This proves that $\phi$ gives a well-defined map.

To see that $\phi$ is surjective, suppose that $p,c\in K$ are such that $p$ is a point of type $4_2$ for the map $f_c$. Then $q:=p^2+c$ is a point of type $4_1$, so $r:=-q$ is a point of period 4. By Lemma \ref{8_4_curve} there are elements $x,y\in K$ satisfying $y^2=-x(x^2+1)(x^2-2x-1)$ and $y(x^2-1)\ne 0$ such that 
\[
c = \frac{(x^2 - 4x - 1)(x^4 + x^3 + 2x^2 - x + 1)}{4x(x-1)^2(x+1)^2} \mbox{\;,\;\;} r = \frac{x-1}{2(x+1)} + \frac{y}{2x(x-1)}.
\]
Clearing denominators in the equation 
\[-p^2-\frac{(x^2 - 4x - 1)(x^4 + x^3 + 2x^2 - x + 1)}{4x(x-1)^2(x+1)^2}= \frac{x-1}{2(x+1)} + \frac{y}{2x(x-1)}
\] and letting $z=2x(x^2-1)p$ we obtain $z^2 = x(-x^6 + x^5 + 7x^4 + 10x^3 - 7x^2 + 5x + 1) - 2x(x - 1)(x + 1)^2y$. Thus, we have a point $(x,y,z)\in C(K)$ with $y(x^2-1)\ne 0$ and $\phi(x,y,z)=(p,c)$. Furthermore, we must have $y(x+1)+x(x-1)^2\ne 0$ because otherwise \eqref{12_4_rzero} would imply that $q=r$, and hence that $q$ has period 4, a contradiction.  

To see that $\phi$ is injective, one can verify that if $\phi(x,y,z)=(p,c)$, then
\[x=\frac{f_c^4(p)+f_c^2(p)-1}{f_c^4(p)+f_c^2(p)+1}\;\;,\;\;y=-\frac{2x(x^2-1)f_c(p)+x(x-1)^2}{x+1}\;\;,\;\;z=2x(x^2-1)p.\qedhere\]
\end{proof}

\begin{lem}\label{12_4_aux_curve_pts} Let $X/\Q$ be the hyperelliptic curve of genus 5 defined by the equation 
\[w^2=x^{12} + 2x^{11} - 13x^{10} - 26x^9 + 67x^8 + 124x^7 + 26x^6 - 44x^5 + 179x^4 - 62x^3 - 5x^2 + 6x + 1.\]
Then $X(\Q)$ contains the points $\infty^{+}, \infty^{-}, (\pm 1,\pm 16), (0, \pm 1), (-3, \pm 368)$ and at most 10 other points.
\end{lem}

\begin{proof} Using the Magma function \texttt{Points} we search for rational points on $X$ of height at most $10^5$ and obtain the points listed above. Using the \texttt{RankBound} function we obtain an upper bound of 4 for the rank of $\Jac(X)(\Q)$; we are thus in a position to bound the number of rational points on $X$ using the method of Chabauty and Coleman. The prime 7 is of good reduction for $X$, and Magma's \texttt{Points} function yields $\#X(\F_{7})=12$. Applying the Lorenzini-Tucker bound (Theorem \ref{lorenzini_tucker_bound}) with $p=7$, $d=2$ we obtain \[\#X(\Q) \le \#X(\F_7) + \frac{6}{5}(8)=12+\frac{6}{5}(8) <22,\]
so $\#X(\Q)\le 21$. However, since $X$ has no rational point with $w = 0$, the number of rational points must be even, and therefore $\#X(\Q) \le 20$. Since we have already found 10 rational points, we conclude that there are at most 10 additional rational points on $X$.
\end{proof}

\begin{thm}\label{12_4_points} With $C$ as in Lemma \ref{12_4_curve} we have the following:
\begin{enumerate}
\item $C(\Q)=\{(0,0,0),(\pm 1,\pm 2,\pm 4)\}$.
\item If $(x,y,z)$ is a quadratic point on $C$, then $x\in\Q\setminus\{0,\pm 1\}$. Moreover, there exists $w\in\Q$ such that $(x,w)\in X(\Q)$, where $X$ is the curve defined in Lemma \ref{12_4_aux_curve_pts}.
\end{enumerate}
\end{thm}

\begin{proof} As noted in \S\ref{modular_genus2_section}, the curve $y^2 = -x(x^2+1)(x^2-2x-1)$ is an affine model for the modular curve $X_1(16)$, which has exactly six rational points:
\[X_1(16)(\Q)=\{\infty,(0,0),(\pm 1,\pm 2)\}.\] It follows that if $(x,y,z)\in C(\Q)$, then $(x,y)=(0,0)$ or $(x,y)\in\{(\pm 1,\pm 2)\}$. Setting these values of $x,y$ in the system \eqref{12_4_curve_equations} and solving for $z$ we conclude that $C(\Q)=\{(0,0,0),(\pm 1,\pm 2,\pm 4)\}$.

Suppose now that $(x,y,z)\in C(\overline\Q)$ is a point with $[\Q(x,y,z):\Q]=2$, and let $K=\Q(x,y,z)$. We cannot have $x\in\{0,\pm 1\}$ since this would imply that $(x,y,z)\in C(\Q)$; it follows that $(x,y)$ cannot be a rational point on $X_1(16)$, and must therefore be quadratic. We claim that $x$ must be rational. Indeed, if $x\notin\Q$, then by Theorem \ref{131618_quad}, either $x^2=-1$ and $y=0$, or $x^2-2x-1=0$ and $y=0$. However, in both cases we find that the second equation in \eqref{12_4_curve_equations} has no solution in $K$. Thus, we conclude that $x$ must be a rational number different from $0,\pm 1$. Since $y^2\in\Q$ and $y\notin\Q$, the Galois conjugate of $y$ is $-y$. Hence, taking norms on both sides of  the second equation in \eqref{12_4_curve_equations} we obtain 
\[w^2=x^{12} + 2x^{11} - 13x^{10} - 26x^9 + 67x^8 + 124x^7 + 26x^6 - 44x^5 + 179x^4 - 62x^3 - 5x^2 + 6x + 1,\]
where $w=N_{K/\Q}(z)/x$.
\end{proof}

\begin{cor} In addition to the known pair $(\Q(\sqrt{105}),-95/48)$ there are at most five pairs $(K,c)$, with $K$ a quadratic number field and $c\in K$, for which $G(f_c,K)$ contains a graph of type {\rm 12(4)}. Moreover, for every such pair we must have $c\in\Q$.
\end{cor}

\begin{proof} Suppose that $(K,c)$ is such a pair. Since $f_c$ has a point of type $4_2$ in $K$, then Lemma \ref{12_4_curve} implies that there is a point $(x,y,z)\in C(K)$ with $y(x^2-1)\ne 0$ such that 
\begin{equation}\label{12_4_c_param}
c=\frac{(x^2 - 4x - 1)(x^4 + x^3 + 2x^2 - x + 1)}{4x(x^2 - 1)^2}.
\end{equation} 
We cannot have $(x,y,z)\in C(\Q)$ since this would imply that $x\in\{0,\pm 1\}$. Hence, $(x,y,z)$ is a quadratic point on $C$. By Theorem \ref{12_4_points}, $x\in\Q$ and thus $c\in\Q$. Moreover, there is a rational number $w$ such that $(x,w)\in X(\Q)$. It follows from Lemma \ref{12_4_aux_curve_pts} that either $x=-3$ or $x$ belongs to a list of at most 5 other rational numbers. Setting $x = -3$ yields $c = -95/48$, and the system \eqref{12_4_curve_equations} becomes
\begin{equation*}
\begin{cases}
& y^2 = 420\\
& z^2 = 2256 - 96y.
\end{cases}
\end{equation*}
Hence, $y = \pm 2\sqrt{105}$ and $z=\pm(2y-24)$. In particular, $K = \Q(\sqrt{105})$, so we have recovered the known pair $(\Q(\sqrt{105}),-95/48)$. If $x\ne -3$, then there are at most five options for $x$; each value of $x$ determines the number $c$ by \eqref{12_4_c_param} and the field $K$ by  \eqref{12_4_curve_equations}. This gives at most five options for the pair $(K,c)$.  
\end{proof}


\subsection{Graph 12(4,2)}\label{12_42_section}

Our search described in \S\ref{quad_prep_comp} produced the pair \[ (K,c) = \left(\Q(\sqrt{-15}), -\frac{31}{48}\right) \] for which the graph $G(f_c,K)$ is of type 12(4,2). We will show here that in addition to this known example there is at most one other such pair $(K,c)$ consisting of a quadratic number field $K$ and an element $c\in K$.

\begin{lem}\label{12_42_curve} Let $C/\Q$ be the affine curve of genus 9 defined by the equations
\begin{equation}\label{12_42_curve_equations}
\begin{cases}
& y^2 = -x(x^2 + 1)(x^2 - 2x - 1) \\
& z^2=-x(x^6-3x^4-16x^3+3x^2-1).
\end{cases}
\end{equation}

Consider the rational map $\varphi: C \dashrightarrow \mathbb{A}^3 = \Spec \Q[a,b,c]$ given by 
\[a=\frac{z-x(x^2-1)}{2x(x^2-1)}\;\;,\;\;b = \frac{x-1}{2(x+1)} + \frac{y}{2x(x-1)}\;\;,\;\;c = \frac{(x^2 - 4x - 1)(x^4 + x^3 + 2x^2 - x + 1)}{4x(x^2-1)^2}.\]
For every number field $K$, the map $\phi$ induces a bijection from the set \[\{(x,y,z)\in C(K): yz(x^2-1)\ne 0\}\] to the set of all triples $(a,b,c)\in K^3$ such that $a$ is a point of period 2 for the map $f_c$ and $b$ is a point of period 4.
\end{lem}
\begin{figure}[h!]
\begin{center}\includegraphics[scale=0.5]{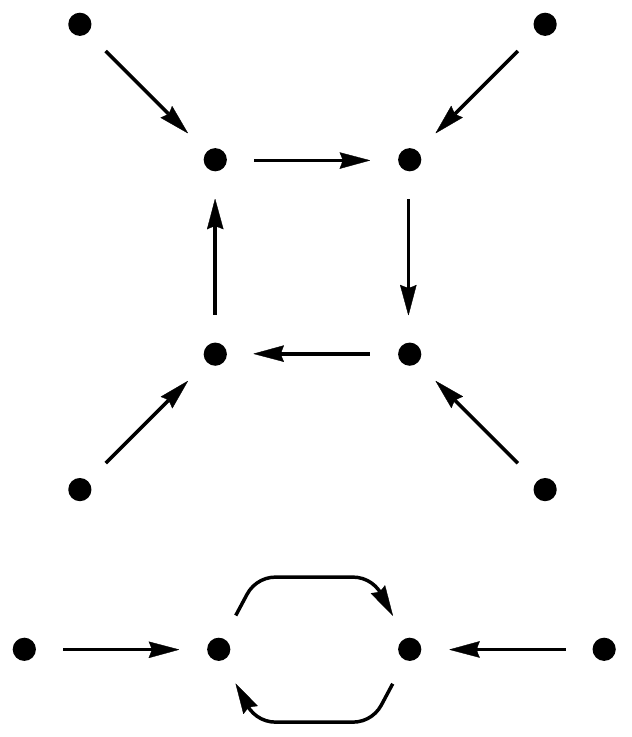}\end{center}
\caption{Graph type 12(4,2)}
\end{figure}
\begin{proof}
Fix a number field $K$ and suppose that $(x,y,z)\in C(K)$ is a point satisfying $x(x^2-1)\ne 0$. Defining $a,b,c\in K$ as in the lemma, it is a straightforward calculation to verify that $f_c^2(a)=a$, $f_c^4(b)=b$, and
\begin{equation}\label{12_42_relations}
a-f_c(a)=\frac{z}{x(x^2-1)}\;\;,\;\;b-f_c^2(b)=\frac{y}{x(x-1)}.
\end{equation}
It follows from these relations that if $yz\ne 0$, then $a$ is a point of period 2 for $f_c$ and $b$ is a point of period 4. Hence, $\phi$ gives a well-defined map.

To see that $\phi$ is surjective, suppose that $a,b,c\in K$ satisfy the conditions of the lemma. Since $a$ is a point of period 2 for $f_c$, then by Proposition \ref{cycles_prop} there is an element $\sigma\in K$ such that
\begin{equation}\label{12_42_acvalues}a=\sigma-1/2\;\;, \;\;c=-3/4-\sigma^2.
\end{equation}
Since $b$ is a point of period 4, then by Proposition \ref{cycles_prop} there are elements $x,y\in K$ with $y(x^2-1)\ne 0$ such that $y^2 = -x(x^2 + 1)(x^2 - 2x - 1)$ and 
\begin{equation}\label{12_42_bcvalues}b = \frac{x-1}{2(x+1)} + \frac{y}{2x(x-1)}\;\;,\;\;c = \frac{(x^2 - 4x - 1)(x^4 + x^3 + 2x^2 - x + 1)}{4x(x-1)^2(x+1)^2}.
\end{equation}
Equating the two expressions for $c$ given in \eqref{12_42_acvalues} and \eqref{12_42_bcvalues}, and setting $z=2x(x-1)(x+1)\sigma$, we obtain the relation $z^2=-x(x^6-3x^4-16x^3+3x^2-1)$. Thus, we have a point $(x,y,z)\in C(K)$ with $\phi(x,y,z)=(a,b,c)$ and $y(x^2-1)\ne 0$. Furthermore, the relations \eqref{12_42_relations} imply that $z\ne 0$.  To see that $\phi$ is injective, one can verify that if $\phi(x,y,z)=(a,b,c)$, then
\[x=\frac{1+b+f_c^2(b)}{1-b-f_c^2(b)}\;\;,\;\;y=\frac{2x(x^2-1)b-x(x-1)^2}{x+1}\;\;,\;\;z=x(x^2-1)(2a+1).\qedhere\]
\end{proof}

\begin{lem}\label{12_42_aux_curve_pts} Let $X/\Q$ be the hyperelliptic curve of genus 4 defined by the equation 
\[w^2=(x^2 + 1)(x^2 - 2x - 1)(x^6-3x^4-16x^3+3x^2-1).\]
Then $X(\Q)$ contains the twelve points 
\[\infty^+,\infty^-,(\pm 1,\pm 8),(0,\pm 1),(3,\pm 40),(-1/3,\pm 40/243)\] 
and at most four other points.
\end{lem}

\begin{proof} Using the Magma function \texttt{Points} we search for rational points on $X$ of height at most $10^5$ and obtain the points listed above. Using the \texttt{RankBound} function  we find that $\Jac(X)(\Q)$ has rank at most 2; thus, we may apply the method of Chabauty and Coleman to bound the number of rational points on $X$. The prime 11 is of good reduction for $X$, and we compute $\#X(\F_{11})=12$; Stoll's bound (Theorem \ref{stoll_bound}) then yields $\#X(\Q)\le 16$.
\end{proof}

\begin{rem} The curve $X$ from Lemma \ref{12_42_aux_curve_pts} has an automorphism of order 4 given by $(x,w)\mapsto (-1/x,w/x^5)$. Hence, if there is some rational point $(x_0,w_0)\in X(\Q)$ missing from the list above, then in fact there are exactly four of them, namely $(x_0,w_0), (-1/x_0,w_0/x_0^5), (x_0,-w_0)$, and $(-1/x_0,-w_0/x_0^5)$.
\end{rem}

\begin{thm}\label{12_42_points} With $C$ as in Lemma \ref{12_42_curve} we have the following:
\begin{enumerate}
\item $C(\Q)=\{(0,0,0), (\pm 1,\pm 2,\pm 4)\}$.
\item The quadratic points $(x,y,z)$ on $C$ with $x\notin\Q$ are the points $(x,0,\pm(6x+2))$ with $x^2-2x-1=0$, which are defined over the field $\Q(\sqrt 2)$.
\item If $(x,y,z)$ is a quadratic point on $C$ with $x\in\Q$, then there exists a rational number $w$ such that $(x,w)\in X(\Q)$, where $X$ is the curve defined in Lemma \ref{12_42_aux_curve_pts}. Moreover, $x\notin\{0,\pm 1\}$.
\end{enumerate}
\end{thm}

\begin{proof} As noted in \S\ref{modular_genus2_section}, the curve defined by the equation $y^2 =-x(x^2 + 1)(x^2 - 2x - 1)$ is an affine model for the modular curve $X_1(16)$, which has exactly six rational points:
\begin{equation}\label{12_42_C1pts}
X_1(16)(\Q)= \{\infty,(0,0),(\pm 1,\pm 2)\}.
\end{equation}

Let $C_2$ denote the hyperelliptic curve defined by the equation $z^2=-x(x^6-3x^4-16x^3+3x^2-1)$. Note that $C_2$ has an involution given by $(x,z)\mapsto (-1/x,z/x^4)$. The quotient of $C_2$ by this involution is the elliptic curve $E$ defined by the Weierstrass equation $s^2+s=r^3$; the quotient map $C_2\to E$ of degree 2 is given by 
\[r=\frac{1-x^2}{4x}\;\;,\;\;s=\frac{-4x^2-z}{8x^2}.\] The elliptic curve $E$ is the curve 27a3 in Cremona's tables, and has exactly three rational points. It follows that $C_2$ has at most six rational points, and a search quickly yields six such points:
\begin{equation}\label{12_42_C2pts}
C_2(\Q)= \{\infty,(0,0),(\pm 1,\pm 4)\}.
\end{equation}
From \eqref{12_42_C1pts} and \eqref{12_42_C2pts} we deduce that $C(\Q)=\{(0,0,0), (\pm 1,\pm 2,\pm 4)\}$.

Suppose now that $(x,y,z)\in C(\overline\Q)$ is a point with $[\Q(x,y,z):\Q]=2$, and let $K=\Q(x,y,z)$.

{\bf Case 1:} $x$ is quadratic. In the terminology of \S\ref{quad_pts_on_crvs}, the point $(x,y)$ is then a non-obvious quadratic point on $X_1(16)$, so Theorem \ref{131618_quad} implies that either $x^2+1=0$ or $x^2-2x-1=0$. If $x^2+1=0$, one can check that the equation $z^2=-x(x^6-3x^4-16x^3+3x^2-1)$ has no solution $z\in K$; hence, this cannot occur. Assuming $x^2-2x-1=0$, we can solve the system \eqref{12_42_curve_equations} to obtain $y=0$, $z=\pm(6x+2)$. Thus, we have shown that the only quadratic points on $C$ having a quadratic $x$-coordinate are the points $(x,0,\pm(6x+2))$ with $x^2-2x-1=0$.

{\bf Case 2:} $x\in\Q$. We cannot have $x\in\{0,\pm 1\}$ since this would imply that $y\in\{0,\pm 2\}$ and $z\in\{0,\pm 4\}$, contradicting the assumption that $(x,y,z)$ is a quadratic point on $C$; hence, $x(x^2-1)\ne 0$. It follows that $y\notin \Q$, since having $x,y\in\Q$ would imply that $x\in\{0,\pm 1\}$, by \eqref{12_42_C1pts}. Similarly, \eqref{12_42_C2pts} implies that $z\notin\Q$. Thus, we have $K=\Q(y)=\Q(z)$, so the rational numbers $-x(x^2 + 1)(x^2 - 2x - 1)$ and $-x(x^6-3x^4-16x^3+3x^2-1)$ must have the same squarefree part; hence, their product is a square, so there is a rational number $v$ such that 
\[v^2=x^2(x^2 + 1)(x^2 - 2x - 1)(x^6-3x^4-16x^3+3x^2-1).\] Letting $w=v/x$ we then have $(x,w)\in X(\Q)$.
\end{proof}

\begin{cor} In addition to the known pair $(\Q(\sqrt{-15}),-31/48)$ there is at most one pair $(K,c)$, with $K$ a quadratic number field and $c\in K$, for which $G(f_c,K)$ contains a graph of type {\rm 12(4,2)}. Moreover, for every such pair we must have $c\in\Q$.
\end{cor}

\begin{proof} Suppose that $(K,c)$ is such a pair. By Lemma \ref{12_42_curve} there is a point $(x,y,z)\in C(K)$ with $y(x^2-1)\ne 0$ such that 
\begin{equation}\label{12_42_cvalue}c = \frac{(x^2 - 4x - 1)(x^4 + x^3 + 2x^2 - x + 1)}{4x(x-1)^2(x+1)^2}.
\end{equation}
It follows from Theorem \ref{12_42_points} that there is a rational number $w$ such that $(x,w)\in X(\Q)$; in particular, $c\in\Q$. We consider two possibilities: either $X(\Q)$ consists of the 12 points listed in  Lemma \ref{12_42_aux_curve_pts}, or $\#X(\Q)=16$ (see the remark following the lemma). 

Suppose first that $\#X(\Q)=12$. Then we must have $x=3$ or $x=-1/3$, since the values $x=0,\pm 1$ are not allowed. If $x=3$, then \eqref{12_42_cvalue} yields $c=-31/48$ and we have $K=\Q(y)$ with $y^2 = -60$, so $K=\Q(\sqrt{-15})$. Thus, we recover the pair $(K,c)$ in the statement of the corollary. With $x=-1/3$ we again obtain $c=-31/48$ and $K=\Q(\sqrt{-15})$.

Suppose now that $\#X(\Q)=16$. Then there are rational numbers $x_0$ and $w_0$ such that $X(\Q)$ consists of the twelve points listed in Lemma \ref{12_42_aux_curve_pts} and the four points 
\[(x_0,w_0), (-1/x_0,w_0/x_0^5), (x_0,-w_0), (-1/x_0,-w_0/x_0^5).\]
We have already determined what $K$ and $c$ must be if $x\in\{3,-1/3\}$, so it remains only to consider the case where $(x,w)$ is one of the four points listed above. Each one of these points may lead via \eqref{12_42_curve_equations} and \eqref{12_42_cvalue} to a new pair $(K,c)$; however, we make the following observations: first, the expression for $c$ given in \eqref{12_42_cvalue} is invariant under the change of variables $x\mapsto -1/x$, so that all four points will yield the same value of $c$. Second, we know that $K=\Q(y)$ and $y^2=-x(x^2 + 1)(x^2 - 2x - 1)$. One can check that, regardless of whether $x=x_0$ or $x=-1/x_0$, the field $K$ remains the same, since the value of $-x(x^2 + 1)(x^2 - 2x - 1)$ only changes by a factor of $x_0^6$ when we substitute $x=x_0$ and $x=-1/x_0$. Hence, the case $\#X(\Q)=16$ will yield at most one other possibility for the pair $(K,c)$.
\end{proof}


\subsection{Graph 14(2,1,1)}\label{14_211_section}
Our search described in \S\ref{quad_prep_comp} produced the pair \[ (K,c) = \left(\Q(\sqrt{17}), -\frac{21}{16}\right) \] for which the graph $G(f_c,K)$ is of type 14(2,1,1). We will show here that this is the only such pair $(K,c)$ consisting of a quadratic number field $K$ and an element $c \in K$.

\begin{lem}\label{14_211_curve}
Let $C/\Q$ be the affine curve of genus 5 defined by the equations
	\begin{equation}\label{14_211_curve_equations}
		\begin{cases}
			& y^2 = 2(x^3 + x^2 - x + 1)\\
			& z^2 = 2x(x^3 + x^2 + x - 1).
		\end{cases}
	\end{equation}
Consider the rational map $\phi : C \dashrightarrow \mathbb{A}^3 = \Spec \Q[a,b,c]$ given by
	\[ a = \frac{y}{x^2 - 1}\ , \ b = \frac{z}{x^2 - 1}\ , \ c = -\frac{x^4 + 2x^3 + 2x^2 - 2x + 1}{(x^2 - 1)^2}.\]
For every number field $K$, the map $\phi$ induces a bijection from the set
	\[ \{ (x,y,z) \in C(K) : x(x^4-1)(x^2 + 4x - 1) \ne 0 \} \]
to the set of all triples $(a,b,c) \in K^3$ such that $a$ and $b$ are points of type $2_3$ for the map $f_c$ satisfying $f_c^2(a)=f_c^2(b)$ and $f_c(a) \ne f_c(b)$.
\end{lem}
\begin{figure}[h!]
\begin{center}\includegraphics[scale=0.5]{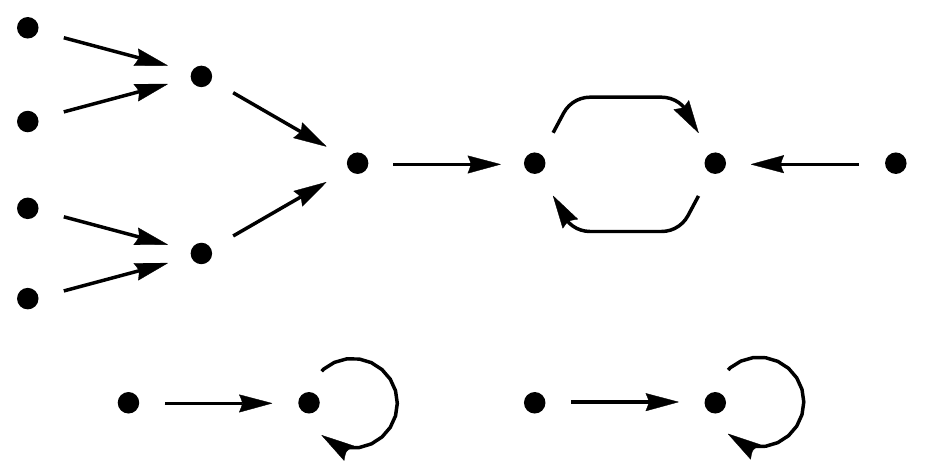}\end{center}
\caption{Graph type 14(2,1,1)}
\end{figure}
\begin{proof} Fix a number field $K$ and suppose that $(x,y,z)\in C(K)$ is a point with $x^2\ne 1$. Defining $a,b,c\in K$ as in the lemma, it is a straightforward calculation to verify $f_c^2(a)=f_c^2(b)$, $f_c^3(a)=f_c^5(a)$ and 
\begin{equation}\label{14_211_relations}
f_c(b)-f_c(a)=\frac{2(x^2+1)}{x^2-1}\;\;,\;\;f_c^4(a)-f_c^3(a)=\frac{x^2+4x-1}{x^2-1}\;\;,\;\;f_c^4(a)-f_c^2(a)=\frac{4x}{x^2-1}.
\end{equation}

It follows from these relations that if $x(x^2+1)(x^2+4x-1)\ne 0$, then $a$ and $b$ are points of type $2_3$ for $f_c$ such that $f_c^2(a)=f_c^2(b)$ and $f_c(a) \ne f_c(b)$. Hence, $\phi$ gives a well-defined map.

To see that $\phi$ is surjective, suppose that $a,b,c\in K$ satisfy the conditions of the lemma. Then an argument given in \cite[23]{poonen_prep} shows that there is an element $x\in K\setminus\{-1,0,1\}$ such that 
\begin{equation}\label{14_211_cparam}
 c = -\frac{x^4 + 2x^3 + 2x^2 - 2x + 1}{(x^2 - 1)^2} \;\;\;\mathrm{and}\;\;\; a^2 = \frac{2(x^3+x^2-x+1)}{(x^2-1)^2}.
\end{equation}
Since $f_c(b)=-f_c(a)$, then $b^2=-a^2-2c$, so using \eqref{14_211_cparam} we obtain 
\begin{equation}\label{14_211_bparam}b^2=\frac{2x(x^3+x^2+x-1)}{(x^2-1)^2}.
\end{equation}
Letting $y=a(x^2-1)$ and $z=b(x^2-1)$ we see that $(x,y,z)\in C(K)$ and $\phi(x,y,z)=(a,b,c)$. Furthermore, the relations \eqref{14_211_relations} imply that necessarily $x(x^2+1)(x^2+4x-1)\ne 0$.

 To see that $\phi$ is injective, one can verify that if $\phi(x,y,z)=(a,b,c)$, then
\[x=\frac{f_c(a)-1}{f_c^2(a)}\;\;,\;\;y=a(x^2-1)\;\;,\;\;z=b(x^2-1).\qedhere\]
\end{proof}

\begin{lem}\label{14_211_aux_curve_pts} Let $X/\Q$ be the hyperelliptic curve of genus 3 defined by the equation 
\[s^2= x(x^3 + x^2 + x - 1)(x^3 + x^2 - x + 1).\]
Then $X(\Q)=\{(0,0), (\pm 1, \pm 2), \infty\}$.
\end{lem}

\begin{proof}
The curve $X$ has an involution given by $(x,s)\mapsto (-1/x,-s/x^4)$; the quotient of $X$ by this involution is the curve $Y$ of genus 2 defined by the equation
 \[r^2=2u^5 + 2u^4 + 4u^3 + 3u^2 + 2u + 1,\] and the quotient map $X\to Y$ of degree 2 is given by
 \[u=\frac{x^2-1}{2x}\;,\;r=\frac{s(x^2+1)}{4x^3}.\]
 Using the Magma function \texttt{RankBound} we find that $\Jac(Y)(\Q)$ has rank 0, so we can easily obtain the rational points on $Y$ (for instance, using the function \texttt{Chabauty0}). Carrying out this calculation,  we obtain \[Y(\Q)=\{(0,\pm 1),\infty\}.\] Since $Y$ has three rational points, then $X$ can have at most six rational points, and we have already listed six points. 
\end{proof}

\begin{thm}\label{14_211_points} With $C$ as in Lemma \ref{14_211_curve} we have the following:
\begin{enumerate}
\item $C(\Q)=\{(\pm 1,\pm 2,\pm 2)\}$.
\item If $K$ is a quadratic field different from $\Q(\sqrt{2})$ and $\Q(\sqrt{17})$, then $C(K)=C(\Q)$.
\item For $K=\Q(\sqrt 2)$, $C(K)\setminus C(\Q)=\{(0,\pm\sqrt 2,0)\}$.
\item For $K=\Q(\sqrt{17})$, $C(K)\setminus C(\Q)=\{(x,\pm(4x+2),\pm(12x+2)): x^2-8x-1=0\}$.
\end{enumerate}
\end{thm}

\begin{proof} 
Let $C_1$ be the elliptic curve defined by the equation $y^2 = 2(x^3 + x^2 - x + 1)$. As shown in the proof of Theorem \ref{12_2_points}, 
\begin{equation}\label{14_211_C1pts}
C_1(\Q)=\{\infty,(\pm 1, \pm 2)\},
\end{equation} 
since $C_1$ is the modular curve $X_1(11)$. Now let $C_2$ be the hyperelliptic curve of genus 1 defined by the equation $z^2 = 2x(x^3 + x^2 + x - 1)$. Since $C_2$ has a rational point --- for instance, the point (0,0) --- it is an elliptic curve; in fact, $C_2$ is also the curve $X_1(11)$, as one can check that the map $(x,z)\mapsto(-1/x,z/x^2)$ defines a birational morphism from $C_2$ to $C_1$. We conclude that $C_2$ has exactly five rational points, and these are easily found:
\begin{equation}\label{14_211_C2pts}
C_2(\Q)=\{(0,0),(\pm 1,\pm 2)\}.
\end{equation}
From \eqref{14_211_C1pts} and \eqref{14_211_C2pts} we deduce, in particular, that $C(\Q)=\{(\pm 1,\pm 2,\pm 2)\}$.

Suppose now that $(x,y,z)\in C(\overline\Q)$ is a point with $[\Q(x,y,z):\Q]=2$, and let $K=\Q(x,y,z)$.

{\bf Case 1:} $x\in\Q$. We cannot have $x=\pm 1$, since this would imply that $y=\pm 2$ and $z=\pm 2$, contradicting the assumption that $(x,y,z)$ is a quadratic point on $C$. It follows that $y\notin \Q$, since having $x,y\in\Q$ would imply that $x=\pm 1$, by \eqref{14_211_C1pts}. If $z\in\Q$, then \eqref{14_211_C2pts} implies that $x=z=0$, and then $y^2= 2(x^3 + x^2 - x + 1)=2$. Thus, we obtain the quadratic points $(0,\pm\sqrt 2,0)$. 

Suppose now that $z\notin\Q$. Then $K=\Q(y)=\Q(z)$, so the rational numbers $ 2(x^3 + x^2 - x + 1)$ and $2x(x^3 + x^2 + x - 1)$ must have the same squarefree part; hence, their product is a square, so there is a rational number $s$ such that
\[s^2= x(x^3 + x^2 + x - 1)(x^3 + x^2 - x + 1).\]
Since $x\ne\pm 1$, Lemma \ref{14_211_aux_curve_pts} implies that $x=0$, so $z^2=2x(x^3 + x^2 + x - 1)=0$ and therefore $z=0\in\Q$, a contradiction. Hence, the case $z\notin\Q$ cannot occur.

Thus, we have shown that the only quadratic points on $C$ having a rational $x$-coordinate are the points $(0,\pm\sqrt 2,0)$.
 
{\bf Case 2:} $x$ is quadratic. Let $t^2+at+b\in\Q[t]$ be the minimal polynomial of $x$. By Lemma \ref{ellcrv_quad_pts} applied to the equation $y^2 = 2(x^3 + x^2 - x + 1)$, there exist a rational number $v$ and a point $(x_0,y_0)\in\{(\pm 1,\pm 2)\}$ such that
\begin{equation}\label{14_211_xpoly}a=\frac{2x_0-v^2+2}{2}\;\;,\;\;b=\frac{2x_0^2+v^2x_0+2x_0-2y_0v-2}{2}.\end{equation}
We now use the relation $z^2 = 2x(x^3 + x^2 + x - 1)$ to obtain a different expression for $a$ and $b$. Defining $p:=-1/x$ and $q:=z/x^2$ we have $q^2 = 2(p^3 + p^2 - p + 1)$ with $p$ quadratic, so again by Lemma \ref{ellcrv_quad_pts}, there exist a rational number $w$ and a point $(p_0,q_0)\in\{(\pm 1,\pm 2)\}$ such that the minimal polynomial of $p$ has the form $t^2+dt+e\in\Q[t]$, where
\begin{equation}\label{14_211_ppoly}d=\frac{2p_0-w^2+2}{2}\;\;,\;\; e=\frac{2p_0^2+w^2p_0+2p_0-2q_0w-2}{2}.
\end{equation}
Since $x$ and $p$ are related by $p=-1/x$, it is easy to see that $a=-d/e$ and $b=1/e$. Hence, we arrive at the system
	\begin{equation}\label{14_211_scheme}
		\begin{cases}
			& ae+d=0\\
			& be-1=0,
		\end{cases}
	\end{equation}
where $a,b,d,e$ are given by \eqref{14_211_xpoly} and \eqref{14_211_ppoly}. For each choice of the points $(x_0,y_0)$ and $(p_0,q_0)$, the system \eqref{14_211_scheme} defines a 0-dimensional scheme $S$ in $(v,w)$-plane over $\Q$, so its set of rational points can be determined. There are a total of 16 choices of pairs of points; for each pair we use the Magma function \texttt{RationalPoints} to determine all rational points on the corresponding scheme $S$, and we discard those rational points that lead to a polynomial $t^2+at+b$ that is reducible. The result of this computation is that in every case when a rational point on one of the 16 schemes $S$ leads to an irreducible polynomial $t^2+at+b$, this polynomial is $t^2-8t-1$; hence, this must be the minimal polynomial of $x$. Knowing that $x^2-8x-1=0$, the system \eqref{14_211_curve_equations} can now be solved to obtain $y=\pm(4x+2)$ and $z=\pm(12x+2)$.

We conclude that every quadratic point on $C$ having a quadratic $x$-coordinate is defined over the field $K=\Q(\sqrt{17})$ and has the form $(x,\pm(4x+2),\pm(12x+2))$ where $x^2-8x-1=0$. The theorem now follows immediately.
\end{proof}

\begin{cor}\label{14_211} Let $K$ be a quadratic field, and let $c\in K$. Suppose that $G(f_c,K)$ contains a graph of type {\rm 14(2,1,1)}. Then $c=-21/16$ and $K=\Q(\sqrt{17})$.
\end{cor}

\begin{proof} By Lemma \ref{14_211_curve}, there is a point $(x,y,z)\in C(K)$ with $x(x^2-1)\ne 0$ such that 
\begin{equation}\label{14_211_cvalue}c=-\frac{x^4+2x^3+2x^2-2x+1}{(x^2-1)^2}.
\end{equation} It follows from Theorem \ref{14_211_points} that $K=\Q(\sqrt{17})$ and that $x\in K$ satisfies $x^2-8x-1=0$. Using \eqref{14_211_cvalue} we then obtain $c=-21/16$. 
\end{proof} 


\subsection{Graph 14(3,1,1)}\label{14_311_section}
Our search described in \S\ref{quad_prep_comp} produced the pair \[ (K,c) = \left(\Q(\sqrt{33}), -\frac{29}{16}\right) \] for which the graph $G(f_c,K)$ is of type 14(3,1,1). We will show here that this is the only such pair $(K,c)$ consisting of a quadratic number field $K$ and an element $c\in K$.

\begin{lem}\label{14_311_curve}
Let $C/\Q$ be the affine curve of genus 9 defined by the equations
	\begin{equation}\label{14_311_curve_equations}
		\begin{cases}
			& y^2 = x^6 + 2x^5 + 5x^4 + 10x^3 + 10x^2 + 4x + 1\\
			& z^2 = x^6 - 2x^4 + 2x^3 + 5x^2 + 2x + 1.
		\end{cases}
	\end{equation}
Consider the rational map $\phi : C \dashrightarrow \mathbb{A}^3 = \Spec \Q[a,b,c]$ given by
	\[ a = \frac{y+x^2+x}{2x(x+1)}\ , \ b = \frac{z}{2x(x+1)}\ , \ c = -\frac{x^6+2x^5+4x^4+8x^3+9x^2+4x+1}{4x^2(x+1)^2}.\]
For every number field $K$, the map $\phi$ induces a bijection from the set
	\[ \{ (x,y,z) \in C(K) : x(x+1)(x^2+x+1)(x^3 + 2x^2 + x + 1) \ne 0 \} \]
to the set of all triples $(a,b,c) \in K^3$ such that $a$ is a fixed point for  the map $f_c$ and $b$ is a point of type $3_2$.
\end{lem}
\begin{figure}[h!]
\begin{center}\includegraphics[scale=0.5]{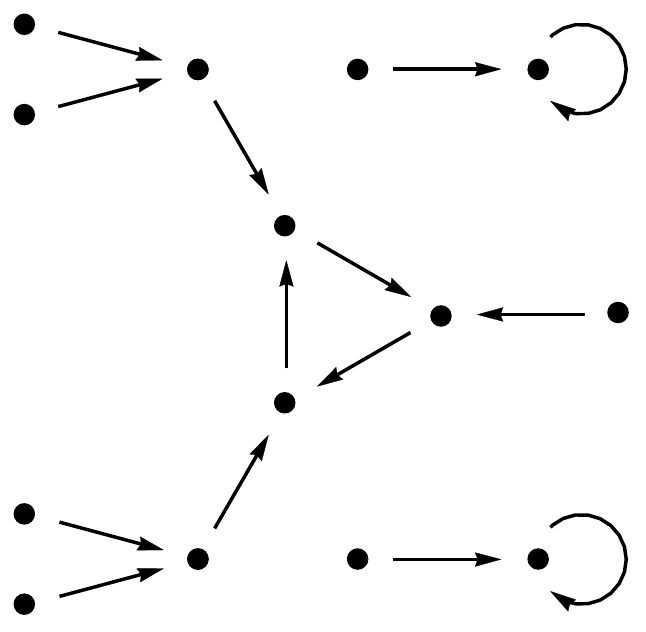}\end{center}
\caption{Graph type 14(3,1,1)}
\end{figure}
\begin{proof} Fix a number field $K$ and suppose that $(x,y,z)\in C(K)$ is a point satisfying $x(x+1)\ne 0$. Defining $a,b,c\in K$ as in the lemma, it is a routine calculation to verify that $f_c(a)=a$, $f_c^5(b)=f_c^2(b)$, and
\begin{equation}\label{14_311_relations}
f_c^2(b)-f_c^3(b)=\frac{x^2+x+1}{x+1}\;\;,\;\;f_c^4(b)-f_c(b)=\frac{x^3 + 2x^2 + x + 1}{x^2 + x}.
\end{equation}
It follows that if $(x^2+x+1)(x^3 + 2x^2 + x + 1)\ne 0$, then $b$ is a point of type $3_2$ for $f_c$. Hence, $\phi$ gives a well-defined map.

To see that $\phi$ is surjective, suppose that $a,b,c\in K$ satisfy the conditions of the lemma. Since $a$ is a fixed point for $f_c$, then by Proposition \ref{cycles_prop} there is an element $\rho\in K$ such that 
\begin{equation}\label{14_311_acvalues}a=\rho+1/2\;\;, \;\;c=1/4-\rho^2.\end{equation}
Since $b$ is a point of type $3_2$ for $f_c$, it follows from the discussion in \cite[23-24]{poonen_prep} that there is an element $x\in K\setminus\{-1,0\}$ such that
\begin{equation}\label{14_311_bcvalues} b^2=\frac{x^6-2x^4+2x^3+5x^2+2x+1}{4x^2(x+1)^2}\;\;,\;\;c = -\frac{x^6+2x^5+4x^4+8x^3+9x^2+4x+1}{4x^2(x+1)^2}.
\end{equation}
Letting $z=2x(x+1)b$ we then have the relation $z^2 = x^6 - 2x^4 + 2x^3 + 5x^2 + 2x + 1$. Equating the two expressions for $c$ given in \eqref{14_311_acvalues} and \eqref{14_311_bcvalues}, and defining $y=2x(x+1)\rho$, we obtain the equation $y^2 = x^6 + 2x^5 + 5x^4 + 10x^3 + 10x^2 + 4x + 1$. Note that 
\[a=\frac{1}{2}+\rho=\frac{1}{2}+\frac{y}{2x(x+1)}=\frac{y+x^2+x}{2x(x+1)}.\] 
Thus, we have a point $(x,y,z)\in C(K)$ such that $\phi(x,y,z)=(a,b,c)$ and $x(x+1)\ne 0$. Furthermore, \eqref{14_311_relations} implies that $(x^2+x+1)(x^3 + 2x^2 + x + 1)\ne 0$.

 To see that $\phi$ is injective, one can verify that if $\phi(x,y,z)=(a,b,c)$, then
\[x=f_c^2(b)-f_c(b)\;\;,\;\;y=x(x+1)(2a-1)\;\;,\;\;z=2x(x+1)b.\qedhere\]
\end{proof}

\begin{lem}\label{14_311_aux_curve_pts} Let $X/\Q$ be the hyperelliptic curve defined by the equation 
\[33v^2=x^6 - 2x^4 + 2x^3 + 5x^2 + 2x + 1.\]
Then $X(\Q)=\{(-2,\pm 1)\}$. 
\end{lem}

\begin{proof} Using the Magma functions \texttt{RankBounds} and \texttt{TorsionSubgroup} we find that $\Jac(X)(\Q)\cong\Z$; thus, we may apply the method of Chabauty and Coleman to bound the number of rational points on $X$. We will in fact be able to determine all rational points on $X$ by using the method outlined in \cite[\S 4.4]{bruin-stoll}, which is available in Magma via the function \texttt{Chabauty}. This function will use a Mordell-Weil sieve and Chabauty's method in order to compute $X(\Q)$. The required input is a rational point on $X$ and a generator for $J(\Q)$, where $J=\Jac(X)$. For purposes of working with Magma it will be convenient to use a slightly different equation for $X$, namely $y^2=33(x^6 - 2x^4 + 2x^3 + 5x^2 + 2x + 1)$; we will then show that the only rational points on this curve are the points $P_0=(-2,33)$ and $P_1=(-2,-33)$.

Let $D\in J(\Q)$ be the divisor class of $P_1-P_0$; we claim that $D$ generates $J(\Q)$. Assuming this for the moment, we use the \texttt{Chabauty} function with input $D\in J(\Q)$ and $P_1\in X(\Q)$ to obtain as output the complete list of rational points on $X$, which consists of the two points we have listed. 

In order to show that $D$ generates $J(\Q)$ we will use the method proposed by Stoll in \cite[\S 7]{stoll_height_const2}. Let $h$ and $\hat h$ denote the naive and canonical heights on $J(\Q)$, respectively. (See \cite[335]{flynn-smart} for the definitions.) Using Magma's \texttt{Height} function we find that $\hat h(D)=4.91745...$. In particular, $D$ has infinite order in $J(\Q)$. We claim that $D$ is not divisible in $J(\Q)$. Suppose that it is, and let $Q\in J(\Q)$ be a point such that $[n]Q=D$ for some $n>1$. Then $n^2\hat h(Q)=\hat h(D)=4.91745...$, so $\hat h(Q)<4.92/4=1.23$. The \texttt{HeightConstant} function applied to $J$ yields the number $c=3.34897...$ with the property that $h(P)-\hat h(P)\le c$ for every point $P\in J(\Q)$. In particular, we have $h(Q)< 1.23 + c < 4.58$. Using Stoll's \texttt{j-points} program, which is available in Magma via the function \texttt{Points}, we compute all points in $J(\Q)$ whose naive height does not exceed 4.6. The result is that there is only one such point, namely the trivial point. Thus, we conclude that $Q=0$, and hence $D=[n]Q=0$, a contradiction. This shows that $D$ is not divisible in $J(\Q)$, and since it has infinite order, it must therefore generate $J(\Q)$. This completes the proof.
\end{proof}

\begin{thm}\label{14_311_points} With $C$ as in Lemma \ref{14_311_curve} we have the following:
\begin{enumerate}
\item $C(\Q)=\{(0,\pm 1,\pm 1),(-1,\pm 1,\pm 1)\}$.
\item If $K$ is a quadratic field different from $\Q(\sqrt{33})$, then $C(K)=C(\Q)$.
\item For $K=\Q(\sqrt{33})$, $C(K)\setminus C(\Q)=\{(1,\pm\sqrt{33},\pm 3), (-2,\pm\sqrt{33},\pm\sqrt{33})\}$.
\end{enumerate}
\end{thm}

\begin{proof} As noted in \S\ref{modular_genus2_section}, the curve $y^2 = x^6 + 2x^5 + 5x^4 + 10x^3 + 10x^2 + 4x + 1$ is an affine model for the modular curve $X_1(18)$, which has exactly six rational points, namely $\infty^+,\infty^-,(0,\pm 1),(-1,\pm 1)$. The curve $z^2 = x^6 - 2x^4 + 2x^3 + 5x^2 + 2x + 1$ was studied by Poonen in \cite[\S4]{poonen_prep}, where it is shown that its only affine rational points are $(-1,\pm 1),(0,\pm 1)$, and $(1,\pm 3)$. Thus, we know all rational solutions to each of the defining equations of $C$, and we conclude that $C(\Q)=\{(0,\pm 1,\pm 1),(-1,\pm 1,\pm 1)\}$.

Suppose now that $(x,y,z)\in C(\overline\Q)$ is a point with $[\Q(x,y,z):\Q]=2$, and let $K=\Q(x,y,z)$.

{\bf Case 1:} $x$ is quadratic. In the terminology of \S\ref{quad_pts_on_crvs}, the point $(x,y)$ is then a non-obvious quadratic point on $X_1(18)$, so Theorem \ref{131618_quad} implies that $x$ is a primitive cube root of unity, and in particular $K=\Q(\sqrt{-3})$. However, in this case one can check that the equation $z^2 = x^6 - 2x^4 + 2x^3 + 5x^2 + 2x + 1$ has no solution $z\in K$. This is a contradiction, so we conclude that $x$ cannot be quadratic.

{\bf Case 2:} $x\in\Q$. We cannot have $y\in\Q$, since this would imply that $x\in\{-1,0\}$, which in turn implies that $z=\pm 1$, contradicting the assumption that $(x,y,z)$ is a quadratic point on $C$. If $z\in\Q$, then $x\in\{-1,0,1\}$; however $x$ cannot equal $-1$ or 0, since this would imply that $y=\pm 1\in\Q$. Therefore, if $z\in\Q$, then $x=1$. The system \eqref{14_311_curve_equations} can then be solved to obtain $y=\pm\sqrt{33}$, $z=\pm 3$. Thus, we have shown that the only quadratic points on $C$ having $z\in\Q$ are $(1,\pm\sqrt{33},\pm 3)$.

We assume henceforth that $z$ is quadratic. Define polynomials $f(t),g(t)\in\Q[t]$ by 
\begin{align*}
& f(t) = t^6 + 2t^5 + 5t^4 + 10t^3 + 10t^2 + 4t + 1,\\
& g(t) = t^6 - 2t^4 + 2t^3 + 5t^2 + 2t + 1.
\end{align*}
Since both $y$ and $z$ are quadratic, then $K=\Q(y)=\Q(z)$, so the rational numbers $f(x)$ and $g(x)$ must have the same squarefree part $d$; hence, there are rational numbers $u,v$ such that 
\begin{equation}\label{14_311_twists}
 \begin{cases}
	du^2 = f(x)\\
	dv^2 = g(x).
\end{cases}
\end{equation}
Suppose that $p$ is a prime number dividing $d$. The above equations imply that $u,v$, and $x$ all lie in the local ring $\Z_{(p)}$, so we may reduce \eqref{14_311_twists} modulo $p$ to obtain $f(x)\equiv g(x)\equiv 0\bmod p$. Hence, the polynomials $f(t)$ and $g(t)$ have a common root modulo $p$, so their resultant, which is $4521=3\cdot 11\cdot 137$, must be divisible by $p$. Therefore, $d$ can only be divisible by primes in the set $\{3,11,137\}$. Furthermore, the polynomial function $f:\R\to\R$ induced by $f(t)$ only takes positive values, so the first equation in \eqref{14_311_twists} implies that $d>0$. Finally, since $K=\Q(\sqrt d)$ is a quadratic field, we cannot have $d=1$. Therefore, we conclude that \[d\in\{3,11,137,33, 411,1507, 4521\}.\]
We proceed now to narrow down the possible values of $d$. For each number $d$ in the above set, we can check whether the hyperelliptic curves $du^2=f(x)$ and $dv^2=g(x)$ have points over all completions of $\Q$; the Magma function \texttt{HasPointsEverywhereLocally} can be used for this. Carrying out this computation we find that in all cases, except when $d=33$, at least one of these curves fails to have 2-adic or 3-adic points; in particular, for these values of $d$ the system \eqref{14_311_twists} cannot have a rational solution. Therefore, we must have $d=33$. It follows from \eqref{14_311_twists} that $(x,v)\in X(\Q)$, where $X$ is the curve defined in Lemma \ref{14_311_aux_curve_pts}, so the lemma implies that $x=-2$. The system \eqref{14_311_curve_equations} can then be solved to obtain $(x,y,z)=(-2,\pm\sqrt{33},\pm\sqrt{33})$. Thus, we have shown that the only quadratic points on $C$ are the points $(1,\pm\sqrt{33},\pm 3)$ and $(-2,\pm\sqrt{33},\pm\sqrt{33})$.
\end{proof}

\begin{cor} Let $K$ be a quadratic field and let $c\in K$. Suppose that $G(f_c,K)$ contains a graph of type {\rm 14(3,1,1)}. Then $c=-29/16$ and $K=\Q(\sqrt{33})$.
\end{cor}

\begin{proof} By Lemma \ref{14_311_curve} there is a point $P=(x,y,z)\in C(K)$ with $x(x+1)\ne 0$ such that 
\begin{equation}\label{14_311_cvalue}c = -\frac{x^6+2x^5+4x^4+8x^3+9x^2+4x+1}{4x^2(x+1)^2}.
\end{equation} 
Since $x\notin\{0,-1\}$, then $P$ cannot be a rational point on $C$. It follows from Theorem \ref{14_311_points} that $K=\Q(\sqrt{33})$ and that $x=1$ or $x=-2$. Substituting these values of $x$ into \eqref{14_311_cvalue} we obtain $c=-29/16$ in both cases.
\end{proof} 


\subsection{Graph 14(3,2)}\label{14_32_section}
Our search described in \S\ref{quad_prep_comp} found the pair
\[ (K,c) = \left(\Q(\sqrt{17}), -29/16\right) \]
for which the graph $G(f_c,K)$ is of type 14(3,2). We will show here that this is the only such pair $(K,c)$ consisting of a quadratic number field $K$ and an element $c\in K$.

\begin{lem}\label{14_32_curve}
Let $C/\Q$ be the affine curve of genus 9 defined by the equations
	\begin{equation}\label{14_32_curve_equations}
		\begin{cases}
			& y^2 = x^6 + 2x^5 + x^4 + 2x^3 + 6x^2 + 4x + 1\\
			& z^2 = x^6 - 2x^4 + 2x^3 + 5x^2 + 2x + 1.
		\end{cases}
	\end{equation}
Consider the rational map $\phi : C \dashrightarrow \mathbb{A}^3 = \Spec \Q[a,b,c]$ given by
	\[ a = \frac{y-x^2-x}{2x(x+1)}\ , \ b = \frac{z}{2x(x+1)}\ , \ c = -\frac{x^6+2x^5+4x^4+8x^3+9x^2+4x+1}{4x^2(x+1)^2}.\]
For every number field $K$, the map $\phi$ induces a bijection from the set
	\[ \{ (x,y,z) \in C(K) : xy(x+1)(x^2+x+1)(x^3 + 2x^2 + x + 1) \ne 0 \} \]
to the set of all triples $(a,b,c) \in K^3$ such that $a$ is a point of period 2 and $b$ is a point of type $3_2$ for  the map $f_c$.
\end{lem}
\begin{figure}[h!]
\begin{center}\includegraphics[scale=0.5]{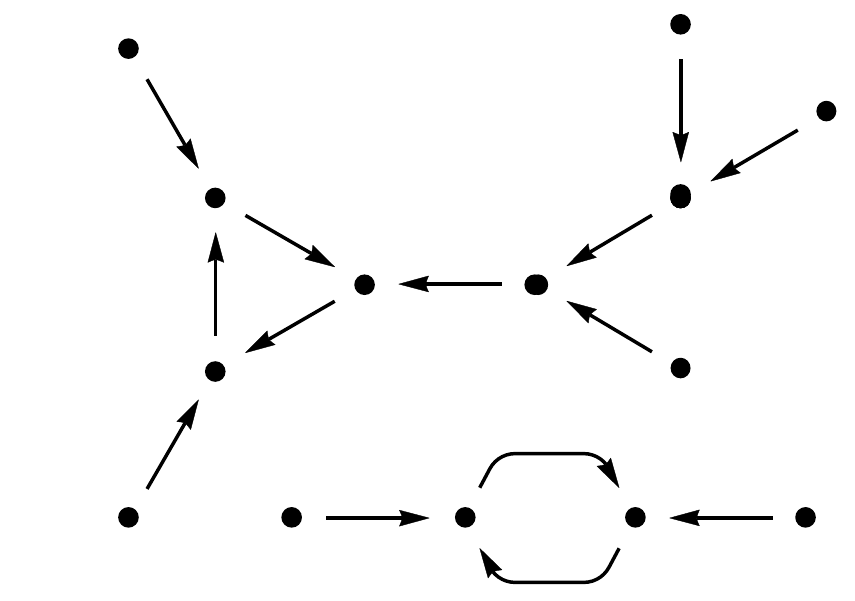}\end{center}
\caption{Graph type 14(3,2)}
\end{figure}
\begin{proof}
Fix a number field $K$ and suppose that $(x,y,z)\in C(K)$ is a point satisfying $x(x+1)\ne 0$. Defining $a,b,c\in K$ as in the lemma, it is a straightforward calculation to verify that $f_c^2(a)=a$, $f_c^5(b)=f_c^2(b)$, and 
\begin{equation}\label{14_32_relations}
a-f_c(a)=\frac{y}{x(x+1)}\;\;,\;\;f_c^2(b)-f_c^3(b)=\frac{x^2+x+1}{x+1}\;\;,\;\;f_c^4(b)-f_c(b)=\frac{x^3 + 2x^2 + x + 1}{x^2+x}.
\end{equation}
It follows from these relations that if $y(x^2+x+1)(x^3 + 2x^2 + x + 1)\ne 0$, then $a$ is a point of period 2 for $f_c$ and $b$ is a point of type $3_2$. Hence, $\phi$ gives a well-defined map.

To see that $\phi$ is surjective, suppose that $a,b,c\in K$ satisfy the conditions of the lemma. Since $a$ is a point of period 2 for $f_c$, Proposition \ref{cycles_prop} implies that there is an element $\sigma\in K$ such that
\begin{equation}\label{14_32_acvalues}a=\sigma-1/2\;\;, \;\;c=-3/4-\sigma^2.
\end{equation}
Since $b$ is a point of type $3_2$ for $f_c$, it follows from the discussion in \cite[23-24]{poonen_prep} that there is an element $x\in K\setminus\{-1,0\}$ such that
\begin{equation}\label{14_32_bcvalues} b^2=\frac{x^6-2x^4+2x^3+5x^2+2x+1}{4x^2(x+1)^2}\;\;,\;\;c = -\frac{x^6+2x^5+4x^4+8x^3+9x^2+4x+1}{4x^2(x+1)^2}.
\end{equation}
In particular, letting $z=2x(x+1)b$ we have the relation $z^2 = x^6 - 2x^4 + 2x^3 + 5x^2 + 2x + 1$. Equating the two expressions for $c$ given in \eqref{14_32_acvalues} and \eqref{14_32_bcvalues}, and defining $y=2x(x+1)\sigma$, we obtain the equation $y^2 = x^6 + 2x^5 + x^4 + 2x^3 + 6x^2 + 4x + 1$. Note that 
\[a=\sigma-\frac{1}{2}=\frac{y}{2x(x+1)}-\frac{1}{2}=\frac{y-x^2-x}{2x(x+1)}.\] Thus, we have a point $(x,y,z)\in C(K)$ such that $\phi(x,y,z)=(a,b,c)$ and $x(x+1)\ne 0$. Finally, the relations \eqref{14_32_relations} imply that $y(x^2+x+1)(x^3 + 2x^2 + x + 1)\ne 0$.

 To see that $\phi$ is injective, one can verify that if $\phi(x,y,z)=(a,b,c)$, then
\[x=f_c^2(b)-f_c(b)\;\;,\;\;y=x(x+1)(2a+1)\;\;,\;\;z=2x(x+1)b.\qedhere\]
\end{proof}

\begin{thm}\label{14_32_points} With $C$ as in Lemma \ref{14_32_curve} we have the following:
\begin{enumerate}
\item $C(\Q)=\{(0,\pm 1,\pm 1),(-1,\pm 1,\pm 1)\}$.
\item If $K$ is a quadratic field different from $\Q(\sqrt{17})$, then $C(K)=C(\Q)$.
\item For $K=\Q(\sqrt{17})$, $C(K)\setminus C(\Q)=\{(1,\pm\sqrt{17},\pm 3)\}$.
\end{enumerate}
\end{thm}

\begin{proof}
As noted in the proof of Theorem \ref{10_32_charac}, the only rational solutions to the equation $y^2 = x^6 + 2x^5 + x^4 + 2x^3 + 6x^2 + 4x + 1$ are $(0,\pm 1)$ and $(-1,\pm 1)$, since the hyperelliptic curve defined by this equation is the modular curve $X_1(13)$. The curve $z^2 = x^6 - 2x^4 + 2x^3 + 5x^2 + 2x + 1$ was studied in \cite[\S4]{poonen_prep}, where it is shown that its only affine rational points are $(-1,\pm 1), (0,\pm 1)$, and $(1,\pm 3)$. We conclude that $C(\Q)=\{(0,\pm 1,\pm 1),(-1,\pm 1,\pm 1)\}$.

Suppose now that $(x,y,z)\in C(\overline\Q)$ is a point with $[\Q(x,y,z):\Q]=2$, and let $K=\Q(x,y,z)$. We claim that $x$ must be rational. Indeed, if $x$ were quadratic, then $(x,y)$ would be a non-obvious quadratic point on $X_1(13)$, and such points do not exist, by Theorem \ref{131618_quad}; hence, $x\in\Q$. We cannot have $y\in\Q$, since this would imply that $x\in\{-1,0\}$, which in turn implies that $z=\pm 1$, thus contradicting the assumption that $(x,y,z)$ is a quadratic point on $C$; hence, $y\notin\Q$. 

{\bf Case 1:} $z\in\Q$. Since $x\in\Q$, then $x\in\{-1,0,1\}$. However, $x$ cannot equal $-1$ or 0 since this would imply that $y=\pm 1\in\Q$; therefore, $x=1$. The system \eqref{14_32_curve_equations} can then be solved to obtain $y=\pm\sqrt{17}$, $z=\pm 3$. Thus, we obtain the quadratic points $(1,\pm\sqrt{17},\pm 3)$.

{\bf Case 2:} $z$ is quadratic. Define polynomials $f(t),g(t)\in\Q[t]$ by 
\begin{align*}
& f(t) = t^6 + 2t^5 + t^4 + 2t^3 + 6t^2 + 4t + 1,\\
& g(t) = t^6 - 2t^4 + 2t^3 + 5t^2 + 2t + 1.
\end{align*}
Since both $y$ and $z$ are quadratic, then $K=\Q(y)=\Q(z)$, so the rational numbers $f(x)$ and $g(x)$ must have the same squarefree part $d$; hence, there are rational numbers $u,v$ such that 
\begin{equation}\label{14_32_twists}
 \begin{cases}
	du^2 = f(x)\\
	dv^2 = g(x).
\end{cases}
\end{equation}
The polynomial function $f:\R\to\R$ induced by $f(t)$ only takes positive values, so the first equation in \eqref{14_32_twists} implies that $d>0$. Moreover, since $K=\Q(\sqrt d)$ is a quadratic field, we cannot have $d=1$. Let $p$ be a prime number dividing $d$. The above equations imply that $u,v$, and $x$ all lie in the local ring $\Z_{(p)}$, so we may reduce \eqref{14_32_twists} modulo $p$ to obtain $f(x)\equiv g(x)\equiv 0\bmod p$. Hence, the polynomials $f(t)$ and $g(t)$ have a common root modulo $p$, so their resultant, which is $4321=29\cdot 149$, must be divisible by $p$. Therefore, $d$ can only be divisible by primes in the set $\{29,149\}$, and so \[d\in\{29,149,4321\}.\] For the values $d=29$ and 149 one can check that the hyperelliptic curves $du^2 = f(x)$ and $dv^2 = g(x)$ have no 2-adic point, and hence have no rational point; therefore, $d$ must equal 4321.

Let $X$ be the hyperelliptic curve defined by the equation $4321u^2=f(x)$. Using the Magma functions \texttt{RankBound} and \texttt{TorsionSubgroup} we find that $\Jac(X)(\Q)$ is trivial. Since $X$ has no rational Weierstrass point, the number of rational points on $X$ must be even. However, any two points in $X(\Q)$ would yield a nontrivial point in $\Jac(X)(\Q)$, so we must have $X(\Q)=\emptyset$. This is a contradiction since $(x,u)\in X(\Q)$. 

Since the assumption that $z$ is quadratic has led to a contradiction, we conclude that this case cannot occur, so the analysis done in the case $z\in\Q$ shows that the only quadratic points on $C$ are $(1,\pm\sqrt{17},\pm 3)$.
\end{proof}

\begin{cor}\label{14_32} Let $K$ be a quadratic field and let $c\in K$. Suppose that $G(f_c,K)$ contains a graph of type {\rm 14(3,2)}. Then $c=-29/16$ and $K=\Q(\sqrt{17})$.
\end{cor}

\begin{proof} By Lemma \ref{14_32_curve} there is a point $P=(x,y,z)\in C(K)$ with $x(x+1)\ne 0$ such that 
\begin{equation}\label{14_32_cvalue}c = -\frac{x^6+2x^5+4x^4+8x^3+9x^2+4x+1}{4x^2(x+1)^2}.
\end{equation} 
Since $x\notin \{0,-1\}$, Theorem \ref{14_32_points} implies that $P$ cannot be a rational point on $C$ and is therefore quadratic. It follows from Theorem \ref{14_32_points} that $K=\Q(\sqrt{17})$ and $P\in\{(1,\pm\sqrt{17},\pm 3)\}$. In particular, $x=1$, so \eqref{14_32_cvalue} yields $c=-29/16$.
\end{proof}


\section{Computation of preperiodic points}\label{data_gathering}
We explain in this section the method by which we gathered the data summarized in Appendices \ref{graph_pictures} and \ref{graph_data}. The idea is to systematically choose a collection of number fields of fixed degree, and a list of quadratic polynomials defined over these fields, and then compute the rational preperiodic points for all the chosen polynomials. More precisely, the steps we take for gathering data on preperiodic points over number fields of a fixed degree $n$ are the following:

\begin{enumerate}[itemsep = 1mm]
\item Choose a bound $D$ and find all number fields $K$ of degree $n$ whose discriminant satisfies $|\Delta_K|\leq D$. 
\item Choose a bound $B$ and for each number field $K$ from the previous step, find all elements $c\in K$ such that $H_K(c)\leq B$. Here, $H_K$ denotes the relative height function on $K$ (see \S\ref{prep_height_bound} below).
\item For each number field $K$, and for all $c\in K$ from the previous step, determine the set $\PrePer(f_c,K)$.
\end{enumerate} 

The list of fields from step (1) is known to be finite (see \cite[\S III.2]{neukirch}) and can be obtained for small $n$ and $D$, for instance, from Jones's online database \cite{jones}. For step (2) we use the main algorithm discussed in \cite{doyle-krumm} for listing elements of bounded height in a given number field. The method used to carry out step (3) will be developed in \S\ref{julia_sets} - \S\ref{prep_algorithm} below. Finally, in \S\ref{quad_prep_comp} we describe a specific computation done following the above steps in the case of quadratic fields.

\begin{rem}Hutz \cite{hutz_wiki,hutz} has designed a suite of algorithms for studying arithmetic dynamics in Sage, including an alternate approach to computing preperiodic points for morphisms of projective space over $\Q$. 
\end{rem}

\subsection{Filled Julia sets}\label{julia_sets}
In this section we include proofs of the theoretical results upon which the algorithms in \S\ref{prep_algorithm} are based. Let $K$ be a number field and let $c\in K$. The properties of the filled Julia sets of the map $f_c$ at places of $K$ can be used to deduce bounds on the heights of preperiodic points for $f_c$, as well as to create simple tests for eliminating many points $P\in K$ as possible elements of the set $\PrePer(f_c,K)$. In some cases these tests can even show that $f_c$ has no (finite) $K$-rational preperiodic point at all --- see Lemmas \ref{finiteobstructions} and \ref{infobstructions} in \textsection\ref{prep_algorithm}. Further discussion of filled Julia sets, along with proofs of some closely related results, may be found in \cite[\textsection 6]{call-goldstine}.

Let $M_K$ denote the set of nontrivial places of $K$, and let $M_K^{\infty}$ and $M_K^0$ denote the sets of archimedean (infinite) places and non-Archimedean (finite) places, respectively. For each $v \in M_K^{\infty}$ there is an embedding $\sigma_v : K \hookrightarrow \C$ such that $|x|_v = |\sigma_v(x)|_{\C}$,  where $|\;|_{\C}$ is the usual complex absolute value. For each $v \in M_K^0$ there is a maximal ideal $\p$ in the ring of integers $\O_K$ such that $|x|_v = (N(\p))^{-\ord_{\p}(x)/(e(\p)f(\p))}$, where $N(\p)$, $e(\p)$, and $f(\p)$ are the norm, ramification index, and residue degree of $\p$, respectively. When it is clear from the context which place $v$ is being considered, we write the corresponding absolute value simply as $|\;|$.

For a place $v \in M_K$, we denote by $K_v$ the completion of $K$ with respect to $v$, and we denote by $\C_v$ the completion of an algebraic closure of $K_v$. For $a \in \C_v$ and $r \in \R_{\ge 0}$, we denote by $D(a,r)$ the closed disk (in $\C_v$) of radius $r$ centered at $a$. For all subsequent uses of the notation $D(a,r)$, it will be clear which place $v$ is being considered.

For $c \in K_v$, we define the {\it filled Julia set} of $f_c$ at $v$ to be the set
\[ \K_{v,c} := \{x \in \C_v : \{f_c^n(x)\}_{n \ge 0} \mbox{ is a bounded set in } \C_v\}.\]
The following properties of $\K_{v,c}$ may be deduced immediately from the definition:

\begin{itemize}[itemsep = 1.5mm]
\item The filled Julia set $\K_{v,c}$ is totally invariant under $f_c$; that is, $f_c^{-1}(\K_{v,c}) = \K_{v,c} = f_c(\K_{v,c})$.
\item The set $\K_{v,c}$ is invariant under the map $z \mapsto -z$.
\item All of the preperiodic points for $f_c$ --- with the exception of the fixed point at $\infty$ --- lie in $\K_{v,c}$.
\end{itemize}

It is this third property which motivates our consideration of filled Julia sets, as it implies that 
\[ \PrePer(f_c,K) \subseteq \bigcap_{v \in M_K} (\K_{v,c}\cap K).\]

(Equality actually holds; the reverse containment may be proved using the theory of canonical heights.)
In particular, we can in many cases show that, for some $v \in M_K$, the $v$-adic filled Julia set of $f_c$ has an empty intersection with $K$, in which case we may conclude that $f_c$ has no $K$-rational preperiodic point.

We collect now a series of results giving precise bounds on filled Julia sets.

\subsection*{Infinite places} Recall that if $v$ is an infinite place of $K$, there is a corresponding embedding $\sigma_v:K\hookrightarrow\C$ which extends to an embedding $\sigma_v:K_v\hookrightarrow\C$. The image of $K_v$ is either $\R$ or $\C$, depending on whether $\sigma_v$ embeds $K$ into $\R$. For convenience of notation, throughout this section we will denote by $|\;|$ the absolute value corresponding to $v$, given by $|x|=|\sigma_v(x)|_{\C}$, and we identify $K_v$ with its image under $\sigma_v$.

	\begin{prop}\label{infjulia}
	Let $v$ be an infinite place of the number field $K$, and let $c \in K_v \subseteq \C$. Then
	\[ \K_{v,c} \subseteq D\left(0, \frac{1}{2} + \sqrt{\frac{1}{4} + |c|}\right). \]
	\end{prop}
	
	\begin{proof}
	Suppose $|x| > \frac{1}{2} + \sqrt{\frac{1}{4} + |c|}$, and choose $\epsilon > 0$ such that 
		\[
			|x| > \frac{1}{2} + \sqrt{\frac{1}{4} + |c| + \epsilon}.
		\] Then $|x|^2 - |c| > |x| + \epsilon$, so $|f_c(x)|> |x| + \epsilon$. By induction, we have $|f_c^n(x)| > |x| + n\epsilon$, so $|f_c^n(x)| \to \infty$. Therefore $x \not \in \K_{v,c}$.
	\end{proof}

	In the case that $K_v = \R$ it may happen that $\K_{v,c}$ does not intersect $K_v$. More precisely, we have the following trichotomy:
	\begin{prop}\label{realjulia}
	Let $v$ be a real place of the number field $K$, and suppose $c \in K_v = \R$.
	\begin{enumerate}[itemsep = 1.6mm]
	\item\label{item1: realjulia} If $c > \frac{1}{4}$, then $\K_{v,c} \cap \R = \emptyset$.
	\item\label{item2: realjulia} If $-2 \le c \le \frac{1}{4}$, then $\K_{v,c} \cap \R = [-a, a]$, where $a = \frac{1}{2} + \sqrt{\frac{1}{4} - c}$ is the larger of the two real fixed points of $f_c$.
	\item\label{item3: realjulia} If $c < -2$, then $-a-c \ge 0$, and
	\[ \K_{v,c} \cap \R \subseteq [-a, -\sqrt{-a-c}] \cup [\sqrt{-a-c}, a]. \]
	\end{enumerate}
	\end{prop}
	
	\begin{proof} 
		We start with \eqref{item1: realjulia}. For any $x \in \R$ we have
			\[ f_c(x) = \left(x - \frac{1}{2}\right)^2 + x + \left(c - \frac{1}{4}\right) \ge x + \left(c - \frac{1}{4}\right). \]
		It follows by induction that $f_c^n(x) \ge x + n\left(c - \frac{1}{4}\right)$, and since $c > \frac{1}{4}$ this grows without bound as $n \to \infty$.
		
		Next we prove \eqref{item2: realjulia}. Suppose $|x| > a$. Choosing $\epsilon > 0$ such that $|x| > \frac{1}{2} + \sqrt{\frac{1}{4} - c + \epsilon}\;$ we see that
			\begin{align*}
			|x| - \frac{1}{2} > \sqrt{\frac{1}{4} - c + \epsilon}
				&\implies x^2 - |x| + \frac{1}{4} > \frac{1}{4} - c + \epsilon\\
				&\implies f_c(x) = x^2 + c > |x| + \epsilon.
			\end{align*}
		By induction it follows that $f_c^n(x) > |x| + n\epsilon$ for all $n\geq 1$, so $|f_c^n(x)| \to \infty$. Therefore $\K_{v,c} \subseteq [-a,a]$.
		
		 Now suppose that $|x| \le a$. Then
		\begin{equation}\label{eq1}
		0 \le x^2 \le a^2 \implies c \le x^2 + c \le a^2 + c \implies c \le f_c(x) \le f_c(a) = a.
		\end{equation}
		We now claim that $-a \le c$. Indeed,
		\begin{align}\label{eq2}
	-2 \le c \le \frac{1}{4} &\implies 0 \le \frac{1}{4} - c \le \frac{9}{4} \notag \\
		&\implies \sqrt{\frac{1}{4} - c} \le \frac{3}{2} \\
		&\implies -1 \le \frac{1}{2} - \sqrt{\frac{1}{4} - c} = 1 - a \notag \\
		&\implies -a \le a - a^2 = c. \notag
		\end{align}
		Combining \eqref{eq1} and \eqref{eq2} shows that if $|x| \le a$, then $|f_c(x)| \le a$, and therefore $|f_c^n(x)| \le a$ for all $n \ge 1$. It follows that $[-a,a] \subseteq \K_{v,c} \cap \R$, which completes the proof of statement \eqref{item2: realjulia}.
		
		Finally, we prove \eqref{item3: realjulia}. If $|x| < \sqrt{-a-c}$, then $x^2 < -a - c$, and hence $f_c(x) < -a$. Therefore $|f_c(x)| > a$, and the argument used to prove statement \eqref{item2: realjulia} shows that $|f_c(x)| \to \infty$. 
	\end{proof}
	
\subsection*{Finite places} As in the case of infinite places, for finite places $v$ of $K$ we can bound the filled Julia set $\K_{v,c}$ and in some cases give obstructions to having $K_v$-rational points in it. 

	\begin{prop}\label{finitefj}
	Let $v$ be a finite place of the number field $K$, and let $c \in K_v$.
	\begin{enumerate}[itemsep = 1.6mm]
	\item\label{Item1:finitefj} If $|c| \le 1$, then $\K_{v,c} = D(0,1)$.
	\item\label{Item2:finitefj} If $|c| > 1$, then $\K_{v,c} \subseteq \left\{x \in \C_v : |x| = |c|^{1/2}\right\}$.
	\end{enumerate}
	\end{prop}
	
	\begin{proof}
	We begin with the proof of \eqref{Item1:finitefj}. If $|x| \le 1$, then $|f_c(x)| = |x^2 + c| \le \max\{|x|^2, |c|\} \le 1$. On the other hand, if $|x| > 1$, then $|f_c(x)| = |x|^2 > 1$. By induction we have $|f_c^n(x)| = |x|^{2^n} $, so $|f_c^n(x)| \to \infty$.
	
	For the proof of \eqref{Item2:finitefj}, if $|x| > |c|^{1/2}$, then $|f_c(x)| = |x|^2 > |c| > |c|^{1/2}$. By induction, $|f_c^n(x)| = |x|^{2^n}$, so $|f_c^n(x)| \to \infty$. If instead we have $|x| < |c|^{1/2}$, then $|f_c(x)| = |c| > |c|^{1/2}$, and we reduce to the previous case to show that $|f_c^n(x)| \to \infty$. 
	\end{proof}
	
	\begin{cor}\label{valuegroup}
	Let $v$ be a finite place of the number field $K$. Suppose $c \in K_v$ is such that $|c| > 1$ and $|c| \not \in |K_v^{\times}|^2$. Then $\K_{v,c} \cap K_v = \emptyset$.
	\end{cor}
	
	If we assume that $v$ is a finite place that lies above an odd prime in $\Z$, we can improve the statement of Corollary \ref{valuegroup} by weakening the hypothesis that $|c| \not \in |K_v^{\times}|^2$.
	
	\begin{prop}\label{oddprime}
	Let $v$ be a finite place of the number field $K$ corresponding to a maximal ideal $\p \subset \O_K$ that lies above an odd rational prime. Suppose $c \in K_v$ is such that $|c| > 1$ and $-c$ is not a square in $K_v$. Then $\K_{v,c} \cap K_v = \emptyset$.
	\end{prop}
	
	\begin{proof}
	Assume for the sake of a contradiction that there exists some $x \in \K_{v,c} \cap K_v$. The invariance of the filled Julia set implies that $f_c(x) \in \K_{v,c}$ as well, so by Proposition \ref{finitefj} we have $|x^2 + c| = |c|^{1/2}$, and therefore 
	\begin{equation}\label{eq4}
	\left| \frac{x^2}{-c} - 1 \right| = \frac{1}{|c|^{1/2}} < 1.
	\end{equation}
	Let $\pi$ be a uniformizer for the valuation ring $R_v \subset K_v$, and write $-c = \pi^r u$, where $r < 0$ is an integer and $u \in R_v^{\times}$. Similarly, write $x = \pi^s w$, with $s \in \Z$ and $w \in R_v^{\times}$. Then \eqref{eq4} implies that $\pi^{2s-r}\frac{w^2}{u} \in R_v^{\times}$, so that $r = 2s$. We can therefore rewrite \eqref{eq4} as $|w^2 - u| < 1$. Since $\p$ lies above an odd prime, it follows by Hensel's Lemma that $u$ must be a square in $K_v$, so $-c = \pi^{2s}u$ is a square as well.
	\end{proof}
	
	\begin{rem}
	The same statement does not necessarily hold for places $v$ lying above 2. For example, suppose $K = \Q$ and $c = -\frac{3}{4}$. Then $-c$ is not a square in $\Q_2$, but the map $f_c$ admits $\pm \frac{1}{2}$ as preperiodic points, and therefore the 2-adic filled Julia set of $f_c$ contains $\Q_2$-rational points.
	\end{rem}


\subsection{Height bounds on preperiodic points}\label{prep_height_bound}
Continue with the notation from \S\ref{julia_sets}. The {\it relative height} function on $K$ is the map $H_K:K\to\R_{\geq 1}$ defined by \[ H_K(x) = \prod_{v \in M_K} \max\{|x|_v, 1\}^{n_v}, \]
where $n_v = [K_v : \Q_v]$ is the local degree of $K$ at the place $v$. 

 It is a well-known fact (see \cite[\S3.1]{silverman_dynamics}) that for any bound $B$, the set $\{x\in K:H_K(x)\leq B\}$ is finite. Fix an element $c\in K$. The standard way of showing that the set $\PrePer(f_c,K)$ is finite is to prove the existence of a constant $B(K,c)$ such that for all points $P\in \PrePer(f_c,K)$, $H_K(P)\leq B(K,c)$. Since any set of points of bounded height is finite, this observation yields a finite search space for determining all preperiodic points for $f_c$. The standard technique using canonical heights yields the weak bound $B(K, c) = 2^{[K:\Q]} H_K(c)$. (See, e.g.,  the proofs of Theorems 3.11 and 3.20 in \cite{silverman_dynamics}.) Our analysis of filled Julia sets in the previous section allows us to improve this bound significantly, and therefore make the search space substantially smaller when the height of $c$ is large.


\begin{thm}\label{preperheight}
Let $K$ be a number field and let $c \in K$. For all points $P \in \PrePer(f_c,K)$ we have
\begin{equation}\label{eq:heightbound2}
H_K(P) \le \left(\frac{1 + \sqrt{5}}{2}\right)^{[K:\Q]} H_K(c)^{1/2}.
\end{equation}
\end{thm}

\begin{proof}
Let $P$ be a preperiodic point for $f_c$. Then $P \in \K_{v,c}$ for all places $v$ of $K$. Applying Proposition \ref{infjulia} to the infinite places of $K$ yields
\begin{equation}\label{infabsbound}
|P|_v \le \frac{1}{2} + \sqrt{\frac{1}{4} + |c|_v}
\end{equation}
for all $v \in M_K^{\infty}$. Similarly, for finite places $v$ we use Proposition \ref{finitefj} to obtain
\begin{equation}\label{finabsbound}
|P|_v \le 1 \mbox{\;\;if\;\;} |c|_v \le 1 \mbox{ \ \ \ and \ \ \ } |P|_v = |c|_v^{1/2} \mbox{\;\;if\;\;} |c|_v > 1
\end{equation}
for all $v \in M_K^0$. By combining \eqref{infabsbound} and \eqref{finabsbound}, we arrive at the following upper bound for $H_K(P)$:
\begin{align*}
H_K(P) &= \prod_{v \in M_K} \max\{|P|_v,1\}^{n_v} \\
&\le \prod_{v \in M_K^{\infty}} \left( \frac{1}{2} + \sqrt{\frac{1}{4} + |c|_v} \right)^{n_v} \cdot \prod_{v \in M_K^0} \max\{|c|_v^{1/2}, 1\}^{n_v}\\
&= H_K(c)^{1/2} \prod_{v \in M_K^{\infty}} \left( \frac{\frac{1}{2} + \sqrt{\frac{1}{4} + |c|_v}}{\max\{|c|_v^{1/2}, 1\}} \right)^{n_v}.
\end{align*}
The maximum value of the function
\[ y \mapsto \frac{\frac{1}{2} + \sqrt{\frac{1}{4} + y^2}}{\max\{y,1\}} \mbox{ \ \ , \ \ } y \in [0,\infty) \]
is $\frac{1 + \sqrt{5}}{2}$ (attained at $y = 1$). Therefore,
\[ H_K(P) \le H_K(c)^{1/2} \cdot \prod_{v \in M_K^{\infty}} \left(\frac{1 + \sqrt{5}}{2}\right)^{n_v} = \left( \frac{1 + \sqrt{5}}{2}\right)^{[K:\Q]} H_K(c)^{1/2}. \qedhere \]
\end{proof}


\subsection{Algorithm for computing preperiodic points}\label{prep_algorithm}
Using results from the previous sections we discuss here a method for computing the set of $K$-rational preperiodic points for a collection of maps $f_c$ defined over a number field $K$. We let $\O_K$ denote the ring of integers in $K$. For convenience in computations we will work with the valuations $v(\cdot) = \ord_{\p}(\cdot)$ rather than their corresponding non-archimedean absolute values $|\;|_v$. 

We begin by combining the statements of Propositions \ref{finitefj} and \ref{oddprime} and rephrasing them in terms of these valuations on $K$, using also the fact that preperiodic points belong to every filled Julia set of $f_c$.

\begin{lem}\label{finiteobstructions}
Let $K$ be a number field and let $P,c \in K$.
	\begin{enumerate}[itemsep = 1.2mm]
	\item If $\ord_{\p}(P) < 0 \le \ord_{\p}(c)$ for some maximal ideal $\p$ in $\O_K$, then $P$ is not preperiodic for $f_c$.
	\item If $\ord_{\p}(c) < 0$ and $\ord_{\p}(P) \ne \frac{1}{2} \ord_{\p}(c)$ for some maximal ideal $\p$ in $\O_K$, then $P$ is not preperiodic for $f_c$. In particular, if $\ord_{\p}(c)$ is negative and odd for some maximal ideal $\p$, then $f_c$ has no finite $K$-rational preperiodic point.
	\item If there exists a maximal ideal $\p$ in $\O_K$, lying above an odd rational prime, such that $\ord_{\p}(c) < 0$ and $-c$ is not a square in $K_{\p}$, then $f_c$ has no finite $K$-rational preperiodic point.
	\end{enumerate}
\end{lem}

We also record the dynamical consequences of Propositions \ref{infjulia} and \ref{realjulia}.

\begin{lem}\label{infobstructions}
Let $K$ be a number field and let $P,c \in K$.
	\begin{enumerate}[itemsep = 1.2mm]
	\item If $|\sigma(P)| > \frac{1}{2} + \sqrt{\frac{1}{4} + |\sigma(c)|}\;$ for some $\sigma: K \hookrightarrow \C$, then $P$ is not preperiodic for $f_c$.
	\item If $\sigma(c) > \frac{1}{4}\;$ for some $\sigma : K \hookrightarrow \R$, then $f_c$ has no finite $K$-rational preperiodic point.
	\item Suppose $\sigma$ is a real embedding such that $\sigma(c) \le \frac{1}{4}$, and set $a = \frac{1}{2} + \sqrt{\frac{1}{4} - \sigma(c)}$. If $|\sigma(P)| > a$, then $P$ is not preperiodic for $f_c$.
	\item Suppose $\sigma$ is a real embedding such that $\sigma(c) < -2$, and set $a = \frac{1}{2} + \sqrt{\frac{1}{4} - \sigma(c)}$. If \[\sigma(P) \not \in \left[-a,-\sqrt{-a-\sigma(c)}\right] \cup \left[\sqrt{-a-\sigma(c)},a\right],\] then $P$ is not preperiodic for $f_c$.
	\end{enumerate}
\end{lem}

\begin{rem}
In the case $\sigma(c) < -2$, we still have $\sigma(P) \in [-a, a]$ for all preperiodic points $P$. However, the set $[-a, -\sqrt{-a-\sigma(c)}] \cup [\sqrt{-a-\sigma(c)}, a]$ has length at most 4, making it much smaller than $[-a,a]$ for large values of $|\sigma(c)|$. Therefore, if $c \in K$ satisfies $\sigma(c) < -2$ for some real embedding $\sigma : K \hookrightarrow \R$, Lemma \ref{infobstructions} (4) provides a considerable reduction in our search space for $K$-rational preperiodic points for $f_c$. 

\end{rem}

We now give the details of a procedure for finding all the preperiodic points for a sample space of quadratic polynomials $f_c$ over an arbitrary number field $K$. To construct this sample space we fix a positive height bound $B$ and set
\[ \calC := \{c \in K : H_K(c) \le B\}. \]
We then take our sample space to be $\{f_c : c \in \calC\}$.

By Theorem \ref{preperheight}, any preperiodic point $P$ for a given map $f_c$ satisfies $H_K(P) \le \left(\frac{1 + \sqrt{5}}{2}\right)^{[K:\Q]} H_K(c)^{1/2}$. Therefore, the set
\[ \calP := \left\{P \in K \ : \ H_K(P) \le \left(\frac{1 + \sqrt{5}}{2}\right)^{[K:\Q]} B^{1/2} \right\} \]
contains all finite preperiodic points for \emph{all} of the maps in our sample space. The sets $\calC$ and $\calP$ can be computed using the main algorithm in \cite{doyle-krumm}. However, the output of this algorithm is modified in two ways:

\begin{itemize}[itemsep = 1.2mm]
\item We delete elements $c\in \calC$ for which we can check that $f_c$ has no $K$-rational preperiodic point, by performing the following steps:

\smallskip
\begin{enumerate}[itemsep = 1.2mm]
\item For each embedding $\sigma: K \hookrightarrow \R$, determine whether $\sigma(c) > \frac{1}{4}$. If this occurs for some $\sigma$, then $f_c$ has no preperiodic point (Lemma \ref{infobstructions}(2)), so we remove $c$ from $\calC$.
\item For each maximal ideal $\p$ such that $\ord_{\p}(c)<0$: if $\ord_{\p}(c)$ is odd, then $f_c$ has no preperiodic point (Lemma \ref{finiteobstructions}(2)), and $c$ is removed from $\calC$. If $\ord_{\p}(c)$ is even and $\p$ lies over an odd prime, determine whether $-c$ is a square in the completion $K_{\p}$. If it is not, then $f_c$ has no preperiodic point (Lemma \ref{finiteobstructions}(3)), so we remove $c$ from $\calC$.
\end{enumerate}

\smallskip
\noindent The last test can be done by working in the finite field $R_{\p} \big / \p R_{\p}$, where $R_{\p}$ is the valuation ring of $K_{\p}$: assuming that $\p$ is a prime of $\mathcal O_K$ lying over an odd rational prime, and that $\ord_{\p}(c)$ is negative and even, we let $\pi$ be a uniformizer in $R_{\p}$, and write $-c = \pi^{2s}u$ for some integer $s < 0$ and some unit $u \in R_{\p}^{\times}$. Then $-c$ is a square in $K_{\p}$ if and only if $u$ is a square in $R_{\p}$, which happens if and only if $u$ is a square modulo $\p R_{\p}$, by Hensel's Lemma.

\item We reorder the elements of $\calP$ by increasing height. 
For a given map $f_c$, every $K$-rational preperiodic point $P$ of $f_c$ satisfies $H_K(P) \le \left(\frac{1 + \sqrt{5}}{2}\right)^{[K:\Q]} H_K(c)^{1/2}$. Rather than searching through \emph{every} element $P \in \calP$ to determine whether $P$ is preperiodic for $f_c$, the ordering by height allows us to search through elements of $\calP$ until the height bound $\left(\frac{1 + \sqrt{5}}{2}\right)^{[K:\Q]} H_K(c)^{1/2}$ is exceeded, and then stop. 

\end{itemize}

Having modified the sets $\calC$ and $\calP$ as explained above, we proceed to determine, for every $c\in\calC$, the set of $K$-rational preperiodic points of the map $f_c$. 
First of all, we need not consider all points in $\calP$, but only those meeting the height bound of Theorem \ref{preperheight}. More importantly, Lemmas \ref{finiteobstructions} and \ref{infobstructions} give us a series of simple tests that can be applied to a point $P$ to conclude that it is not preperiodic. By doing this we quickly eliminate from consideration many elements from the list $\calP$, at which point the naive approach of iterating $f_c$ on each element of $\calP$ at most $\#\calP$ times becomes reasonable. This is, with small adjustments, the method provided in the algorithms below.

\begin{alg}[Preperiodicity tests]\label{ispreperiodic} \mbox{}\\
Input: A number field $K$, a number $c \in K$, and a point $P \in K$.\\
Output: The string "NO" or "MAYBE", depending on whether $P$ failed one of the tests and is therefore not preperiodic, or $P$ passed all of the tests and is therefore potentially preperiodic.

\begin{enumerate}[itemsep = 1.2mm]

\item For each prime ideal $\p$ such that $\ord_{\p}(P)<0$: if $\ord_{\p}(c) \geq 0$, return ``NO" (Lemma \ref{finiteobstructions}(1)).
\item For each prime ideal $\p$ such that $\ord_{\p}(c)<0$: if $\ord_{\p}(P) \ne \frac{1}{2}\ord_{\p}(c)$, return ``NO" (Lemma \ref{finiteobstructions}(2)).
\item For each embedding $\sigma: K \hookrightarrow \C$: if $|\sigma(P)| > \frac{1}{2} + \sqrt{\frac{1}{4} + |\sigma(c)|}$, return "NO" (Lemma \ref{infobstructions}(1)).
\item For each embedding $\sigma : K \hookrightarrow \R$: set $a = \frac{1}{2} + \sqrt{\frac{1}{4} - \sigma(c)}$.
	\begin{enumerate}[itemsep = 1.6mm]			
		\item If $|\sigma(P)| > a$, return ``NO" (Lemma \ref{infobstructions}(3)).
		\item If $\sigma(c) < -2$ and $|\sigma(P)| < \sqrt{-a-\sigma(c)}$, return ``NO" (Lemma \ref{infobstructions}(4)).
	\end{enumerate}		

\item If it has not yet been determined whether $P$ is preperiodic for $f_c$, return ``MAYBE".
\end{enumerate}
\end{alg}

The next algorithm goes through the list $\calC$ and, for each $c\in\calC$, uses Algorithm \ref{ispreperiodic} to find points from the list $\calP$ that are definitely not preperiodic for $f_c$. After this, a relatively small number of points remain whose preperiodicity has not been determined. Each of these is then classified by iterating a number of times which is not greater than the size of the list of undetermined points.

\begin{alg}[Determining all preperiodic points]\label{preperiodicpoints} \mbox{ \ }\\
	Input: A number field $K$, and lists $\calC$ and $\calP$ as described above.\\
	Output: A list whose elements are of the form $\{c, \{P_1,\ldots,P_m\}\}$ for all $c\in\calC$ such that $f_c$ has a $K$-rational preperiodic point, and $\{P_1,\ldots,P_m\} = \PrePer(f_c,K)$ 

\begin{enumerate}
	\item Create an empty list $\calL$.
	\item For each $c \in \calC$:
		\begin{enumerate}[itemsep = 1.5mm]
		\item Create two empty lists, $Y$ and $M$.
		\item For all $P\in\calP$ with $H_K(P) \le \left(\frac{1 + \sqrt{5}}{2}\right)^{[K:\Q]}H_K(c)^{1/2}$ : apply Algorithm \ref{ispreperiodic} to $P$ and $c$. If the output is ``MAYBE", append the point $P$ to the list $M$.
		\item While $M$ is nonempty:
			\begin{enumerate}[itemsep = 1.5mm]
			\item Let $P$ be the first element of $M$.
			\item Construct a list $\calI=\{P\}$ to contain iterates of $P$, and set $Q := f_c(P)$.
			\item While preperiodicity of $P$ has not been determined:
				\begin{enumerate}[itemsep = 1.5mm]
					\item if $Q\in Y$, then $P$ is preperiodic; append all elements of $\calI$ to the list $Y$ and delete them from $M$.
					\item If $Q \in \calI$, then $P$ is preperiodic; append all elements of $\calI$ to the list $Y$ and delete them from $M$.
				\item If $Q \not \in M$, then $P$ is not preperiodic; delete all elements of $\calI$ from $M$.
				\item If no action was taken in steps (A) through (C), append $Q$ to $\calI$, and replace $Q$ with $f_c(Q)$.
				\end{enumerate}
			\end{enumerate}			
		\item If $Y$ is nonempty, append $\{c,Y\}$ to the list $\calL$.	
		\end{enumerate}
	\item Return $\calL$.	
\end{enumerate}
\end{alg}

\begin{rem} The while loop in step 2c (iii) will necessarily terminate in at most $\# M$ steps because, for a given $P\in M$ the first $\#M$ iterates (together with $P$) determine preperiodicity: if $P$ is preperiodic, then either one of these iterates falls in $Y$ (and then the loop terminates in step A), or they all lie in $M$; but then at least two of them must be equal (and the loop terminates in step B). Similarly, if $P$ is not preperiodic, then none of these iterates lies in $Y$, and since they are all distinct, at least one is not in $M$ (hence the loop will terminate in step C). 
\end{rem}


\subsection{Application to quadratic fields}\label{quad_prep_comp}
The algorithms developed in the previous section give us a systematic way of computing preperiodic points for a sample space of quadratic polynomials defined over a given number field. In this section we apply those algorithms in the context of quadratic fields, following --- with minor modifications ---  the steps outlined at the beginning of \S\ref{data_gathering}.

For every quadratic field $K$ with discriminant satisfying $|\Delta_K|\leq 327$ (a total of 200 fields) we computed the set $\PrePer(f_c,K)$ for a collection of elements $c\in K$ chosen as follows: a bound $B_K$ was fixed, and using the main algorithm in \cite{doyle-krumm} we computed all elements $c\in K$ satisfying $H_K(c)\leq B_K$. Rather than choosing the same bound $B_K$ for every field $K$, we varied $B_K$ between 1,000 and 2,200, increasing together with the discriminant of $K$ in order to keep the quantity of $c$ values considered roughly uniform as $K$ varied. In addition we computed, for each field $K$, the set $\PrePer(f_c,K)$ for every $c\in\Q$ of {\it relative} height $\leq B_K'$, where $B_K'$ varied between $300^2$ and $600^2$. 

In total, the set $\PrePer(f_c,K)$ was computed for $256,588$ pairs $(K,c)$; this count does not include millions of values of $c$ for which it was determined by local methods that $f_c$ has no $K$-rational preperiodic point (other than the point at infinity). The resulting $256,588$ directed graphs were then classified by isomorphism to arrive at a list of 45 graphs. These graphs --- and one additional graph --- are displayed in Appendix~\ref{graph_pictures}, and a representative pair $(K,f_c)$ for each graph is given in Appendix \ref{graph_data}. 

There are three items appearing in the appendices which were not found by the computations described above: the graph 10(3,1,1); an example of the graph 10(2) over a real quadratic field; and an example of the graph 8(4) over an imaginary quadratic field. These examples occur over number fields of large enough discriminant, or for polynomials $f_c$ with the height of $c$ large enough, that they are beyond the search range used for our main computation. For more information we refer the reader to the discussion of these graphs in \S\ref{classification}.


\section{PCF maps and maps with a unique fixed point}\label{pcf_ufp}

In this section we determine all rational and quadratic algebraic numbers $c$ such that the polynomial $f_c$ has one of two properties: either it is postcritically finite, or it has a unique fixed point. Furthermore, for every such number $c$ we show that all possible graphs $G(f_c,K)$, with $K$ a quadratic field, appear in Appendix \ref{graph_pictures}. Recall that a rational function $\phi \in K(z)$ is called \textit{postcritically finite} (or PCF) if all of its critical points are preperiodic for $\phi$. In particular, a polynomial of the form $f(z) = z^2 + c$ is PCF if and only if 0 is preperiodic for $f(z)$.

We will need two preliminary results concerning heights of algebraic numbers. Recall that the \textit{absolute height} function $H:\qbar\to\R_{\geq 1}$ is defined by \[H(\alpha)=H_K(\alpha)^{1/[K:\Q]},\] where $K$ is any number field containing $\alpha$. Here $H_K$ is the relative height function on $K$ defined in \S\ref{prep_height_bound}. The following result is an immediate consequence of Theorem \ref{preperheight}.

\begin{lem}\label{prep_height_bound_lemma}
Let $K$ be a quadratic number field and let $c\in K$. If $x\in K$ is preperiodic for the map $f(z) = z^2 + c$, then
\[H(x) \leq \left(\frac{1+\sqrt{5}}{2}\right) \ H(c)^{1/2}.\]
\end{lem}

\begin{lem}\label{quad_bdd_height} Let $x\in\qbar$ be a quadratic algebraic integer such that $H(x)\leq B$. Then $x$ satisfies an equation \[x^2+a_1x+a_0=0,\] where $a_0, a_1$ are integers with $|a_0|\leq B^2$ and $|a_1|\leq 2B^2$.
\end{lem}

\begin{proof} Let $K=\Q(x)$. The minimal polynomial of $x$ has the form $m(t)=t^2+a_1t+a_0$ for some integers $a_0,a_1$. Factor $m(t)$ over $K$ as $m(t)=(t-x)(t-y)$. For every place $v$ of $K$ we have \[|a_0|_v=|x|_v|y|_v\leq\max(1,|x|_v)\cdot\max(1,|y|_v),\] so \[|a_0|^2=H_K(a_0)=\prod_{v|\infty}\max(1,|a_0|_v^{n_v})\leq\prod_{v|\infty}\max(1,|x|_v^{n_v})\cdot\max(1,|y|_v^{n_v})=H_K(x)H_K(y)\leq B^4\] and therefore $|a_0|\leq B^2$. Similarly, for every place $v$ we have \[|a_1|_v\leq|x|_v+|y|_v\leq2\cdot\max(|x|_v,|y|_v)\leq 2\cdot\max(1,|x|_v)\cdot\max(1,|y|_v),\] so \[|a_1|^2=H_K(a_1)=\prod_{v|\infty}\max(1,|a_1|_v^{n_v})\leq\prod_{v|\infty}2^{n_v}\cdot\max(1,|x|_v^{n_v})\cdot\max(1,|y|_v^{n_v})=2^2H_K(x)H_K(y),\] and therefore $|a_1|\leq 2B^2$.
\end{proof}

\subsection{PCF maps}\label{pcf} Using the above results, we now determine all rational and quadratic numbers $c$ such that the map $f_c$ is PCF. As a result of our analysis, we will be able to list all pairs $(K,c)$ consisting of a quadratic number field $K$ and an element $c\in K$ for which the graph $G(f_c,K)$ has an odd number of vertices.

\begin{prop}\label{pcf_maps_prop} Let $c\in\qbar$ satisfy $[\Q(c):\Q]\le 2$, and suppose that the map $f(z)=z^2+c$ is PCF. Then $c\in\{0,-1,-2,\pm i\}$, where $i$ denotes a square root of $-1$. Moreover, for every quadratic field $K$ containing $c$, the graph $G(f,K)$ is isomorphic to one of the graphs listed in Appendix \ref{graph_pictures}.
\end{prop}
\begin{proof} Let $K$ be a quadratic number field containing $c$. Since 0 is preperiodic for $f$, then also $c = f(0)$ is preperiodic. Hence, $c$ satisfies an equation $f^m(c) = f^n(c)$ for some $m < n$, so $c$ is integral over $\Z$. In particular, $|c|_v \leq 1$ for every non-Archimedean place $v$ of $K$. As $0$ is preperiodic and therefore in the filled Julia set $\mathcal{K}_{v,c}$ for each Archimedean place $v$, we find $|c|_v \leq 2$ (Lemma~\ref{infobstructions}(2,4)).

Hence, we have
\begin{equation}\label{odd_height_bound}
H(c) = \left(\prod_{v \mid \infty} \max\{1, |c|_v\}^{n_v} \right)^{1/2} \leq 2. 
\end{equation}

Moreover, since 0 is preperiodic for $f(z)$, then by Lemma \ref{prep_height_bound_lemma} we have the following for all $n\geq 0$:
\begin{equation}\label{0iterate} H\left(f^n(0)\right) \leq \left( \frac{1+\sqrt{5}}{2}\right) H(c)^{1/2}<2.29.\end{equation}

If $c\in\Q$, then \eqref{odd_height_bound} implies $c\in\{0,\pm 1,\pm 2\}$ and \eqref{0iterate} with $n=3$ eliminates $c=1$ and 2, so $c\in\{0,-1,-2\}$. If $c$ is quadratic, we claim that $c=\pm i$. Using Lemma \ref{quad_bdd_height} we compute the set $S$ of all quadratic algebraic integers of height at most 2. For every element $c\in S$ we then check whether \eqref{0iterate} is satisfied for $n=5$ in order to eliminate values of $c$. After this process, the only possibilities left for $c$ are $\pm i$; this proves the first part of the proposition.

For the second statement, we begin by computing the full list of rational and quadratic preperiodic points for the map $f(z)=z^2+c$ for each $c \in \{0, -1, -2, \pm i \}$. Note that, since $c$ is an algebraic integer, every preperiodic point for $f(z)$ must also be an algebraic integer.

Consider first the case $c=i$. Since $c\in K$, we must have $K=\Q(i)$. By Lemma \ref{prep_height_bound_lemma}, every element $x\in\PrePer(f,K)$ must satisfy $H(x)< 1.62$. Writing $x=a+bi$ with $a,b\in\Z$, we then have $a^2+b^2<2.63$, so $a,b\in\{0,\pm 1\}$. This shows that $\PrePer(f,K)\subseteq \{0,\pm 1,\pm i, \pm 1\pm i\}$. The numbers $0,\pm i, \pm(1-i)$ are easily seen to be preperiodic; the remaining numbers, namely $\pm 1$ and $\pm (1+i)$, are not preperiodic because their orbits under $f$ are not contained in the set $\{0,\pm 1,\pm i, \pm 1\pm i\}$. Therefore, we conclude that \[\PrePer(z^2+i,\Q(i))=\{ 0,\pm i, \pm(1-i)\}.\] One can easily check that the corresponding graph $G(f,K)$ appears in the appendix with the label 5(2)a.

We consider now the values $c=0,-1,-2$. Let $\PrePer_2(f)$ denote the set of all preperiodic points for $f$ in $\qbar$ of degree at most 2 over $\Q$. Since $c$ has very small height, we can quickly compute the set $\PrePer_2(f)$ as follows: first, we use Lemma \ref{quad_bdd_height} to compute the set $S$ of all algebraic integers $x$ satisfying the height bound in Lemma \ref{prep_height_bound_lemma}, so that $\PrePer_2(f)\subseteq S$. We now iterate each element $s\in S$ until a value in its orbit is repeated, in which case $s$ is preperiodic, or its orbit leaves the set $S$, in which case it is not preperiodic. This will require at most $1+\#S$ iterations of $s$. The results of these calculations are summarized in the table below. Here, $\omega$ denotes a primitive cube root of unity.
\begin{center}
	\begin{tabular}[h]{c|c}
		$c$ & $x \in \qbar$ is preperiodic for $f$ and $[\Q(x) : \Q] \leq 2$ \\
		\hline 
		\hline
		$0$ & $\displaystyle 0, \pm 1, \pm i, \pm \omega, \pm \omega^2$ \\
		\hline
		$-1$ & $ 
			0, \ \pm 1, \ \pm \sqrt{2}, \ \pm \frac{1}{2}(1+\sqrt{5}), 
				\ \pm \frac{1}{2} (1 - \sqrt{5})$\\
		\hline
		$-2$ & $ 0, \ \pm 1, \ \pm 2, \ \pm \sqrt{2}, \ \pm \sqrt{3}, 
			\ \pm \frac{1}{2}(1+\sqrt{5}), \ \pm \frac{1}{2}(1-\sqrt{5})$ \\
	\end{tabular}
\end{center}

For each of the above values of $c$ we can now determine the structure of the graphs $G(f,K)$ as $K$ ranges over all quadratic fields. In every case, we find that the graph is one of those listed in the appendix; the corresponding labels are given below.

For $c=0$ we obtain 
\[G(f_c,K)=
\begin{cases}
\text{5(1,1)b} & \text{if\;} K=\Q(i)\\
\text{7(2,1,1)a} & \text{if\;}K=\Q(\omega)\\
\text{3(1,1)} &\text{otherwise}.
\end{cases}\]

For $c=-1:$ 
\[G(f_c,K)=
\begin{cases}
\text{5(2)b} & \text{if\;} K=\Q(\sqrt 2)\\
\text{7(2,1,1)b} & \text{if\;}K=\Q(\sqrt 5)\\
\text{3(2)} &\text{otherwise}.
\end{cases}\]

For $c=-2:$ 
\[G(f_c,K)=
\begin{cases}
\text{7(1,1)a} & \text{if\;} K=\Q(\sqrt 2)\\
\text{7(1,1)b} & \text{if\;}K=\Q(\sqrt 3)\\
\text{9(2,1,1)} & \text{if\;}K=\Q(\sqrt 5)\\
\text{5(1,1)a} &\text{otherwise}.
\end{cases}\]

This completes the proof of the proposition.
\end{proof}

\begin{cor}\label{odd_vertices_cor} Let $K$ be a quadratic field and let $c\in K$. If the graph $G(f_c,K)$ has an odd number of vertices, then it is isomorphic to one of the graphs listed in Appendix \ref{graph_pictures}.
\end{cor}

\begin{proof} Note that an element $x\in K$ is preperiodic for $f_c$ if and only if $-x$ is preperiodic. Hence, the preperiodic points for $f_c$ in $K$ can be grouped into pairs $(x,-x)$. If the number of such points is odd, this means that there is some point $x$ such that $x=-x$; in other words, 0 is preperiodic for $f_c$. By definition, this means that $f_c$ is PCF, and now the conclusion follows from Proposition \ref{pcf_maps_prop}.
\end{proof}

\begin{rem} Proposition \ref{pcf_maps_prop} and its proof achieve, for graphs $G$ with an odd number of vertices, our main classification goal stated in \S\ref{classification}; namely, an explicit description of all pairs $(K,c)$ such that $G(f_c,K)\cong G$.
\end{rem}

\subsection{Maps with a unique fixed point}\label{ufp}
Having classified the preperiodic graph structures with an odd number of vertices, we consider now the three graphs in our list that have a unique 1-cycle. The following result will give, for each one of these graphs, a complete list of all pairs $(K,c)$ giving rise to the graph.

\begin{prop}\label{ufp_prop} Let $K$ be a quadratic field, and suppose that $c \in K$ is such that the map $f(z)=z^2+c$ has a unique fixed point in $K$. Then $G(f,K)$ is one of the graphs listed in Appendix \ref{graph_pictures}. More precisely, $c = 1/4\;$ and:
\[G(f,K)=
\begin{cases}
\mathrm{6(2,1)} & \mathrm{if\;} K=\Q(i)\\
\mathrm{4(1)} & \mathrm{if\;}K=\Q(\sqrt{-3})\\
\mathrm{2(1)} &\mathrm{otherwise}.
\end{cases}\]
\end{prop}

\begin{proof} If $f(z)$ has a unique fixed point, then the equation $z^2-z+ c=0$ has exactly one solution, so $c=1/4$. The set $\PrePer(f,\Q)$ is easily seen to be equal to $\{\pm\frac{1}{2}\}$: by Lemma \ref{prep_height_bound_lemma}, any rational preperiodic point for $f$ has height at most 3; this gives a short list of possible preperiodic points, and the preperiodicity of each point is decided within a few iterations. The corresponding graph $G(f,\Q)$ is the graph 2(1). We will now determine all the quadratic fields $K$ where $f$ might have preperiodic points that are not rational.

Suppose that $x\in\qbar$ is a quadratic preperiodic point for $f$, and let $K=\Q(x)$. It follows from Lemma \ref{finiteobstructions} that for every non-Archimedean place $v$ of $K$ we have
\begin{equation*}\label{x_abs}|x|_v
\begin{cases}
=2 & \mathrm{if\;} v|2\\
\leq 1 & \mathrm{otherwise}.
\end{cases}\end{equation*}
Hence, $2x$ is an algebraic integer. Since $x$ is preperiodic, Lemma~\ref{infobstructions}(1) gives
	\[
		|x|_v \leq \frac{1}{2} + \sqrt{\frac{1}{4} + \frac{1}{4}} < 1.21
	\]
for each $v \mid \infty$. Hence
	\begin{equation}\label{Height2x}
		H(2x) = \left(\prod_{v \mid \infty} \max \left\{ 1, |2x|_v\right\}^{n_v} \right)^{1/2}
			\leq \left(\prod_{v \mid \infty} 2^{n_v} \cdot \max \left\{ 1, |x|_v\right\}^{n_v} \right)^{1/2}< 2.42.		
	\end{equation}

Write $m(t) =t^2 + a_1t + a_0$ for the minimal polynomial of $2x$, where $a_0, a_1 \in \Z$. Applying Lemma~\ref{quad_bdd_height} with the bound in \eqref{Height2x}, we find that
	\[
		|a_1| \leq 2(2.42)^2 < 11.7 \ \;\;\;\text{and}\;\;\; \ |a_0| \leq (2.42)^2 < 5.9 .
	\]
Since all iterates of $x$ under $f$ are preperiodic, we have $H(f^n(x))<3.24$ for all $n$ (Lemma \ref{prep_height_bound_lemma}). 
Using Sage we consider all quadratic numbers $x$ satisfying these properties (with $n\leq3$), and find that the only possibilities have $a_1=0,a_0=3$ and $a_1=\pm 2, a_0=5$. This leads to the preperiodic points \[\pm \frac{\sqrt{-3}}{2} ,\;\; \pm \frac{1}{2} \pm i,\] and the result now follows easily.
\end{proof}


\appendix

\section{Proof of Lemma \ref{poly_lemma}}\label{counting_appendix}

Let $H:\Q\to\R_{\geq 1}$ be the standard multiplicative height function on $\Q$ given by $H(a/b)=\max\{|a|,|b|\}$ if $a$ and $b$ are coprime integers. Given real numbers $\alpha<\beta$, define a counting function on $\R_{\ge 1}$ by 
\[N(T;\alpha,\beta)=\#\{r\in\Q^{\ast}: \alpha\le r\le\beta\;\mathrm{and}\;H(r)\le T\}.\]
To prove Lemma \ref{poly_lemma} we will need the following asymptotic estimate for $N(T;\alpha,\beta)$.
\begin{prop}\label{counting_asymptotic} There is a constant $c=c(\alpha,\beta)$ such that
\[N(T;\alpha,\beta)\sim cT^2.\]
\end{prop}
\begin{proof}
We begin by making a series of reductions. The relation $H(r)=H(-r)$ for $r\in\Q$ implies that
\[
N(T; \alpha, \beta) = 
\begin{cases}
N(T; -\beta, -\alpha) & \text{if $\alpha < \beta \le 0$} \\
N(T; 0, -\alpha) + N(T; 0, \beta)  & \text{if $\alpha < 0 < \beta$}.
\end{cases}
\]
This shows that in order to prove the result it suffices to consider the case $0\le\alpha$. Making this assumption, the relation $H(r)=H(1/r)$ for $r\in\Q^{\ast}$ implies that
\[
N(T; \alpha, \beta) = 
\begin{cases}
N(T; \beta^{-1}, \alpha^{-1}) & \text{if $1< \alpha < \beta$} \\
N(T; \alpha,1) + N(T; \beta^{-1},1)  & \text{if $\alpha \leq 1 < \beta$}.
\end{cases}
\]
Therefore, it suffices to prove the result when $0\le\alpha<\beta\le 1$. Finally, making this assumption we have $N(T; \alpha, \beta)=N(T; 0, \beta)$ if $\alpha=0$, and otherwise $N(T; \alpha, \beta) = N(T; 0, \beta)-N(T; 0,\alpha) +\varepsilon$, where $\varepsilon\in\{0,1\}$; hence, it suffices to prove that for any real number $\eta\in(0,1]$ there is a constant $c=c(\eta)$ such that $N(T; 0, \eta)\sim cT^2$. The following estimate will be needed to prove this.

\begin{lem}\label{S_bounds_lemma} Let $\eta\in(0,1]$ be a real number, and define a function $S:[1,\infty)\to\Z_{\ge 0}$ by
\[S(X)=\#\{(x,y)\in\Z^2: 1\le y\le X \text{\;and\;} 1\le x \le \eta y \}.\]
Then, setting $g(X)=\eta X(X+1)/2 -S(X)$, we have $0\le g(X)< 2X$ for every $X\ge 1$.
\end{lem}

\begin{proof} Note that
\begin{equation}\label{floor_expansion}
S(X)=\sum_{1\le y\le X}\#\{x\in\Z:1\le x\le \eta y\}=\sum_{1\le y\le X}\lfloor\eta y\rfloor.
\end{equation}

Applying the trivial bound $\lfloor \eta y\rfloor\le \eta y$ we obtain $S(X)\le\eta X(X+1)/2$, and the inequality $g(X)\ge 0$ follows immediately. To show that $g(X)<2X$, write $\eta y=\lfloor \eta y\rfloor +\{\eta y\}$ with $0\le \{\eta y\}<1$. Then \eqref{floor_expansion} yields
\[S(X)=\sum_{1\le y\le X}\eta y \;-\sum_{1\le y\le X}\{\eta y\}>\eta\cdot\sum_{1\le y\le X}y\;-\sum_{1\le y\le X} 1=\eta\lfloor X\rfloor(\lfloor X\rfloor+1)/2-\lfloor X\rfloor.\]
Using the inequalities $X-1<\lfloor X\rfloor\le X$ and $\eta\le 1$ we have
\[\eta\lfloor X\rfloor(\lfloor X\rfloor+1)/2-\lfloor X\rfloor>\eta(X-1)X/2-X=\eta X(X+1)/2-\eta X-X\ge \eta X(X+1)/2-2X.\]
Therefore, $S(X)>\eta X(X+1)/2-2X$, and hence $g(X)<2X$.
\end{proof}

Fix a real number $\eta\in(0,1]$. In order to complete the proof of the proposition we need to show that $N(T; 0, \eta)\sim cT^2$ for some constant $c$ depending only on $\eta$. For any real number $T\ge 1$ we have
\begin{align*}
N(T; 0, \eta)&=\#\{r\in\Q:0<r\le\eta\text{\;and\;} H(r)\le T\}\\
&=\#\{(x,y)\in\N^2: \gcd(x,y)=1,\ y\le T, \ x\le\eta y\}\\
&=\sum_{1\le y\le T}\#\{x\in\N: \gcd(x,y)=1,\ x\le\eta y\}.
\end{align*}

We recall the following facts concerning the M\"obius function $\mu$ (see \cite[Thms. 263, 287]{hardy-wright}).
\begin{equation}\label{mu_props}
\sum_{d \mid n} \mu(d) = \begin{cases} 1 & \text{if} \ n = 1 \\ 0 & \text{if} \ n > 1 \end{cases} \;\;\;\; \text{and} \;\;\;\; \frac{1}{\zeta(s)}=\sum_{n=1}^{\infty}\frac{\mu(n)}{n^s} \; , \; s>1.
\end{equation}
Here, $\zeta$ denotes the Riemann zeta function. Using the first property of $\mu$ we can now write
\begin{align*}
N(T; 0, \eta)&= \sum_{y\le T}\sum_{x\le \eta y}\sum_{d|\gcd(x,y)}\mu(d) \\
&= \sum_{1\le d\le T}\mu(d)\cdot\#\{(x,y)\in\N^2: d|x, \ d|y, \ y\le T, \ x\le\eta y\} \\
&= \sum_{1\le d\le T}\mu(d)\cdot\#\{(a,b)\in\N^2:b\le T/d, \ a\le\eta b\}\\
&= \sum_{1\le d\le T}\mu(d)\cdot S(T/d),
\end{align*}
where $S$ is the function defined in Lemma \ref{S_bounds_lemma}. With the notation used in the lemma, we have
\begin{align*}
N(T; 0, \eta)&= \sum_{1\le d\le T}\mu(d)\cdot\left(\eta (T/d)(1+T/d)/2 - g(T/d)\right)\\
&=  (\eta T^2/2)\cdot \sum_{1\le d\le T}\frac{\mu(d)}{d^2} +  (\eta T/2)\cdot \sum_{1\le d\le T}\frac{\mu(d)}{d} - \sum_{1\le d\le T}\mu(d)\cdot g(T/d).
\end{align*}
Let $P(T)=\sum_{1\le d\le T}\mu(d)/d^2$, so that by \eqref{mu_props}, $P(T)\to 1/\zeta(2)$ as $T\to\infty$. We have
\begin{equation}\label{N_final}
\frac{2\cdot\zeta(2)\cdot N(T; 0, \eta)}{\eta T^2}=\zeta(2)\cdot P(T)\cdot\left(1+\frac{\sum_{1\le d\le T}\frac{\mu(d)}{d}}{T\cdot P(T)}-\frac{2\cdot\sum_{1\le d\le T}\mu(d)g(T/d)}{\eta T^2\cdot P(T)}\right).
\end{equation}

The bounds $0\le g(T/d)<2T/d$ proved in Lemma \ref{S_bounds_lemma} imply that
\[\left|\sum_{1\le d\le T}\mu(d)g(T/d)\right|\le2T\cdot\sum_{1\le d\le T}\frac{1}{d}<2T(\log T+1).\]
Hence,
\[\lim_{T\to\infty}\frac{\sum_{1\le d\le T}\mu(d)g(T/d)}{T^2}=0.\]

Using the fact that $\sum_{d=1}^{\infty}\mu(d)/d=0$ we see that the right-hand side of \eqref{N_final} converges to 1. Therefore, 
\[N(T; 0, \eta)\sim\left(\frac{\eta}{2\cdot\zeta(2)}\right)\cdot T^2,\] and this completes the proof of the proposition.
\end{proof}

The final ingredient needed in the proof of Lemma \ref{poly_lemma} is the following special case of \cite[Thm. B.6.1]{hindry-silverman}.
\begin{thm}\label{curve_counting}
Let $X/\Q$ be a curve of genus $g>0$ with $X(\Q)\neq\emptyset$. Let $H_X$ be a multiplicative Weil height on $X(\qbar)$ and define a counting function on $X(\Q)$ by
\[N(X(\Q),T)=\#\{P\in X(\Q): H_X(P)\le T\}.\] Then there are constants $a$ and $b$ such that
\[N(X(\Q), T)\sim
\begin{cases}
 a(\log T)^b & \text{if} \ g=1 \ (a>0, b\ge 0) \\
 a & \text{if} \ g\ge 2.
 \end{cases}
\]
\end{thm}

Recall the statement of Lemma \ref{poly_lemma}: Let $p(x)\in\Q[x]$ have nonzero discriminant and degree $\ge 3$. For every rational number $r$, define a field $K_r$ by \[K_r:=\Q\left(\sqrt{p(r)}\right).\] Then, for every interval $I\subset\R$ of positive length, the set $\Sigma(I,p)=\{K_r:r\in\Q\cap I\}$ contains infinitely many quadratic fields.

\begin{proof}[Proof of Lemma \ref{poly_lemma}]
Let $C$ be the hyperelliptic curve defined by the equation $y^2=p(x)$. If $d$ is a squarefree integer, we denote by $C_d$ the quadratic twist of $C$ by $d$, i.e., the curve defined by $dy^2=p(x)$. Note that in order to prove the lemma it suffices to consider finite closed intervals $I$; hence, we fix an interval $I=[\alpha,\beta]$. 

Let $d_1,\ldots, d_n$ be squarefree integers, and define $d_0=1$. We will show that $\Sigma(I,p)$ must contain a field different from $\Q(\sqrt{d_i})$ for every $i$, which will prove the lemma. For $T\in\R_{\ge 1}$ define
\[\mathcal N(T)=\#\{r\in \Q\cap I: H(r)\le T \ \text{and} \ K_r=\Q(\sqrt{d_i}) \ \text{for some} \ i\}.\]

For any fixed $d\in\{d_0,\ldots, d_n\}$, we define a height function on the curve $X=C_d$ by $H_X(x,y)=H(x)$; this is the naive Weil height on $C_d$. It follows from Theorem \ref{curve_counting} that (in particular) 
\begin{equation}\label{twist_count}
\sum_{d=0}^n N(C_d(\Q),T)\ll T \hspace{1cm} (T\to\infty).
\end{equation} 

Note that if $r\in\Q\cap I$ is such that the field $K_r$ is equal to $\Q(\sqrt{d})$, then there is a rational number $s$ such that $(r,s)\in C_d(\Q)$. Thus, for every real number $T$,
\[\#\{r\in \Q\cap I: H(r)\le T\text{\;and\;}K_r=\Q(\sqrt d)\}\le N(C_d(\Q),T).\]
Therefore, by \eqref{twist_count} we have $\mathcal N(T)\ll T$. Now, by Proposition \ref{counting_asymptotic} we know that
 \[\#\{r\in\Q\cap I: H(r)\le T\}\sim cT^2\]
 for some constant $c$ depending on $\alpha$ and $\beta$. Hence, 
 \[\mathcal N(T) < \#\{r\in\Q\cap I: H(r)\le T\} \;\; \text{for $T\gg 0$}.\]
By choosing $T$ large enough, this shows that there is a rational number $r\in I$ for which the field $K_r$ is different from all the fields $\Q(\sqrt{d_i})$. Hence, as claimed, the set $\Sigma(I,p)$ contains a field different from $\Q(\sqrt{d_i})$ for every $i$.
\end{proof}


\section{Preperiodic graph structures}\label{graph_pictures}
The following is a list of 46 graphs representing preperiodic structures discovered via the methods described in \S\ref{data_gathering}, applied to the case of quadratic number fields as explained in \textsection\ref{quad_prep_comp}. The label of each graph is in the form $N(\ell_1, \ell_2, \ldots)$, where $N$ denotes the number of vertices in the graph and $\ell_1, \ell_2, \ldots$ are the lengths of the directed cycles in the graph in nonincreasing order. If more than one isomorphism class of graphs with this data was observed, we add a lowercase roman letter to distinguish them. For example, the labels 5(1,1)a and 5(1,1)b correspond to the two isomorphism classes of graphs observed that have five vertices and two fixed points. In all figures below we omit the connected component corresponding to the point at infinity.

We remark that all of the graphs from Poonen's paper \cite{poonen_prep} appear below, with the labels 0, 2(1), 3(1,1), 3(2), 4(1,1), 4(2), 5(1,1)a, 6(1,1), 6(2), 6(3), 8(2,1,1), and 8(3).

\newpage
\includegraphics[scale=0.8]{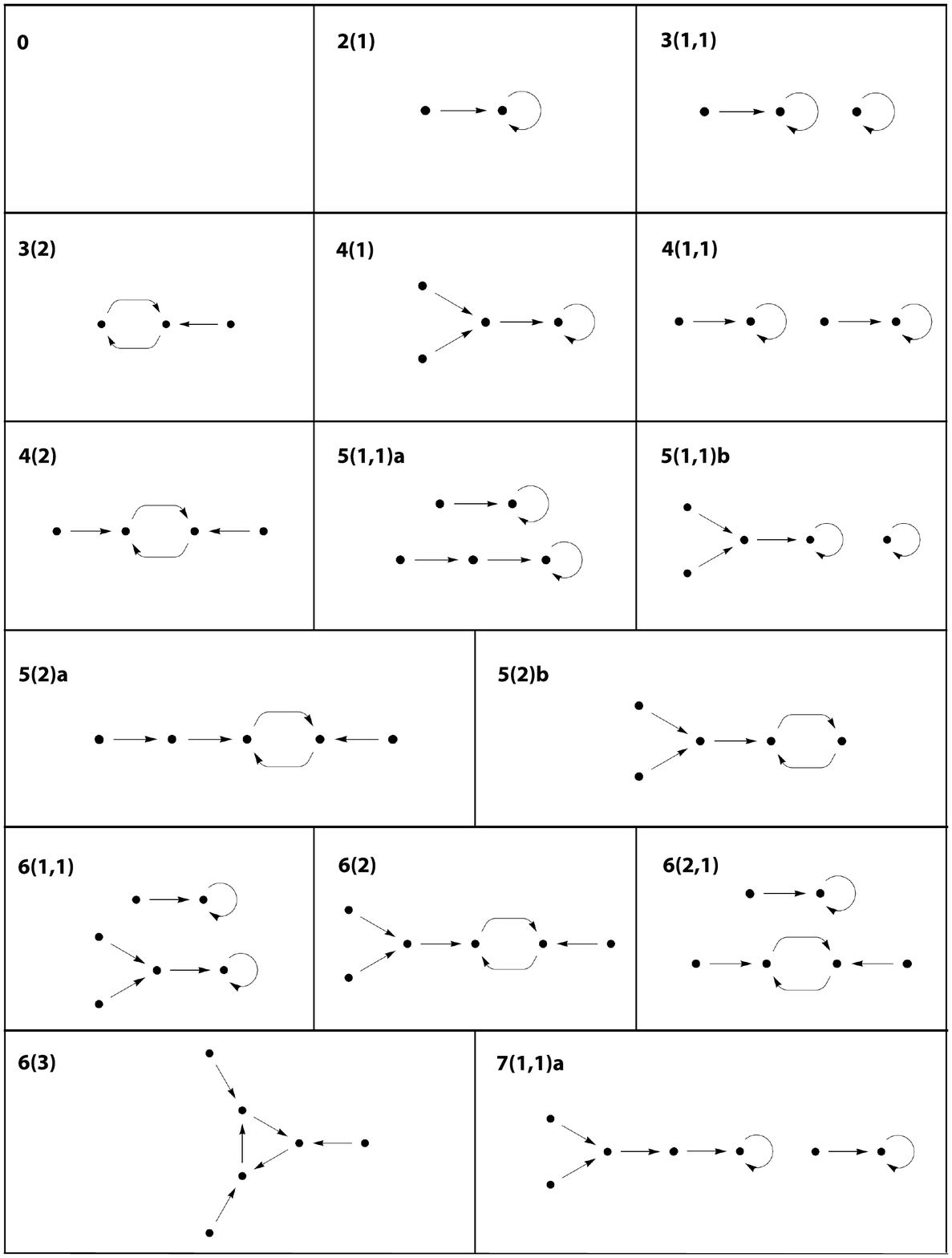}
\newpage
\includegraphics[scale=0.8]{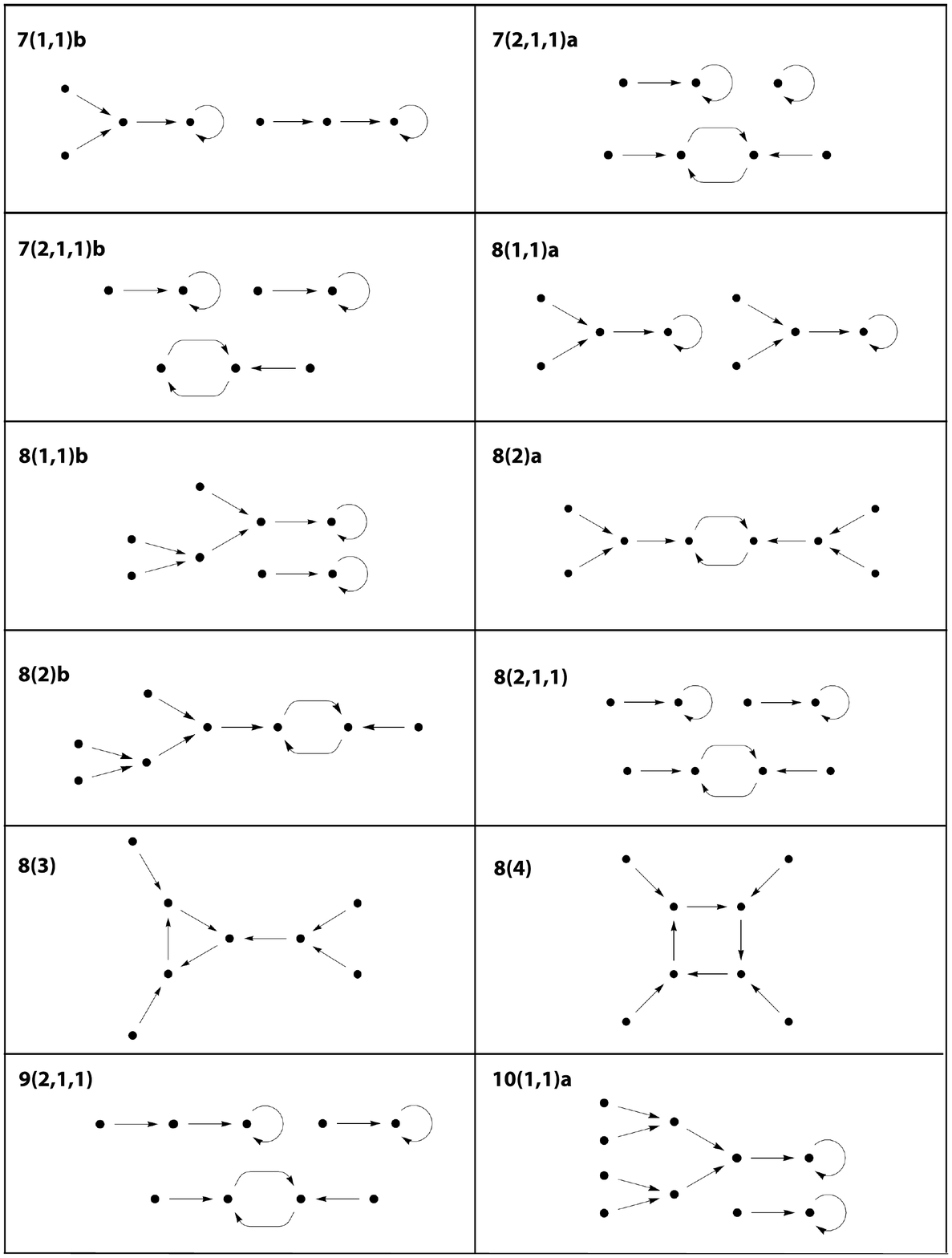}
\newpage
\includegraphics[scale=0.8]{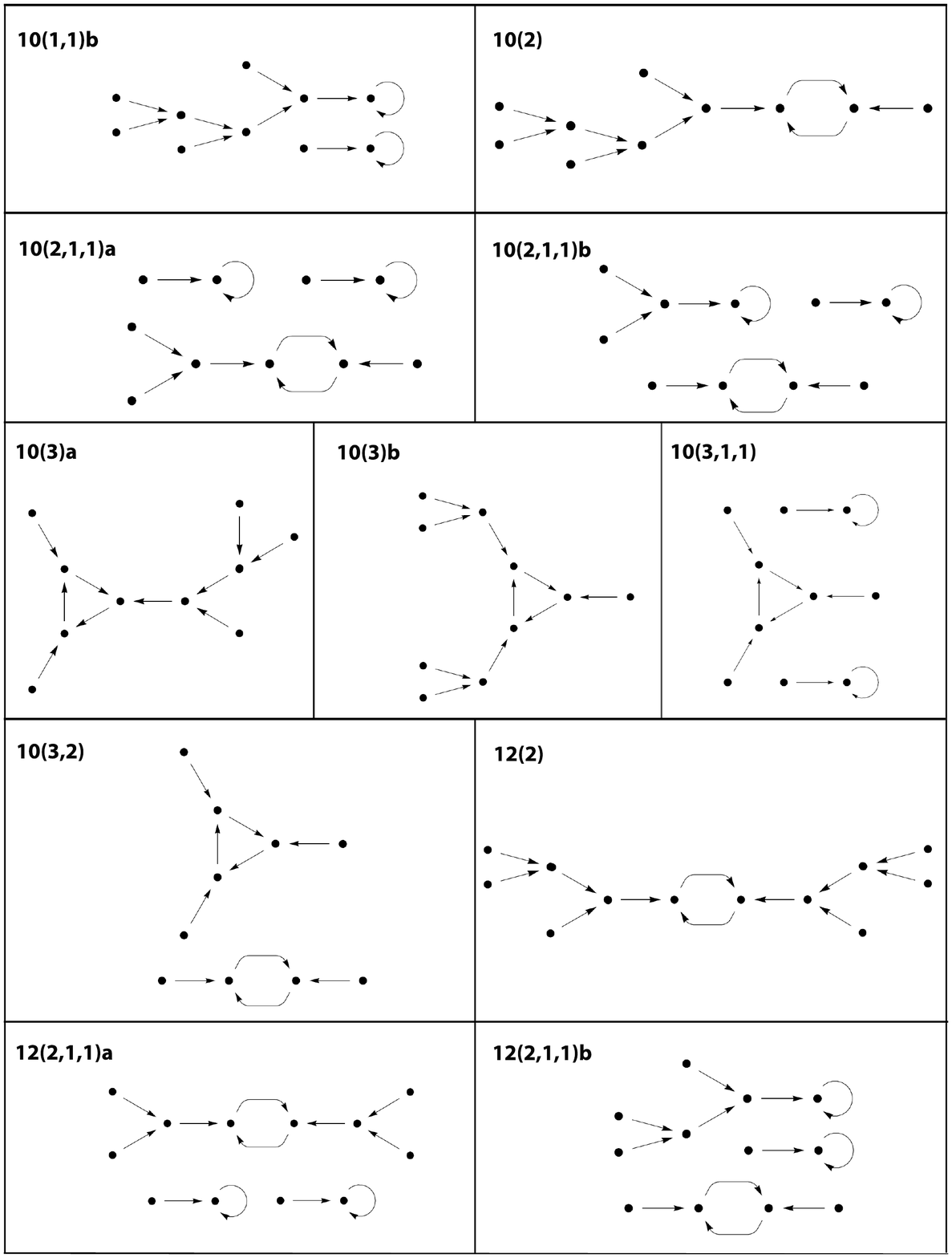}
\newpage
\includegraphics[scale=0.8]{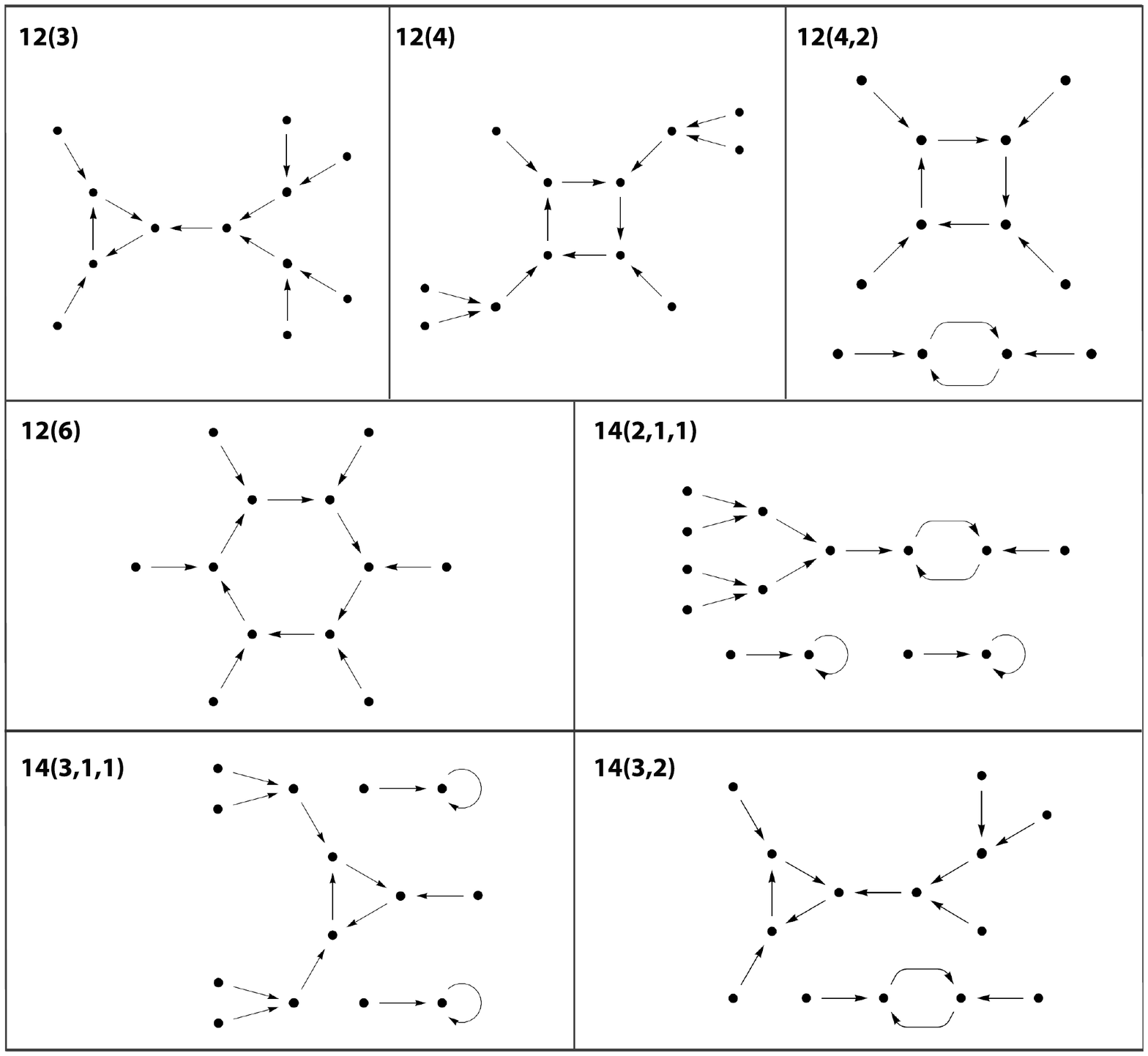}


\section{Representative data}\label{graph_data}
We give here a representative set of data for each graph in Appendix \ref{graph_pictures}. Each item in the list below includes the following information: 
\[
	K, p(t), c, \PrePer(f_c,K)'.
\]
Here $K = \Q(\sqrt{D})$ is a quadratic field over which this preperiodic structure was observed; $p(t)$ is a defining polynomial for $K$ with a root $g \in K$; $c$ is an element of $K$ such that the set $\PrePer(f_c,K)\backslash\{\infty\}$, when endowed with the structure of a directed graph, is isomorphic to the given graph; and $\PrePer(f_c,K)'$ is an abbreviated form of the full set of finite $K$-rational preperiodic points for $f_c$: since $x \in \PrePer(f_c, K)$ if and only if $-x \in \PrePer(f_c,K)$, we list only one of $x$ and $-x$ in the set $\PrePer(f_c,K)'$. We do not make explicit the correspondence between individual elements of this set and vertices of the graph. If a particular graph was observed over both real and imaginary quadratic fields, we give a representative set of data for each case.

{\bf 0.}\; 

\begin{itemize}[label={}, itemsep = 1.2mm, leftmargin = 1.5em]
\item $\Q(\sqrt{5}), \;\; t^2 - t - 1,\;\; 1, \;\; \emptyset$
\item $\Q(\sqrt{-3}), \;\; t^2 - t + 1,\;\; 2, \;\; \emptyset$
\end{itemize}

\bigskip
{\bf 2(1).}\;

\begin{itemize}[label={}, itemsep = 1.2mm, leftmargin = 1.5em]
\item $\Q(\sqrt{5}), \;\; t^2 - t - 1,\;\; \frac{1}{4}, \;\; \left\{\frac{1}{2}\right\}$
\item $\Q(\sqrt{-7}), \;\; t^2 - t + 2,\;\; \frac{1}{4}, \;\; \left\{\frac{1}{2}\right\}$
\end{itemize}

\bigskip
{\bf 3(1,1).}\;

\begin{itemize}[label={}, itemsep = 1.2mm, leftmargin = 1.5em]
\item $\Q(\sqrt{5}), \;\; t^2 - t - 1,\;\; 0, \;\; \left\{0, 1\right\}$
\item $\Q(\sqrt{-7}), \;\; t^2 - t + 2,\;\; 0, \;\; \left\{0, 1\right\}$
\end{itemize}

\bigskip
{\bf 3(2).}\;

\begin{itemize}[label={}, itemsep = 1.2mm, leftmargin = 1.5em]
\item $\Q(\sqrt{3}), \;\; t^2 - 3,\;\; -1, \;\; \left\{0, 1\right\}$
\item $\Q(\sqrt{-3}), \;\; t^2 - t + 1,\;\; -1, \;\; \left\{0, 1\right\}$
\end{itemize}

\bigskip
{\bf 4(1).}\; $\Q(\sqrt{-3}), \;\;t^2 - t + 1,\;\; \frac{1}{4}, \;\; \left\{ \frac{1}{2},  g - \frac{1}{2}\right\}$

\bigskip
{\bf 4(1,1).}\;

\begin{itemize}[label={}, itemsep = 1.2mm, leftmargin = 1.5em]
\item $\Q(\sqrt{5}), \;\; t^2 - t - 1,\;\; \frac{1}{5}, \;\; \left\{\frac{1}{5} g + \frac{2}{5}, \frac{1}{5} g - \frac{3}{5}\right\}$
\item $\Q(\sqrt{-3}), \;\;t^2 - t + 1, \;\;1, \;\; \left\{ g,  g - 1\right\}$
\end{itemize}

\newpage
\bigskip
{\bf 4(2).}\;

\begin{itemize}[label={}, itemsep = 1.2mm, leftmargin = 1.5em]
\item $\Q(\sqrt{5}), \;\; t^2 - t - 1,\;\; -\frac{4}{5}, \;\; \left\{\frac{1}{5} g + \frac{2}{5}, \frac{1}{5} g - \frac{3}{5}\right\}$
\item $\Q(\sqrt{-3}), \;\;t^2 - t + 1, \;\;-\frac{2}{3}, \;\; \left\{\frac{1}{3} g - \frac{2}{3}, \frac{1}{3} g + \frac{1}{3}\right\}$
\end{itemize}

\bigskip
{\bf 5(1,1)a.}\; 

\begin{itemize}[label={}, itemsep = 1.2mm, leftmargin = 1.5em]
\item $\Q(\sqrt{13}), \;\; t^2 - t - 3,\;\; -2, \;\; \left\{0, 2, 1\right\}$
\item $\Q(\sqrt{-3}), \;\; t^2 - t + 1,\;\; -2, \;\; \left\{0, 2, 1\right\}$
\end{itemize}

\bigskip
{\bf 5(1,1)b.}\; $\Q(\sqrt{-1}), \;\;t^2 + 1, \;\;0, \;\; \left\{0,  1,  g\right\}$

\bigskip
{\bf 5(2)a.}\; $\Q(\sqrt{-1}), \;\;t^2 + 1, \;\;g, \;\;\left\{0,  g, g - 1\right\}$

\bigskip
{\bf 5(2)b.}\; $\Q(\sqrt{2}), \;\;t^2 - 2, \;\;-1, \;\; \left\{0,  1,  g\right\}$

\bigskip
{\bf 6(1,1).}\; 

\begin{itemize}[label={}, itemsep = 1.2mm, leftmargin = 1.5em]
\item $\Q(\sqrt{5}), \;\; t^2 - t - 1,\;\; -\frac{3}{4}, \;\; \left\{\frac{1}{2},  g - \frac{1}{2}, \frac{3}{2}\right\}$
\item $\Q(\sqrt{-3}), \;\; t^2 - t + 1,\;\; -\frac{3}{4}, \;\; \left\{\frac{1}{2}, \frac{3}{2}, g - \frac{1}{2}\right\}$
\end{itemize}

\bigskip
{\bf 6(2).}\; 

\begin{itemize}[label={}, itemsep = 1.2mm, leftmargin = 1.5em]
\item $\Q(\sqrt{5}), \;\;t^2 - t - 1, \;\;-3, \;\; \left\{ 1,  2, 2 g - 1\right\}$
\item $\Q(\sqrt{-3}), \;\; t^2 - t + 1,\;\; -\frac{13}{9}, \;\; \left\{\frac{1}{3}, \frac{4}{3}, \frac{5}{3}\right\}$
\end{itemize}

\bigskip
{\bf 6(2,1).}\; $\Q(\sqrt{-1}), \;\;t^2 + 1, \;\;\frac{1}{4}, \;\;\left\{ \frac{1}{2},  g - \frac{1}{2}, g + \frac{1}{2}\right\}$

\bigskip
{\bf 6(3).}\; 

\begin{itemize}[label={}, itemsep = 1.2mm, leftmargin = 1.5em]
\item $\Q(\sqrt{33}), \;\; t^2 - t - 8,\;\; -\frac{301}{144}, \;\; \left\{\frac{5}{12}, \frac{19}{12}, \frac{23}{12}\right\}$
\item $\Q(\sqrt{-67}), \;\; t^2 - t + 17,\;\; -\frac{301}{144}, \;\; \left\{\frac{5}{12}, \frac{19}{12}, \frac{23}{12}\right\}$
\end{itemize}

\bigskip
{\bf 7(1,1)a.}\; $\Q(\sqrt{2}), \;\;t^2 - 2, \;\;-2, \;\; \left\{0,  1, 2,  g\right\}$

\bigskip
{\bf 7(1,1)b.}\; $\Q(\sqrt{3}), \;\;t^2 - 3, \;\;-2, \;\; \left\{0,  1, 2, g\right\}$

\bigskip
{\bf 7(2,1,1)a.}\; $\Q(\sqrt{-3}), \;\;t^2 - t + 1,\;\;0, \;\; \left\{0,  1,  g, g - 1\right\}$

\bigskip
{\bf 7(2,1,1)b.}\; $\Q(\sqrt{5}), \;\;t^2 - t - 1, \;\;-1, \;\; \left\{0,  1,  g, g - 1\right\}$

\newpage
\bigskip
{\bf 8(1,1)a.}\; 

\begin{itemize}[label={}, itemsep = 1.2mm, leftmargin = 1.5em]
\item $\Q(\sqrt{13}), \;\; t^2 - t - 3,\;\; -\frac{289}{144}, \;\; \left\{\frac{5}{6} g + \frac{1}{12}, \frac{1}{2} g - \frac{13}{12}, \frac{1}{2} g + \frac{7}{12}, \frac{5}{6} g - \frac{11}{12}\right\}$
\item $\Q(\sqrt{-15}), \;\; t^2 - t + 4,\;\; -\frac{5}{16}, \;\; \left\{\frac{1}{4}, \frac{3}{4}, \frac{5}{4}, \frac{1}{2} g - \frac{1}{4}\right\}$
\end{itemize}

\bigskip
{\bf 8(1,1)b.}\; 

\begin{itemize}[label={}, itemsep = 1.2mm, leftmargin = 1.5em]
\item $\Q(\sqrt{13}), \;\; t^2 - t - 3,\;\; -\frac{40}{9}, \;\; \left\{\frac{4}{3}, \frac{8}{3}, \frac{5}{3}, \frac{4}{3} g - \frac{2}{3}\right\}$
\item $\Q(\sqrt{-2}), \;\; t^2 + 2,\;\; -\frac{10}{9}, \;\; \left\{\frac{2}{3}, \frac{1}{3} g, \frac{4}{3}, \frac{5}{3}\right\}$
\end{itemize}

\bigskip
{\bf 8(2)a.}\; 

\begin{itemize}[label={}, itemsep = 1.2mm, leftmargin = 1.5em]
\item $\Q(\sqrt{10}), \;\; t^2 - 10,\;\; -\frac{13}{9}, \;\; \left\{\frac{1}{3}, \frac{1}{3} g, \frac{4}{3}, \frac{5}{3}\right\}$
\item $\Q(\sqrt{-3}), \;\;t^2 - t + 1,\;\; -\frac{5}{12}, \;\; \left\{ \frac{2}{3} g - \frac{5}{6},  \frac{2}{3} g + \frac{1}{6}, \frac{1}{3} g  + \frac{5}{6},   \frac{1}{3} g - \frac{7}{6} \right\}$
\end{itemize}

\bigskip
{\bf 8(2)b.}\; 

\begin{itemize}[label={}, itemsep = 1.2mm, leftmargin = 1.5em]
\item $\Q(\sqrt{13}), \;\; t^2 - t - 3,\;\; -\frac{37}{9}, \;\; \left\{\frac{4}{3}, \frac{5}{3}, \frac{7}{3}, \frac{4}{3} g - \frac{2}{3}\right\}$
\item $\Q(\sqrt{-7}), \;\; t^2 -t + 2,\;\; -\frac{13}{16}, \;\;  \left\{\frac{1}{4}, \frac{3}{4}, \frac{1}{2} g - \frac{1}{4}, \frac{5}{4}\right\}$
\end{itemize}

\bigskip
{\bf 8(2,1,1).}\; 

\begin{itemize}[label={}, itemsep = 1.2mm, leftmargin = 1.5em]
\item $\Q(\sqrt{5}), \;\; t^2 - t - 1,\;\; -12, \;\; \left\{3, 3 g - 1, 3 g - 2, 4\right\}$
\item $\Q(\sqrt{-3}), \;\; t^2 - t + 1,\;\; \frac{7}{12}, \;\; \left\{\frac{2}{3} g + \frac{1}{6}, \frac{2}{3} g - \frac{5}{6}, \frac{4}{3} g - \frac{7}{6}, \frac{4}{3} g - \frac{1}{6}\right\}$
\end{itemize}

\bigskip
{\bf 8(3).}\; 

\begin{itemize}[label={}, itemsep = 1.2mm, leftmargin = 1.5em]
\item $\Q(\sqrt{5}), \;\; t^2 - t - 1,\;\; -\frac{29}{16}, \;\; \left\{\frac{1}{4}, \frac{5}{4}, \frac{3}{4}, \frac{7}{4}\right\}$
\item $\Q(\sqrt{-3}), \;\; t^2 -t + 1,\;\; -\frac{29}{16}, \;\; \left\{\frac{1}{4}, \frac{5}{4}, \frac{3}{4}, \frac{7}{4}\right\}$
\end{itemize}

\bigskip
{\bf 8(4).}\; 

\begin{itemize}[label={}, itemsep = 1.2mm, leftmargin = 1.5em]
\item $\Q(\sqrt{10}), \;\;t^2 - 10,\;\; -\frac{155}{72}, \;\; \left\{ \frac{1}{4} g - \frac{1}{6},   \frac{1}{4} g + \frac{1}{6},  \frac{1}{12} g - \frac{3}{2},   \frac{1}{12} g + \frac{3}{2}\right\}$
\item $\Q(\sqrt{-455}), \;\;t^2 - t + 114,\;\; \frac{199}{720},\;\;\left\{ \frac{1}{10}g+\frac{17}{60}, \frac{1}{15}g - \frac{47}{60},\frac{1}{10}g-\frac{23}{60},\frac{1}{15}g+\frac{43}{60}\right\}$
\end{itemize}

\bigskip
{\bf 9(2,1,1).}\; $\Q(\sqrt{5}), \;\;t^2 - t - 1, \;\;-2, \;\; \left\{0,  1, 2,  g,  g - 1\right\}$

\bigskip
{\bf 10(1,1)a.}\; $\Q(\sqrt{-7}), \;\;t^2 - t + 2,\;\; \frac{3}{16}, \;\; \left\{\frac{1}{4}, \frac{1}{2} g + \frac{1}{4}, \frac{1}{2} g - \frac{1}{4}, \frac{1}{2} g - \frac{3}{4}, \frac{3}{4}\right\}$

\bigskip
{\bf 10(1,1)b.}\; $\Q(\sqrt{17}), \;\;t^2 - t - 4, \;\;-\frac{1}{2} g - \frac{13}{16}, \;\; \left\{\frac{1}{4}, \frac{1}{2} g + \frac{3}{4}, \frac{3}{4}, \frac{1}{2} g - \frac{1}{4}, \frac{1}{2} g + \frac{1}{4}\right\}$

\newpage
\bigskip
{\bf 10(2).}\; 

\begin{itemize}[label={}, itemsep = 1.2mm, leftmargin = 1.5em]
\item $\Q(\sqrt{73}), \;\;t^2 - t - 18,\;\; \frac{1}{9} g - \frac{205}{144}, \;\; \left\{\frac{1}{6} g + \frac{1}{12}, \frac{1}{6} g - \frac{11}{12}, \frac{1}{6} g + \frac{7}{12}, \frac{1}{3} g - \frac{7}{12}, \frac{1}{3} g - \frac{1}{12} \right\}$
\item $\Q(\sqrt{-7}), \;\;t^2 - t + 2,\;\; -\frac{1}{2} g - \frac{5}{16}, \;\; \left\{\frac{1}{4}, \frac{1}{2} g - \frac{1}{4}, \frac{1}{2} g + \frac{1}{4}, \frac{3}{4}, \frac{1}{2} g + \frac{3}{4}\right\}$
\end{itemize}

\bigskip
{\bf 10(2,1,1)a.}\;

\begin{itemize}[label={}, itemsep = 1.2mm, leftmargin = 1.5em]
\item $\Q(\sqrt{17}), \;\; t^2 - t - 4,\;\; -\frac{273}{64}, \;\; \left\{\frac{11}{8}, \frac{13}{8}, \frac{19}{8}, \frac{5}{4} g - \frac{5}{8}, \frac{21}{8}\right\}$
\item $\Q(\sqrt{-1}), \;\;t^2 + 1,\;\; \frac{3}{8} g - \frac{1}{4}, \;\; \left\{\frac{3}{4} g + \frac{1}{4},  \frac{3}{4} g - \frac{3}{4},   \frac{1}{4} g - \frac{1}{4},  \frac{1}{4} g + \frac{3}{4},   \frac{1}{4} g - \frac{5}{4}\right\}$
\end{itemize}

\bigskip
{\bf 10(2,1,1)b.}\;

\begin{itemize}[label={}, itemsep = 1.2mm, leftmargin = 1.5em]
\item $\Q(\sqrt{13}), \;\;t^2 - t - 3, \;\;-\frac{10}{9}, \;\; \left\{\frac{2}{3}, \frac{4}{3}, \frac{5}{3}, \frac{1}{3} g - \frac{2}{3}, \frac{1}{3} g + \frac{1}{3} \right\}$
\item $\Q(\sqrt{-7}), \;\; t^2 -t + 2,\;\; -\frac{21}{16}, \;\; \left\{\frac{1}{4}, \frac{7}{4}, \frac{1}{2} g - \frac{1}{4}, \frac{3}{4}, \frac{5}{4}\right\}$
\end{itemize}

\bigskip
{\bf 10(3)a.}\; $\Q(\sqrt{41}), \;\;t^2 - t - 10,\;\; -\frac{29}{16}, \;\; \left\{\frac{1}{4}, \frac{5}{4}, \frac{3}{4}, \frac{1}{2} g - \frac{1}{4}, \frac{7}{4}\right\}$

\medskip
{\bf 10(3)b.}\; $\Q(\sqrt{57}), \;\;t^2 - t - 14, \;\;-\frac{29}{16}, \;\; \left\{  \frac{1}{4}, \frac{3}{4}, \frac{5}{4},  \frac{7}{4},  \frac{1}{2} g - \frac{1}{4}\right\}$

\bigskip
{\bf 10(3,1,1)}\; $\Q(\sqrt{337}), \;\;t^2 - t - 84, \;\;-\frac{301}{144}, \;\; \left\{\frac{5}{12}, \frac{19}{12}, \frac{23}{12}, \frac{1}{6}g + \frac{5}{12}, \frac{1}{6}g - \frac{7}{12} \right\}$

\bigskip
{\bf 10(3,2).}\; $\Q(\sqrt{193}), \;\;t^2 - t - 48,\;\; -\frac{301}{144}, \;\; \left\{\frac{5}{12}, \frac{19}{12}, \frac{23}{12}, \frac{1}{6} g + \frac{5}{12}, \frac{1}{6} g - \frac{7}{12} \right\}$

\medskip
{\bf 12(2).}\; $\Q(\sqrt{2}), \;\;t^2 - 2, \;\;-\frac{15}{8}, \;\; \left\{\frac{3}{4} g + \frac{1}{2}, \frac{3}{4} g - \frac{1}{2}, \frac{1}{4} g + \frac{1}{2}, \frac{1}{4} g - \frac{3}{2}, \frac{1}{4} g - \frac{1}{2}, \frac{1}{4} g + \frac{3}{2}\right\}$

\bigskip
{\bf 12(2,1,1)a.}\; $\Q(\sqrt{17}), \;\;t^2 - t - 4, \;\;-\frac{13}{16}, \;\; \left\{\frac{1}{4}, \frac{3}{4}, \frac{5}{4},  \frac{1}{2} g + \frac{1}{4}, \frac{1}{2} g - \frac{3}{4}, \frac{1}{2} g - \frac{1}{4}\right\}$

\bigskip
{\bf 12(2,1,1)b.}\; 

\begin{itemize}[label={}, itemsep = 1.2mm, leftmargin = 1.5em]
\item $\Q(\sqrt{33}), \;\; t^2 - t - 8,\;\; -\frac{45}{16}, \;\; \left\{\frac{3}{4}, \frac{9}{4},  \frac{5}{4}, \frac{1}{2} g - \frac{3}{4}, \frac{1}{2} g + \frac{1}{4}, \frac{1}{2} g - \frac{1}{4}\right\}$
\item $\Q(\sqrt{-7}), \;\;t^2 - t + 2, \;\;-\frac{5}{16}, \;\; \left\{\frac{1}{4}, \frac{3}{4}, \frac{5}{4},  \frac{1}{2} g + \frac{1}{4}, \frac{1}{2} g - \frac{3}{4}, \frac{1}{2} g - \frac{1}{4}\right\}$
\end{itemize}

\medskip
{\bf 12(3).}\; $\Q(\sqrt{73}), \;\;t^2 - t - 18, \;\;-\frac{301}{144}, \;\; \left\{\frac{1}{6} g - \frac{1}{12},  \frac{5}{12}, \frac{19}{12}, \frac{1}{3} g + \frac{1}{12}, \frac{1}{3} g - \frac{5}{12}, \frac{23}{12}\right\}$

\bigskip
{\bf 12(4).}\; $\Q(\sqrt{105}), \;\;t^2 - t - 26,\;\; -\frac{95}{48}$, 
$\;\; \left\{\frac{1}{6} g - \frac{13}{12}, \frac{1}{6} g + \frac{11}{12}, \frac{1}{3} g - \frac{5}{12}, \frac{1}{6} g + \frac{5}{12}, \frac{1}{6} g - \frac{7}{12}, \frac{1}{3} g + \frac{1}{12}\right\}$

\bigskip
{\bf 12(4,2).}\; $\Q(\sqrt{-15}), \;\;t^2 - t + 4,\;\; -\frac{31}{48},$
$ \;\; \left\{ \frac{1}{3} g + \frac{1}{12}, \frac{1}{6} g - \frac{13}{12}, \frac{1}{3} g - \frac{5}{12}, \frac{1}{6} g + \frac{5}{12}, \frac{1}{6} g - \frac{7}{12},  \frac{1}{6} g + \frac{11}{12}\right\}$

\bigskip
{\bf 12(6).}\; $\Q(\sqrt{33}), \;\;t^2 - t - 8, \;\;-\frac{71}{48}$,
$ \;\; \left\{\frac{1}{6} g - \frac{13}{12}, \frac{1}{6} g - \frac{7}{12}, \frac{1}{3} g - \frac{5}{12}, \frac{1}{6} g + \frac{5}{12}, \frac{1}{3} g + \frac{1}{12}, \frac{1}{6} g + \frac{11}{12}\right\}$

\bigskip
{\bf 14(2,1,1).}\; $\Q(\sqrt{17}), \;\;t^2 - t - 4,\;\; -\frac{21}{16}, \;\; \left\{\frac{1}{4},  \frac{3}{4}, \frac{5}{4}, \frac{7}{4}, \frac{1}{2} g - \frac{1}{4}, \frac{1}{2} g - \frac{3}{4}, \frac{1}{2} g + \frac{1}{4} \right\}$

\bigskip
{\bf 14(3,1,1).}\; $\Q(\sqrt{33}), \;\;t^2 - t - 8, \;\;-\frac{29}{16}, \;\; \left\{\frac{1}{4}, \frac{5}{4}, \frac{3}{4}, \frac{1}{2} g - \frac{3}{4}, \frac{1}{2} g + \frac{1}{4}, \frac{1}{2} g - \frac{1}{4}, \frac{7}{4}\right\}$

\bigskip
{\bf 14(3,2).}\; $\Q(\sqrt{17}), \;\;t^2 - t - 4,\;\; -\frac{29}{16}, \;\; \left\{\frac{1}{4}, \frac{5}{4}, \frac{3}{4}, \frac{1}{2} g - \frac{1}{4}, \frac{1}{2} g - \frac{3}{4}, \frac{1}{2} g + \frac{1}{4}, \frac{7}{4}\right\}$

\bigskip
\printbibliography[title=References]

\bigskip
\end{document}